\documentclass[11pt]{amsart}

\usepackage{epigamath}

\usepackage[english]{babel}

\numberwithin{equation}{section}

\usepackage{enumitem}
\usepackage[all]{xy}

\newtheorem*{theorem*}{Theorem}
\newtheorem{thm}{Theorem}[section]
\newtheorem{cor}[thm]{Corollary}
\newtheorem{lem}[thm]{Lemma}
\newtheorem{prop}[thm]{Proposition}

\theoremstyle{definition}
\newtheorem{definition}[thm]{Definition}

\theoremstyle{remark}
\newtheorem{rem}[thm]{Remark}

\newcommand{\ddbar}{\sqrt{-1}\,\bar{\partial}\partial}

\newcommand{\ai}{\sqrt{-1}} 
\newcommand{\isom}{\stackrel{\sim}{\longrightarrow}} 

\newcommand{\Vol}{\mathrm{Vol}}

\newcommand{\rk}{\mathrm{rk}}
\newcommand{\tr}{\mathrm{tr}}

\newcommand{\Id}{\mathrm{Id}}

\newcommand{\surj}{\to\kern-1.8ex\to}

\DeclareMathOperator{\SL}{SL}

\EpigaVolumeYear{5}{2021} \EpigaArticleNr{21} \ReceivedOn{June 18,
2020}
\InFinalFormOn{October 11, 2021}
\AcceptedOn{November 7, 2021}

\title{Quot-scheme limit of Fubini--Study metrics and Donaldson's functional for vector bundles}
\titlemark{Quot-scheme limit and Donaldson's functional}

\author{Yoshinori Hashimoto}
\email{hashimoto@math.titech.ac.jp}
\address{Department of Mathematics, Tokyo Institute of Technology, 2-12-1 Ookayama, Meguro-ku, Tokyo, 152-8551, Japan}

\author{Julien Keller}
\email{keller.julien.3@uqam.ca}
\address{D\'epartement de Math\'ematiques, Universit\'e du Qu\'ebec \`a Montr\'eal, C.P. 8888, Succ. Centre-Ville, Montr\'eal H3C 3P8 Canada}

\authormark{Y.~Hashimoto and J.~Keller}

\AbstractInEnglish{\footnotesize{For a holomorphic vector bundle $\mathcal{E}$ over a polarised K\"ahler manifold, we establish a direct link between the slope stability of $\mathcal{E}$ and the asymptotic behaviour of Donaldson's functional, by defining the Quot-scheme limit of Fubini--Study metrics. In particular, we provide an explicit estimate which proves that Donaldson's functional is coercive on the set of Fubini--Study metrics if $\mathcal{E}$ is slope stable, and give a new proof of Hermitian--Einstein metrics implying slope stability.}}



\TitleInFrench{Limite au sens du sch\'ema Quot des m\'etriques de Fubini--Study et fonctionnelle de Donaldson pour les fibr\'es vectoriels}

\AbstractInFrench{\footnotesize{Pour un fibr\'e vectoriel holomorphe $\mathcal{E}$ sur une vari\'et\'e k\"ahl\'erienne polaris\'ee, nous \'etablissons un lien direct entre la stabilit\'e de $\mathcal{E}$ et le comportement asymptotique de la fonctionnelle de Donaldson, en d\'efinissant la limite au sens du sch\'ema Quot des m\'etriques de Fubini--Study. En particulier, nous fournissons une estimation explicite qui montre que la fonctionnelle de Donaldson est coercive sur l'ensemble des m\'etriques de Fubini--Study si $\mathcal{E}$ est stable et nous donnons une nouvelle d\'emonstration du fait que l'existence de m\'etriques d'Hermite--Einstein implique la stabilit\'e.}}

\acknowledgement{The work of both authors has been carried out in the framework of the Labex Archim\`ede (ANR-11-LABX-0033) and of the A*MIDEX project (ANR-11-IDEX-0001-02), funded by the ``Investissements d'Avenir" French Government programme managed by the French National Research Agency (ANR). The first named author was supported by the Post-doc position in Complex and Symplectic Geometry ``Bando de Bartolomeis 2017'', at Dipartimento di Matematica e Informatica U.~Dini, Universit\`a di Firenze, dedicated to the memory of Professor Paolo de Bartolomeis and funded by his family, and also by Universit\`a di Firenze, Progetto di ricerca scientifica d'ateneo (ex quota 60\%) Anno 2017: ``Geometria differenziale, algebrica, complessa e aritmetica''. The second named author was also supported by the ANR project EMARKS (ANR-14-CE25-0010).}

\begin{document}


\removeabove{0.7cm}
\removebetween{0.6cm}
\removebelow{0.8cm}

\maketitle

\begin{prelims}

\DisplayAbstractInEnglish

\bigskip

\DisplayKeyWords

\medskip

\DisplayMSCclass

\bigskip

\languagesection{Fran\c{c}ais}

\bigskip

\DisplayTitleInFrench

\medskip

\DisplayAbstractInFrench

\end{prelims}


\newpage

\setcounter{tocdepth}{1}

\tableofcontents


\section{Introduction}
\subsubsection*{Overview} Let $\mathcal{E}$ be a holomorphic vector bundle over a polarised smooth complex projective variety $(X , \mathcal{O}_X (1))$. A fundamental result connecting differential and algebraic geometry states that $\mathcal{E}$ admits Hermitian--Einstein metrics if and only if it is slope polystable (this is also called the Donaldson--Uhlenbeck--Yau theorem or the Kobayashi--Hitchin correspondence in the literature).

One direction of this correspondence was proved by Kobayashi \cite{Kobookjp} and L\"ubke \cite{Luebke}, and the other, arguably harder, direction was proved by Donaldson \cite{Do-1983,Don1985,Don1987}; it was also extended to compact K\"ahler manifolds by Uhlenbeck and Yau \cite{U-Y, U-Y2}. Fundamental as they are, their argument connecting the differential geometry of Hermitian--Einstein metrics and the algebraic geometry of slope  stability is rather technical, as we shall briefly review later.

\medskip
The aim of this paper is to provide the first step towards establishing a more direct link between these two concepts, from the \textit{variational} point of view of a functional defined by Donaldson \cite{Don1985}, whose critical point, if it exists, is the Hermitian--Einstein metric. More precisely, we study the asymptotic behaviour of Donaldson's functional $\mathcal{M}^{\mathrm{Don}}$ on the set of Fubini--Study metrics $\mathcal{H}_k$, which is dense in the space of all Hermitian metrics on $\mathcal{E}$, and prove that it is controlled by the algebro-geometric slope of certain filtrations of $\mathcal{E}$ by subsheaves (see Definition \ref{defmdonna}). While the precise statement will be given in Theorems \ref{thmlnsdf} and \ref{uppbound} in Section \ref{mainsect}, one of our main conclusions can be summarised in the following form.
\begin{theorem*} \label{sumthm}
Suppose that we take $k \in \mathbb{N}$ so that $\mathcal{E}(k) := \mathcal{E} \otimes \mathcal{O}_X(k)$ is globally generated. Let $\{ h_{\sigma_{t}} \}_{t \ge 0}$ be a family of Fubini--Study metrics on $\mathcal{E}$, emanating from a reference metric $h_{\mathrm{ref}}$, defined by a 1-parameter subgroup $\sigma : \mathbb{C}^* \to \SL (H^0 (X, \mathcal{E}(k))^{\vee})$ in the category of algebraic groups. Then there exists an integrally graded filtration $\{ \mathcal{E}_{\le q} \}_{q \in \mathbb{Z}}$ of $\mathcal{E}$ by saturated subsheaves such that
\begin{equation} \label{sumthmeq}
		\lim_{t \to + \infty} \frac{\mathcal{M}^{\mathrm{Don}} (h_{\sigma_t}, h_{\mathrm{ref}})}{t} =  2 \sum_{q \in \mathbb{Z}} \mathrm{rk} (\mathcal{E}_{\le q}) \left(  \mu(\mathcal{E}) - \mu (\mathcal{E}_{\le q}) \right), \tag{A}
\end{equation}
where $\mu (\mathcal{E})$ (resp.~$\mu (\mathcal{E}_{\le q})$) is the slope of $\mathcal{E}$ (resp.~$\mathcal{E}_{\le q}$).
\end{theorem*}

The actual results that we prove in Theorems \ref{thmlnsdf} and \ref{uppbound} are stronger in various ways; in particular, we can relax the condition on the 1-parameter subgroup $\sigma$ in the above, so that the collection of families of Fubini--Study metrics $\{ h_{\sigma_{t}} \}_{t \ge 0}$ as above exhausts a dense subset of the space of Hermitian metrics on $\mathcal{E}$; see (\ref{defbergsp2}).

In particular, we prove that the slope stability of $\mathcal{E}$ implies that Donaldson's functional is coercive on the set of Fubini-Study metrics (Corollary \ref{prop1}). Moreover, we do so by making practically no use of the PDE theory.

An application of the above results provided in this paper is a new proof of the theorem of Kobayashi \cite{Kobookjp} and L\"ubke \cite{Luebke}, which states that the existence of Hermitian--Einstein metrics implies the slope stability of $\mathcal{E}$ (Section \ref{cantostab}). The converse, probably more interesting, direction of the correspondence is treated in the sequel to this paper \cite{HK2}. More specifically, in \cite{HK2} we provide an alternative proof of the Donaldson--Uhlenbeck--Yau theorem, which asserts the existence of Hermitian--Einstein metrics on slope stable vector bundles, by assuming that a ``uniform'' version of the inequality in Theorem \ref{thmlnsdf} holds; the novelty of this alternative proof is that it relies much less on analysis and PDE theory by using instead the methods established in this paper (see below and \cite{HK2} for more details).

\medskip
The key new definitions in achieving the results of the current paper are the {\it renormalised Quot-scheme limit of Fubini--Study metrics} (Section \ref{Qlimit}) and the {\it non-Archimedean Donaldson functional} (Section \ref{RMstab}), which relate a certain limit of Fubini--Study metrics to the flat limit in the Quot-scheme, thereby providing an algebro-geometric description of it. We believe that they are interesting in their own right and have a range of potential applications, \emph{e.g.}~the metric study of Gieseker stability notion for Higgs bundles initiated in \cite{GF-K-R}.

\medskip
\subsubsection*{Comparison to other works}  In order to clarify the novelty of our method, we compare our work to the well-established results concerning the correspondence between Hermitian--Einstein metrics and slope stability.

\medskip
The crux of the theorem of Kobayashi \cite{Kobookjp} and L\"ubke \cite{Luebke} is an application of the Gauss--Codazzi formula, together with a sheaf-theoretic argument which is rather technical (see also \cite[Section 5.8]{Kobook}). While our alternative proof (in Section \ref{cantostab}) is perhaps no simpler in terms of technicality, we believe that it is more \textit{geometric}, as the proof reduces to studying the asymptotic behaviour of a convex energy functional.

\medskip
The other direction of the correspondence, established by Donaldson \cite{Do-1983,Don1985,Don1987} and Uhlenbeck--Yau \cite{U-Y, U-Y2}, rests on a sophisticated application of the nonlinear PDE theory. Both of their works can be described as a ``continuous deformation'' from an initial metric to the Hermitian--Einstein metric. In both cases, it aims to obtain the Hermitian--Einstein metric at the limit of such continuous deformations, where the deformation process and the limiting behaviour are analysed by means of the nonlinear PDE theory. When the bundle $\mathcal{E}$ is slope unstable, both methods prove the existence of a destabilising subsheaf. Donaldson uses the Yang--Mills flow and argues by applying the Mehta--Ramanathan theorem and the induction on the dimension of the variety (see also the recent complementary results on the Yang--Mills flow \cite{DasWent04,Jacob15,Jacob16,Sibley,SibWent} and references therein). The continuity method of Uhlenbeck--Yau relies on a certain compactness result for treating the limit, but it is well known that the argument is highly technical (see also \cite{Popo1}).

Our work, on the other hand, relies much less on analysis. Indeed, in this paper and the sequel \cite{HK2}, the analysis that we use is mostly elementary, except for that the asymptotic expansion of the Bergman kernel (the so-called Tian-Yau-Zelditch expansion), or its well-known consequence that the set of Fubini--Study metrics is dense in the space of Hermitian metrics (\ref{defbergsp2}), will play a crucially important role in \cite{HK2}; it also implies that our results in this paper cover a dense subset of the space of Hermitian metrics on $\mathcal{E}$. It seems interesting to point out that we  establish a more direct relationship between slope stability and Donaldson's functional by \textit{not} relying much on analysis, especially the PDE techniques.

\medskip
While our work has an obvious relationship to many classical results on holomorphic vector bundles, it is important to note that it is also inspired by the recent advances on the Yau--Tian--Donaldson conjecture, relating the existence of constant scalar curvature K\"ahler (cscK) metrics to $K$-stability. More precisely, our work can be regarded as a vector-bundle analogue of the well-known results due to Paul \cite{Paul}, Paul--Tian \cite{PT1,PT2} and Phong--Ross--Sturm \cite{PRS} on the asymptotic behaviour of the Mabuchi energy on the set of Fubini--Study metrics, leading to the work of Boucksom--Hisamoto--Jonsson \cite{BHJ1,BHJ2}. Recall that Boucksom--Hisamoto--Jonsson \cite{BHJ1,BHJ2} identified the asymptotic slope of various energy functionals as a ``non-Archimedean'' functional that is defined in terms of algebro-geometric intersection formulae (see also \emph{e.g.}~\cite{Boucksom-icm}). Their method was applied to the problem of K\"ahler--Einstein metrics by Berman--Boucksom--Jonsson \cite{BBJ}. While their influence on our work should be clear, it is perhaps worth mentioning that we can draw parallels between many results and concepts in cscK metrics (varieties) and Hermitian--Einstein metrics (vector bundles), which is explained in more details in the sequel \cite{HK3}.

Likewise, our proof of Hermitian--Einstein metrics implying slope stability is similar to many established results in cscK metrics, such as Berman--Darvas--Lu \cite{BDLu}, Dervan--Ross \cite{DerRoss}, and Sj\"ostr\"om Dyrefelt \cite{Sj}; the existence of canonical metrics implying stability is proved by evaluating the asymptotic behaviour of appropriate energy functionals (Mabuchi energy for cscK metrics, Donaldson's functional for Hermitian--Einstein metrics). Indeed, if the vector bundle $\mathcal{E}$ is unstable and destabilised by a subsheaf $\mathcal{F} \subset \mathcal{E}$, we prove that we have a 1-parameter subgroup $\sigma_{\mathcal{F}} : \mathbb{C}^* \to \SL(H^0 (\mathcal{E}(k))^{\vee})$ (see Definition \ref{twostepfiltdef}) along which the Donaldson functional is unbounded below, by showing that the right hand side of the equation (\ref{sumthmeq}) for $\sigma_{\mathcal{F}}$ is a multiple of $\mu (\mathcal{E}) - \mu (\mathcal{F})$ and hence is negative (see Proposition \ref{prop2} and Section \ref{cantostab}).

Finally, we note that a specific filtration of the vector bundle $\mathcal{E}$, called the Harder--Narasimhan filtration, played an important role in the study of Hermitian--Einstein metrics since the seminal work of Atiyah--Bott \cite{A-B}. A novel point of our work could be that we consider \textit{arbitrary} filtrations associated to the Quot-schemes, which can be compared to \cite[Section 2]{DonCalabi}.

\medskip
\subsubsection*{Organisation of the paper} In Section \ref{background}, we introduce some background materials on Quot-schemes, Fubini--Study metrics, and the Donaldson functional. In Section \ref{Qlimit}, we define the renormalised Quot-scheme limit of Fubini--Study metrics, which provides us with a link between differential and algebraic geometry. This machinery is applied in Section \ref{unifestimsect}, in which we evaluate the asymptotic behaviour of the Donaldson functional; in particular, we prove that it can be controlled by the algebro-geometric data.  In Section \ref{RMstab}, we define the non-Archimedean Donaldson functional in terms of filtrations of $\mathcal{E}$ by subsheaves. In Section \ref{2stp}, we look at a particular class of filtrations, which we call the two-step filtrations, defined by a saturated subsheaf of $\mathcal{E}$. In Section \ref{mainsect}, we summarise the main results of the paper (Theorems \ref{thmlnsdf}, \ref{uppbound}, and Corollary \ref{prop1}). In Section \ref{cantostab}, we apply them to provide an alternative proof of the theorem of Kobayashi and L\"ubke.

\subsection*{Notation}
We shall consistently write $(X,L)$ for a polarised smooth projective variety over $\mathbb{C}$ of complex dimension $n$, where we further assume that $L$ is very ample. We work with a fixed K\"ahler metric $\omega \in c_1 (L)$ on $X$ defined by a Hermitian metric $h_L$ on $L$. We often write $\mathcal{O}_X (k)$ for $L^{\otimes k}$.

We write $\mathcal{O}_X$ for the sheaf of rings of holomorphic functions on $X$, $C^{\infty}_X$ for the one of $\mathbb{C}$-valued $C^{\infty}$-functions on $X$.

Throughout in this paper, coherent sheaves of $\mathcal{O}_X$-modules will be denoted by calligraphic letters (\emph{e.g.}~$\mathcal{E}$). Block letters (\emph{e.g.}~$E$) will be reserved for complex $C^{\infty}$-vector bundles. This leads to an unavoidable clash of notation for a holomorphic vector bundle, but we use the following convention: both $\mathcal{E}$ and $E$ will be used to denote a holomorphic vector bundle, but $E$ will be treated as a $C^{\infty}$-object whereas $\mathcal{E}$ is meant to be a sheaf. A $C^{\infty}$-vector bundles is clearly a locally free sheaf of $C^{\infty}_X$-modules, but when we say ``locally free'', it is always understood that it is locally free as a sheaf of $\mathcal{O}_X$-modules.

The symbol $\vee$ is used for the dual; for example, for a vector space $V$, we will write $V^{\vee} = \mathrm{Hom}_{\mathbb{C}} (V , \mathbb{C})$ and for a coherent sheaf $\mathcal{F}$ of $\mathcal{O}_X$-modules, we write $\mathcal{F}^{\vee} = \mathcal{H}om_{\mathcal{O}_X} (\mathcal{F} , \mathcal{O}_X)$. Also,  for a $C^{\infty}$-vector bundle $F$, we write $F^{\vee}:= \mathcal{H}om_{C^{\infty}_X} (F , C^{\infty}_X)$. The symbol $*$ is reserved for the metric dual, with respect to some metric which will be specified in the text.

Given a coherent sheaf $\mathcal{F}$ of $\mathcal{O}_X$-modules, $H^0(X, \mathcal{F}) = H^0 (\mathcal{F})$ denotes the set of global sections of $\mathcal{F}$; for example, we shall often write $H^0 (\mathcal{F}(k))$ for $H^0(X , \mathcal{F} \otimes L^{\otimes k})$ for any coherent sheaf $\mathcal{F}$ on $X$. We write $\Gamma_{C^{\infty}_X}(X, F) = \Gamma_{C^{\infty}_X} (F)$ for the one for the $C^{\infty}$-vector bundle $F$. Following the standard notation, we also write $\mathcal{E} nd_{\mathcal{O}_X} (\mathcal{F}) := \mathcal{F}^{\vee} \otimes_{\mathcal{O}_X} \mathcal{F}$ and $\mathrm{End}_{\mathcal{O}_X} (\mathcal{F}) := H^0 ( X, \mathcal{E} nd_{\mathcal{O}_X} (\mathcal{F}) )$. When $\mathcal{F}$ is locally free, by regarding it as a $C^{\infty}$-complex vector bundle, we write $\mathrm{End}_{C^{\infty}_X} (F)$ for the set of $C^{\infty}$-sections of the vector bundle $F^{\vee} \otimes F$.

\subsection*{Acknowledgments}

Both authors would like to thank Ruadha\'i Dervan for helpful conversations, and the anonymous referees for many helpful comments. The first named author also thanks Yanbo Fang, Nicholas Lindsay, and Benjamin Sibley for helpful discussions.

\section{Background}\label{background}

A large part of the materials presented in this section can be found in the references \cite{H-L,Kobook,Siu87,L-T1,M-M,P-S03,Thomas,Wang1}.

\subsection{Slope stability}
Let $(X,L)$ be a polarised smooth projective variety of complex dimension $n$, and $\mathcal{E}$ be a holomorphic vector bundle over $X$ of rank $r$.

\begin{definition} \label{defrkdeg}
	Let $\mathcal{F}$ be a coherent sheaf on $X$. Its \textbf{slope} $\mu (\mathcal{F}) = \mu_L (\mathcal{F}) \in \mathbb{Q}$ is defined as
	\begin{equation*}
		\mu (\mathcal{F}) := \frac{\deg (\mathcal{F})}{\mathrm{rk} (\mathcal{F})},
	\end{equation*}
	where $\mathrm{rk} (\mathcal{F}) \in \mathbb{N}$ is the rank of $\mathcal{F}$ (assumed to be non-zero) where it is locally free (\emph{cf.}~\cite[p.11]{H-L}), and $\deg (\mathcal{F}) \in \mathbb{Z}$ is defined as $\int_X c_1 (\det \mathcal{F}) c_1(L)^{n-1}/(n-1)!$ (where $\det \mathcal{F}$ is a line bundle defined in terms of a locally free resolution of $\mathcal{F}$, see \cite{detdiv}, \cite[V.6]{Kobook}  for details).
\end{definition}

In general we should define the rank and degree in terms of the coefficients of the Hilbert polynomial \cite[Definitions 1.2.2 and 1.2.11]{H-L} and hence they are a priori rational numbers. For us, however, they are defined as integers and the above definition suffices since $X$ is smooth.

The following stability notion was first introduced by Mumford for Riemann surfaces, and was generalised to higher dimensional varieties by Takemoto by choosing a polarisation $L$.

\begin{definition}[Slope stability] \label{defmtstab}
	A holomorphic vector bundle $\mathcal{E}$ is said to be {\bf slope stable} (or Mumford--Takemoto stable) if for any coherent subsheaf $\mathcal{F} \subset \mathcal{E}$ with $0 < \mathrm{rk} (\mathcal{F}) < \rk (\mathcal{E})$ we have $\mu(\mathcal{E}) > \mu(\mathcal{F})$. $\mathcal{E}$ is said to be {\bf slope semistable} if the same condition holds with non-strict inequality, and {\bf slope polystable} if it is a direct sum of slope stable bundles with the same slope.
\end{definition}


\begin{definition}
A holomorphic vector bundle $\mathcal{E}$ is said to be
\begin{enumerate}
	\item \textbf{reducible} if it can be written as a direct sum of two or more nontrivial holomorphic subbundles as $\mathcal{E} = \bigoplus_j \mathcal{E}_j$;
	\item \textbf{irreducible} if it is not reducible;
	\item \textbf{simple} if $\mathrm{End}_{\mathcal{O}_X}(\mathcal{E}) = \mathbb{C}$.
\end{enumerate}
\end{definition}
It follows immediately that simple vector bundles are irreducible. The following fact is well known.
\begin{lem} \textup{(\emph{cf.}~\cite[Corollary~V.7.14]{Kobook})} \label{stablesimple}
	If $\mathcal{E}$ is slope stable, then it is simple.
\end{lem}

\begin{definition}
A coherent sheaf $\mathcal{F}$ is called \textbf{reflexive} if it is isomorphic to its \textbf{reflexive hull} $\mathcal{F}^{\vee \vee}$.
\end{definition}

\begin{definition}
	A subsheaf $\mathcal{F}$ of $\mathcal{E}$ is called \textbf{saturated} if $\mathcal{E} / \mathcal{F}$ is torsion-free. Any subsheaf $\mathcal{F}$ of $\mathcal{E}$ has a minimal saturated subsheaf containing $\mathcal{F}$, called the \textbf{saturation} of $\mathcal{F}$ in $\mathcal{E}$, defined as the kernel of the surjection $\mathcal{E} \to \mathcal{E} / \mathcal{F} \to (\mathcal{E} / \mathcal{F}) / T(\mathcal{E} / \mathcal{F})$, where $T(\mathcal{E} / \mathcal{F})$ is the maximal torsion subsheaf of $\mathcal{E} / \mathcal{F}$ \cite[Section 1.1]{H-L}.
\end{definition}

It is well known that saturated subsheaves of a locally free sheaf $\mathcal{E}$ are reflexive (\emph{cf.}~\cite[Proposition~5.5.22]{Kobook}). An important fact is the following.
\begin{lem}\emph{(\emph{cf.}~\cite[Corollaries 5.5.11, 5.5.15, and 5.5.20]{Kobook})} \label{lemcdrefshf}
	Every coherent sheaf is locally free outside of a Zariski closed subset of codimension at least 1. Moreover, torsion-free $($resp.~reflexive$)$ sheaves are locally free outside of a Zariski closed subset of codimension at least $2$ $($resp.~$3)$.
\end{lem}

 We have the following lemma, which follows from the well-known fact that, in order to test slope stability of a vector bundle $\mathcal{E}$, it is sufficient to consider its saturated subsheaves \cite[Proposition 5.7.6]{Kobook}.

\begin{lem}\label{refl}
	$\mathcal{E}$ is slope stable if and only if $\mu (\mathcal{E}) > \mu (\mathcal{F})$ for any saturated subsheaf $\mathcal{F} \subset \mathcal{E}$ with $0 < \rk (\mathcal{F}) < \rk (\mathcal{E})$.
\end{lem}

\subsection{Filtrations and Quot-scheme limit} \label{sccmrquotschl}
We recall some basic definitions and results that concern Quot-schemes. We refer to the reference book of Huybrechts and Lehn \cite{H-L} for more details.
\begin{definition}
	A coherent sheaf $\mathcal{F}$ is said to be \textbf{$k$-regular} if $H^i (\mathcal{F}(k-i)) =0$ for all $i >0$. The \textbf{Castelnuovo--Mumford regularity} of $\mathcal{F}$ is the integer defined by 
	\begin{equation*}
		\mathrm{reg}(\mathcal{F}) := \inf_{k \in \mathbb{Z}} \{ \mathcal{F} \text{ is $k$-regular.} \} .
	\end{equation*}
\end{definition}
By Serre's vanishing theorem, any coherent sheaf is $k$-regular for $k$ large enough and thus its regularity is well defined. Recall \cite[Lemma 1.7.2]{H-L} that if $\mathcal{F}$ is $k$-regular, then $\mathcal{F}$ is $k'$-regular for all $k' \ge k$ and $\mathcal{F}(k)$ is globally generated. Moreover, the natural homomorphism
\begin{equation*}
	H^0 ( \mathcal{F}(k)) \otimes H^0 ( \mathcal{O}_X (l)) \longrightarrow H^0 ( \mathcal{F}(k+l))
\end{equation*}
is surjective for $l \ge 0$.

In particular, if $\mathcal{F}$ is $k$-regular, $\mathcal{F}$ can be written as a quotient
\begin{equation} \label{fglgenquot}
	\rho : V \otimes \mathcal{O}_X (-k) \relbar\joinrel\twoheadrightarrow \mathcal{F}
\end{equation}
where the vector space $V := H^0 ( \mathcal{F}(k))$ has dimension $P_{\mathcal{F}}(k)$, the Hilbert polynomial of $\mathcal{F}$.

We now consider the holomorphic vector bundle $\mathcal{E}$. We take $k \ge \mathrm{reg} (\mathcal{E})$ so that $\mathcal{E}$ is given as a quotient $\rho : V \otimes \mathcal{O}_X (-k) \to \mathcal{E}$, where
$$V = H^0 (\mathcal{E}(k)), \quad \dim V = P_{\mathcal{E}}(k) = h^0 (\mathcal{E}(k)).$$

In other words, $\rho : V \otimes \mathcal{O}_X (-k) \to \mathcal{E}$ defines a point  in the Quot-scheme $\mathrm{Quot} (\mathcal{Q}, P_{\mathcal{E}})$ where we have set $\mathcal{Q} := V \otimes \mathcal{O}_X (-k)$. The reader is referred to \cite[Chapter 2]{H-L} for more details on Quot-schemes.

Suppose that we are given a 1-parameter subgroup (1-PS for short) $\sigma_T : \mathbb{C}^* \ni T \mapsto \sigma_T \in \SL(V)$. This defines a $\mathbb{C}^*$-orbit in $\mathrm{Quot} (\mathcal{Q}, P_{\mathcal{E}})$ defined by $\{ \rho \circ \sigma_T \}_{T \in \mathbb{C}^*}$. We can give an explicit description of the limit of this orbit as $T \to 0$, following \cite[Section 4.4]{H-L}.

Recall that giving a 1-PS $\sigma_T : \mathbb{C}^* \to \SL(V)$ is equivalent to giving a weight space decomposition $V = \bigoplus_{i \in \mathbb{Z}} V_{i}$ of $V$, where $\sigma_T$ acts by $\sigma_T \cdot v = T^i v$ for $v \in V_{i}$. Note that $V_{i} \neq 0$ for only finitely many $i$'s. This defines a filtration
\begin{equation*}
	V_{\le i} := \bigoplus_{j \le i} V_{j}
\end{equation*}
which in turn induces a subsheaf $\mathcal{E}_{\le i}$ of $\mathcal{E}$ by
\begin{equation*}
	\mathcal{E}_{\le i} := \rho (V_{\le i} \otimes \mathcal{O}_X (-k)) ,
\end{equation*}
defining a filtration $\{ \mathcal{E}_{\le i} \}_{i \in \mathbb{Z}}$ of $\mathcal{E}$ by subsheaves. Then $\rho$ induces surjections $\rho_i : V_{i} \otimes \mathcal{O}_X (-k) \to \mathcal{E}_i$ (with $\mathcal{E}_i : = \mathcal{E}_{\le i} / \mathcal{E}_{\le i-1}$), and summing up these graded pieces we get a point in $\mathrm{Quot} (\mathcal{Q},P_{\mathcal{E}})$ represented by the sheaf
\begin{equation} \label{rholimquot}
	\hat{\rho} := \bigoplus_{i \in \mathbb{Z}} \rho_i : V \otimes \mathcal{O}_X (-k) \longrightarrow \hat{\mathcal{E}} := \bigoplus_{i \in \mathbb{Z}} \mathcal{E}_i.
\end{equation}

The importance of this sheaf is that it is, up to equivalence, the limit of the 1-PS $\sigma_T$ in $\mathrm{Quot} (\mathcal{Q}, P_{\mathcal{E}})$; recall that two quotient sheaves $\rho_i : \mathcal{Q} \to \mathcal{E}_i, \ i = 1,2$ are equivalent if and only if there exists an isomorphism $\psi : \mathcal{E}_1 \to \mathcal{E}_2$ of sheaves such that $\rho_2 = \psi \circ \rho_1$. Writing $[ \rho ]$ for the equivalence class of $\rho$, the above result can be stated as follows.

\begin{lem} \textup{(\emph{cf.}~\cite[Lemma 4.4.3]{H-L})} \label{llimquotsch}
	$[\hat{\rho}] \in \mathrm{Quot} (\mathcal{Q},P_{\mathcal{E}})$ is the limit of $[\rho : V \otimes \mathcal{O}_X (-k) \to \mathcal{E}]$ as $T \to 0$ under the 1-parameter subgroup defined by $\sigma_T : \mathbb{C}^*  \to \SL(V)$.
	\end{lem}

\noindent We shall review an outline of the proof of this result, following \cite{H-L}, for the reader's convenience.

\begin{proof}[Sketch of the proof]
	We construct a family of sheaves 
	\begin{equation*}
		\theta : V \otimes \mathcal{O}_X (-k) \otimes \mathbb{C} [T] \longrightarrow \tilde{\mathcal{E}}
	\end{equation*}
	over $\mathbb{C} = \mathrm{Spec}(\mathbb{C}[T])$, such that we have $[\theta_{\alpha}] = [\rho ] \circ \sigma_T$ for the fibre $\theta_{\alpha}$ over $\alpha \in \mathbb{C}^*$ and look at its behaviour under the limit $\alpha \to 0$; $\alpha \in \mathbb{C}^*$ corresponds to the maximal ideal $(T - \alpha) \in \mathrm{Spec}(\mathbb{C}[T])$, and the limit $\alpha \to 0$ corresponds to $T \to 0$. Throughout in the proof, the tensor product is meant to be over $\mathbb{C}$.
	
	Given the decomposition $V = \bigoplus_{i \in \mathbb{Z}} V_{i}$, we define a map
	\begin{equation*}
		\gamma : V \otimes \mathbb{C} [T] = \bigoplus_{i \in \mathbb{Z}} V_{i} \otimes \mathbb{C} [T] \longrightarrow \bigoplus_{i \in \mathbb{Z}} V_{\le i} \otimes T^i \subset V\otimes \mathbb{C}[T,T^{-1}]
	\end{equation*}
	by $\gamma (v_j \otimes p) := v_j \otimes ( T^j \cdot p)$ for $v_j \otimes p \in V_j \otimes \mathbb{C} [T]$. Note that $\gamma$ is an isomorphism of vector spaces which coincides with the action of $\sigma_T$. In other words, the action of $\sigma_T$ on $V$ defines an isomorphism
	\begin{equation*}
		\gamma : V \otimes \mathcal{O}_X (-k) \otimes \mathbb{C} [T] \stackrel{\sim}{\longrightarrow}   \bigoplus_{i \in \mathbb{Z}} V_{\le i} \otimes \mathcal{O}_X (-k)  \otimes T^i =: \tilde{\mathcal{V}},
	\end{equation*}
	where we used the same symbol $\gamma$ by abuse of notation.
	
	Note also that $\rho$ induces the following surjection
	\begin{equation} \label{eqrtdgrd}
		\tilde{\rho} : \tilde{\mathcal{V}} \longrightarrow \tilde{\mathcal{E}} := \bigoplus_{i \in \mathbb{Z}} \mathcal{E}_{ \le i} \otimes T^i,
	\end{equation}
	defining a family $\tilde{\mathcal{E}}$ of sheaves over $\mathbb{C}$. Thus the desired family $\theta$ can be given by $\tilde{\rho} \circ \gamma$; restricting to the central fibre at $T =0$, we have
	\begin{equation*}
		\tilde{\mathcal{E}}_0 = \tilde{\mathcal{E}} / T \cdot \tilde{\mathcal{E}} = \bigoplus_{i \in \mathbb{Z}} \mathcal{E}_i,
	\end{equation*}
	while the noncentral fibre is equivalent to $\rho \circ \sigma_T$ since $\gamma$ coincides with the action of $\sigma_T$.
\end{proof}

We shall later develop a ``differential-geometric'' version of this construction, in terms of Fubini--Study metrics. For this purpose, we shall rephrase the above lemma as follows. Suppose that we write $$\hat{\rho} = \bigoplus_{i \in \mathbb{Z}} \rho_i,$$ with
\begin{equation*}
	 \rho_{i} :  V_{i } \otimes \mathcal{O}_X(-k) \longrightarrow  \mathcal{E}_{i}= \mathcal{E}_{\le i} / \mathcal{E}_{\le i-1}
\end{equation*}
for the limit in the Quot-scheme by $\sigma_T$ as in (\ref{rholimquot}). By abuse of notation we write $\rho (V_i)$ (resp.~$\rho_i (V_i)$) for $\rho (V_i \otimes \mathcal{O}_X(-k))$ (resp.~$\rho_i (V_i \otimes \mathcal{O}_X(-k))$) in what follows. Noting $\rho (\sigma_T \cdot V_{i } ) = T^i \rho ( V_{i })$ and recalling (\ref{eqrtdgrd}), we find that $\rho (\sigma_T \cdot V_{i })$ must be contained in the sheaf $\mathcal{E}_{\le i}$. Using the notation above, we may thus write
\begin{equation*}
	\rho (\sigma_T \cdot V_{i }) = T^{i} \rho_{i} (V_{i  }) + T^{i } (\text{elements of $\mathcal{E}_{\le i-1}$}).
\end{equation*}
Suppose that we write $\rho_{j \to i} (V_{j })$ for the elements of $V_{j }$ that are mapped to $\mathcal{E}_{i}$. With this notation, we collect the above for all $i$'s by using $V = \bigoplus_{i \in \mathbb{Z}} V_{i}$, to get
\begin{equation} \label{exprhoint}
	\rho (\sigma_T \cdot V ) = \sum_{i \in \mathbb{Z}} T^{i} \left( \rho_{i \to i} (V_{i }) + \sum_{j > i} T^{j - i} \rho_{j \to i} (V_{ j }) \right) .
\end{equation}
Thus, we may write the limit of $\rho \circ \sigma_T$ as $T \to 0$ as
\begin{equation} \label{exprhointsht}
	\rho (\sigma_T \cdot V ) = \bigoplus_{i \in \mathbb{Z}} T^{i} \left( \rho_{i} (V_{i }) + o(T) \right),
\end{equation}
where $o(T)$ stands for the terms that go to zero as $T \to 0$, and $T^i$ is the weight of $\sigma_T$ on $\mathcal{E}_i$.

\begin{rem} \label{remxlfen}
	It is important to note that $\mathcal{E}_i$ above is a priori just a coherent sheaf and in general \textit{not} locally free; it may even have torsion. However, it is still possible to define a Zariski open subset of $X$ over which all $\mathcal{E}_{\le i}$ and $\mathcal{E}_i$ are locally free (torsion sheaves are treated as zero), recalling that $X$ is smooth (by Lemma \ref{lemcdrefshf} or \cite[p.11]{H-L}).
\end{rem}

\begin{rem} The Castelnuovo--Mumford regularity of $\mathcal{E}$ does not play a significant role in what follows, contrary to the important role it plays in the construction of the Quot-scheme as above. In the rest of this paper $k$ only needs to be large enough so that $\mathcal{E} (k)$ is globally generated, since all we need is the quotient map (\ref{fglgenquot}), but there will be little harm in consistently assuming $k \ge \mathrm{reg} (\mathcal{E} )$.  
\end{rem}

\subsection{Fubini--Study metrics on vector bundles} \label{fsmetvb}
Suppose that we fix $k$ such that $\mathcal{E}(k)$ is globally generated, and hence there exists a holomorphic map $\Phi: X \to \mathrm{Gr}(r , H^0(\mathcal{E}(k))^{\vee})$ to the Grassmannian such that the pullback under $\Phi$ of the universal bundle is isomorphic to $\mathcal{E}(k)$. More precisely, writing $\mathcal{U}$ for the tautological bundle over the Grassmannian $\mathrm{Gr}(r , H^0(\mathcal{E}(k))^{\vee})$ of $r$-planes (rather than quotients), we have $\Phi^* \mathcal{U}^{\vee} \cong \mathcal{E}(k)$. 

Recall that a positive definite Hermitian form on $H^0 (\mathcal{E}(k))$ defines a Hermitian metric on (the complex vector bundle) $U^{\vee}$ on $\mathrm{Gr}(r , H^0(\mathcal{E}(k))^{\vee})$, called the Fubini--Study metric on the Grassmannian (see \emph{e.g.}~\cite[Section 5.1.1]{M-M}). By pulling them back by $\Phi$, we have Hermitian metrics on $E$ that are also called Fubini--Study metrics. We recall the details of this construction. Noting that the sections of $\mathcal{U}^{\vee}$ can be naturally identified with $H^0 (\mathcal{E}(k))$, Hermitian forms on $U^{\vee}$ are the elements of $H^0 (\mathcal{E}(k))^{\vee} \otimes \overline{H^0 (\mathcal{E}(k))^{\vee}}$, where the bar stands for the complex conjugate.

We fix once for all a reference Hermitian metric $h_{\mathrm{ref}}$ on $E$. With the fixed Hermitian metric $h_L$ on $L$, we have a Hermitian metric $h_{\mathrm{ref}} \otimes h_L^{\otimes k}$ on $E(k)$, we define $\{ e_1 (x) , \dots , e_r (x) \}$ to be an $h_{\mathrm{ref}} \otimes h_L^{\otimes k}$-local orthonormal frame of $E$. We also fix a reference Hermitian form $I_{\mathbb{S}}$ on $H^0(\mathcal{E}(k))$ to be the $L^2$-Hermitian form defined by $h_{\mathrm{ref}} \otimes h_L^{\otimes k}$ and $\omega^n/n!$ (see also Remark \ref{remfsvbs}); the notation here is meant to imply that the basis $\mathbb{S} := \{ s_1 , \dots , s_{N_k} \}$, $N_k := \dim H^0(\mathcal{E}(k))$, for $H^0(\mathcal{E}(k))$ is a fixed $I_{\mathbb{S}}$-orthonormal basis. We may identify $I_{\mathbb{S}}$ with $\mathbb{S}$ up to unitary equivalence, since fixing a positive Hermitian form is equivalent to fixing its (unitary equivalence class of) orthonormal basis. 


We can also define the metric duals of the above. Fixing $\mathbb{S}$ as above, we have a basis $\{ s^*_1 , \dots , s^*_{N_k}  \}$ for $\overline{H^0(\mathcal{E}(k))^{\vee}}$ that is $I_{\mathbb{S}}$-metric dual of $\mathbb{S}$, and a local frame $\{ e^*_1 (x) , \dots , e^*_r (x) \}$ for $\overline{E(k)^{\vee}}$ that is $h_{\mathrm{ref}} \otimes h_L^{\otimes k}$-metric dual to the one $\{ e_1 (x) , \dots , e_r (x) \}$ for $E(k)$. We also write $s_i(x)$ (resp.~$s^*_i (x)$) for the evaluation of $s_i$ (resp.~$s^*_i$) at $x \in X$ with respect to the trivialisation given by $\{ e_1(x) ,  \dots , e_r (x) \}$.

We first define $C^{\infty}$-maps
\begin{equation*}
	Q_{\mathbb{S}, \mathrm{ref}}: E(k) \longrightarrow H^0 (\mathcal{E}(k)) \otimes C^{\infty}_X,
\end{equation*}
and
\begin{equation*}
	Q^*_{\mathbb{S}, \mathrm{ref}}:  \overline{H^0 (\mathcal{E}(k))^{\vee}} \otimes \overline{C^{\infty}_X} \longrightarrow \overline{E(k)^{\vee}},
\end{equation*}
with respect to $\mathbb{S}$ and $\{ e_1 (x) , \dots , e_r (x) \}$ as fixed above and hence defined globally over $X$. The bar, again, denotes the complex conjugate. A well-known result is the following, whose proof can be found in \cite[Remark 3.5]{Wang1} or \cite[Theorem 5.1.16]{M-M}.

\begin{lem} \textup{(\emph{cf.}~\cite[Remark 3.5]{Wang1} or \cite[Theorem 5.1.16]{M-M})} \label{lemfsvbs}
	Suppose that we fix a Hermitian metric $h_{\mathrm{ref}}$ on $E$ and an $h_{\mathrm{ref}} \otimes h_L^{\otimes k}$-local orthonormal frame of $E(k)$ which we write as $\{ e_1(x) ,  \dots , e_r (x) \}$, as above. Let $\mathbb{S} = \{ s_1 , \dots , s_{N_k} \}$ be an orthonormal basis for $H^0(\mathcal{E}(k))$ with respect to the $L^2$-Hermitian form defined by $h_{\mathrm{ref}} \otimes h_L^{\otimes k}$ and $\omega^n/n!$, and write $I_{\mathbb{S}}$ for the $L^2$-Hermitian form itself. Then there exists a local $N_k \times r$ matrix $Q_{\mathbb{S} , \mathrm{ref}} (x)$ which smoothly depends on $x \in X$ such that
	\begin{equation*}
		s_i (x) = \sum_{\alpha =1}^r Q_{\mathbb{S}, \mathrm{ref}} (x)_{i \alpha}e_{\alpha} (x) ,\ \ \  \sum_{i=1}^{N_k} Q^*_{\mathbb{S}, \mathrm{ref}} (x)_{ \alpha i} s^*_i(x)= e^*_{\alpha},
	\end{equation*}
	where the metric dual $*$ is defined with respect to $h_{\mathrm{ref}} \otimes h_L^{\otimes k}$ and $I_{\mathbb{S}}$.
	
	The Hermitian metric on $E(k)$ defined by
	\begin{equation*}
		h_{\mathbb{S}_k , \mathrm{ref}} \otimes h_L^{\otimes k} (x) := Q_{\mathbb{S}, \mathrm{ref}}^*(x) \circ I_{\mathbb{S}} \circ Q_{\mathbb{S}, \mathrm{ref}}(x)
	\end{equation*}
	agrees with the pullback of the Fubini--Study metric on $U^{\vee}$ with respect to $I_{\mathbb{S}}$ over the Grassmannian $\mathrm{Gr}(r , H^0(\mathcal{E}(k))^{\vee})$ by the map $\Phi: X \hookrightarrow \mathrm{Gr}(r , H^0(\mathcal{E}(k))^{\vee})$.
\end{lem}

\begin{rem} \label{remfsvbs}
	In the above, the $L^2$-Hermitian form $I_{\mathbb{S}}$ (defined by $h_{\mathrm{ref}} \otimes h_L^{\otimes k}$ and $\omega$) in fact involves a scaling by $\dim H^0(\mathcal{E}(k))$; see \emph{e.g.}~\cite[Section 5.2.3]{M-M}.
\end{rem}

Suppose that we now change the basis $\mathbb{S}$ to $\sigma \cdot \mathbb{S}$ with $\sigma \in \SL(H^0(\mathcal{E}(k)))$. This defines another Hermitian metric on the universal bundle $U^{\vee}$ over the Grassmannian, which we can pull back by $\Phi$ to get a Hermitian metric on $E(k)$, exactly as before. Then, noting that $(\sigma^{-1})^* \circ I_{\mathbb{S}} \circ \sigma^{-1}$ has $\sigma \cdot \mathbb{S}$ as its orthonormal basis, we may write the Hermitian metric $h_{\sigma \cdot \mathbb{S} , \mathrm{ref}} \otimes h_L^{\otimes k}$ as
\begin{equation*}
	h_{\sigma \cdot \mathbb{S} , \mathrm{ref}} \otimes h_L^{\otimes k} = Q_{\mathbb{S}, \mathrm{ref}}^* \circ ((\sigma^{-1})^* \circ I_{\mathbb{S}} \circ \sigma^{-1}) \circ Q_{\mathbb{S}, \mathrm{ref}},
\end{equation*}
which is also well known in the literature (see \emph{e.g.}~\cite[Section 2]{P-S03}). The appearance of the inverse in $\sigma^{-1}$ can also be seen from the fact that $I_{\mathbb{S}}$ is an element of $ H^0 (\mathcal{E}(k))^{\vee} \otimes \overline{H^0 (\mathcal{E}(k))^{\vee}}$ and hence $\SL(H^0(\mathcal{E}(k)))$ acts by the dual action.

In what follows, we shall adopt the following convention to simplify the notation. The Hermitian metrics $h_{\mathrm{ref}}$ and $h_L$ will be fixed as above, which defines an $L^2$-hermitian inner product $I_{\mathbb{S}}$ (with the associated orthonormal basis $\mathbb{S}$). Moreover, choosing a local orthonormal frame $e_L (x)$ of $L$ with respect to $h_L$, we identify $\{ e_1 (x) , \dots , e_r (x) \}$ with a local orthonormal frame of $E$ (rather than $E(k)$), which is associated to the Hermitian metric $h_{\mathbb{S} , \mathrm{ref}}$ on $E$ in Lemma \ref{lemfsvbs}. Recalling $Q_{\mathbb{S}, \mathrm{ref}} : E(k) \to H^0 (\mathcal{E}(k)) \otimes C^{\infty}_X$, we thus get a $C^{\infty}$-map
\begin{equation}\label{defQ}
	Q : E \longrightarrow H^0 (\mathcal{E}(k)) \otimes C^{\infty}_X (-k).
\end{equation} 
Similarly, $Q_{\mathbb{S}, \mathrm{ref}}^* : \overline{H^0(\mathcal{E}(k))^{\vee}} \otimes \overline{C^{\infty}_X} \to \overline{E(k)^{\vee}}$ yields a $C^{\infty}$-map
\begin{equation}\label{defQ*}
	Q^* : \overline{H^0 (E(k))^{\vee}} \otimes \overline{C^{\infty}_X (k)} \longrightarrow \overline{E^{\vee}},
\end{equation}
by noting $\overline{E(k)^{\vee}} \cong \overline{E^{\vee}} \otimes \overline{C^{\infty}_X(-k)}$. With these conventions, Lemma \ref{lemfsvbs} says $h_{\mathbb{S} , \mathrm{ref}} = Q^* \circ I_{\mathbb{S}} \circ Q$. In what follows we shall simply write $Q^*Q$ for $h_{\mathbb{S}, \mathrm{ref}}$ and omit $I_{\mathbb{S}}$. Also, in order not to be bothered by the inverse in $\sigma^{-1}$, we shall consistently write
\begin{equation*}
	h_{\sigma} := Q^* \sigma^* \sigma Q,
\end{equation*}
with $\sigma \in \SL(H^0(\mathcal{E}(k))^{\vee})$, to denote $h_{\sigma^{-1} \cdot \mathbb{S} , \mathrm{ref}}$.

We summarise the above argument to define the Fubini--Study metric on the bundle $E$.
\begin{definition} \label{deffsdual}
	Given a reference metric $h_{\mathrm{ref}}$ on $E$ and an $L^2$-Hermitian form $I_{\mathbb{S}}$ for $H^0(\mathcal{E}(k))$ with respect to $h_{\mathrm{ref}} \otimes h_L^{\otimes k}$ as above, there exist $C^{\infty}$-maps $Q$ and $Q^*$ defined by \eqref{defQ} and \eqref{defQ*}, such that the Hermitian metric
	\begin{equation*}
		h:= Q^*Q \in \Gamma_{C^{\infty}_X} (E^{\vee} \otimes \overline{E^{\vee}} )
	\end{equation*}
	agrees with the pullback of the Fubini--Study metric on the universal bundle over the Grassmannian defined by the Hermitian form $I_{\mathbb{S}}$ on $H^0(\mathcal{E}(k))$. $h$ is called the (reference) \textbf{Fubini--Study metric} on $E$ defined by $I_{\mathbb{S}}$.
	
	Moreover, given $\sigma \in \SL(H^0(\mathcal{E}(k))^{\vee})$, the Hermitian metric
	\begin{equation*}
		h_{\sigma} := Q^* \sigma^* \sigma Q \in \Gamma_{C^{\infty}_X} (E^{\vee} \otimes \overline{E^{\vee}} )
	\end{equation*}
	agrees with the one defined by the Hermitian form $I_{\sigma^{-1} \cdot \mathbb{S}} $ on $H^0(\mathcal{E}(k))$. $h_{\sigma}$ is called the \textbf{Fubini--Study metric} on $E$ defined by the positive Hermitian form $\sigma^* \sigma = \sigma^* \circ I_{\mathbb{S}} \circ \sigma$ on $H^0(\mathcal{E}(k))$.
\end{definition}
	
When it is necessary to make the exponent $k$ more explicit, we also write $h_k$ for $h$, for example. The construction above gives us a way of associating a Hermitian metric $FS(H)$ on $E$ to a positive definite Hermitian form $H$ on $H^0 (\mathcal{E}(k))$; we simply choose $\sigma \in GL (H^0(\mathcal{E} (k))^{\vee})$ so that $H = \sigma^* \circ I_{\mathbb{S}} \circ \sigma = \sigma^* \sigma$, and define $FS(H) = h_{\sigma}$ as above.

Recall that there is an alternative definition (\cite{Wang1}, see also \cite[Theorem 5.1.16]{M-M}) of $FS(H)$ by means of the equation
\begin{equation} \label{deffseq}
	\sum_{i=1}^{N_k} s_i\otimes s_i^{*_{FS(H)}} = \Id_E
\end{equation}
where $\{s_i\}$ is an $H$-orthonormal basis for $H^0 (\mathcal{E}(k))$ and $s_i^{*_{FS(H)}}$ is the $FS(H)$-metric dual of $s_i$. Defining

\begin{equation*}
	\mathcal{H}_{\infty} := \{ \text{smooth Hermitian metrics on } E \},
\end{equation*}
we thus have a map $FS$ that maps a positive definite Hermitian form $H$ on $H^0 (\mathcal{E}(k))^{\vee}$ to an element in $\mathcal{H}_{\infty}$. Writing $N_k = \dim H^0 ( \mathcal{E}(k))$ and fixing an isomorphism $H^0 (\mathcal{E}(k)) \cong \mathbb{C}^{N_k}$, recall that the homogeneous space $GL(N_k,\mathbb{C}) / U(N_k)$ can be identified with the set $\mathcal{B}_k$ of all positive definitive Hermitian forms on the vector space $H^0 (\mathcal{E}(k))$, by means of the isomorphism $GL(N_k,\mathbb{C}) / U(N_k) \ni \sigma \mapsto \sigma^* \sigma \in \mathcal{B}_k$.

These classical definitions will be important for us later, and are summarised as follows.

\begin{definition} \label{defbergsp}
	The \textbf{Bergman space} $\mathcal{B}_k$ at level $k$ is defined to be the set of all positive definite Hermitian forms on the vector space $H^0 (\mathcal{E}(k))$, which can be identified with the homogeneous space $\mathcal{B}_k = GL(N_k,\mathbb{C}) / U(N_k)$. The \textbf{Fubini--Study map} $FS:\mathcal{B}_k \to \mathcal{H}_{\infty}$
	is defined by the equation (\ref{deffseq}). The image of $\mathcal{B}_k$ by the Fubini--Study map is a subset of $\mathcal{H}_{\infty}$ denoted by $$\mathcal{H}_k := FS(\mathcal{B}_k).$$
\end{definition}

We thus have a natural inclusion $\mathcal{H}_k \subset \mathcal{H}_{\infty}$ for each $k$, but it turns out that the union of $\mathcal{H}_k$'s is dense in $\mathcal{H}_{\infty}$, \emph{i.e.}
\begin{equation} \label{defbergsp2}	
\mathcal{H}_{\infty} = \overline{\cup_{k>>0} \mathcal{H}_k}.
\end{equation}
More precisely, we have that for any $h \in \mathcal{H}_{\infty}$ and $p \in \mathbb{N}$ there exists a sequence $\{ h_{k,p} \}_{k \in \mathbb{N}}$ of Fubini--Study metrics  $ h_{k,p} \in \mathcal{H}_k$ such that $h_{k,p} \to h$ as $k \to + \infty$ in the $C^p$-norm. The result (\ref{defbergsp2}) follows from the so-called Tian--Yau--Zelditch expansion for the Bergman kernel, which is essentially a theorem in analysis. Several proofs have been written especially when $E$ has rank one. For higher ranks, we refer to \cite{Catlin,Wang2}; see also the book \cite{M-M} and references therein. An elementary proof can be found in \cite{BBS2008}.

Finally, recall that the homogeneous space $\mathcal{B}_k = GL(N_k,\mathbb{C}) / U(N_k)$ has a distinguished Riemannian metric called the bi-invariant metric, with respect to which geodesics can be written as $t \mapsto e^{\zeta t}$ for some $ \zeta \in \mathfrak{gl} (N_k , \mathbb{C})$. Recalling the isomorphism $GL(N_k,\mathbb{C}) / U(N_k) \ni \sigma \mapsto \sigma^* \sigma \in \mathcal{B}_k$, we thus have a 1-PS in $\mathcal{B}_k$ defined for each $\zeta \in \mathfrak{gl} (H^0 (\mathcal{E}(k))^{\vee})$ by $\exp( \zeta^* t + \zeta t)$ (note that we may choose the isomorphism $H^0 (\mathcal{E}(k)) \cong \mathbb{C}^{N_k}$ so that the standard basis for $\mathbb{C}^{N_k}$ is $I_{\mathbb{S}}$-orthonormal, and hence $\zeta^*$ is the $I_{\mathbb{S}}$-Hermitian conjugate of $\zeta$). This defines an important 1-PS of Fubini--Study metrics as follows.

\begin{definition} \label{dfbgmopmsg}
	The \textbf{Bergman 1-parameter subgroup} (or Bergman 1-PS) $\{ h_{\sigma_{t}} \}_{t \ge 0} \subset \mathcal{H}_k$ emanating from $h_{\sigma_0} = Q^*Q$ is a path of Fubini--Study metrics defined by
	\begin{equation*}
		h_{\sigma_{t}} := Q^* \sigma_{t}^* \sigma_{t} Q, 
	\end{equation*}
	where $\sigma_{t}:= e^{\zeta t} \in \SL (H^0 (\mathcal{E}(k))^{\vee})$, in the notation of Definition \ref{deffsdual}. When $\zeta$ has rational eigenvalues, we call $\{ h_{\sigma_{t}} \}_{t \ge 0}$ a \textbf{rational} Bergman 1-PS. We say that a Bergman 1-PS $\{ h_{\sigma_{t}} \}_{t \ge 0} \subset \mathcal{H}_k$ is \textit{generated by} $\zeta \in \mathfrak{sl} (H^0 (\mathcal{E}(k))^{\vee})$ or \textit{associated to} the 1-PS  $\sigma : \mathbb{C}^* \to \SL (H^0 (\mathcal{E}(k))^{\vee})$ (in the category of complex Lie groups), where $\sigma_{t} = e^{\zeta t}$.
\end{definition}

The Bergman 1-PS's can be regarded as an analogue of geodesics in $\mathcal{B}_k = GL(N_k,\mathbb{C}) / U(N_k)$, and play an important role later.

\begin{rem}
Throughout in what follows, we shall only consider $\zeta$'s that have rational eigenvalues, and hence any Bergman 1-PS will be meant to be rational (see \emph{e.g.}~Sections \ref{Qlimit} and \ref{RMstab}).	
\end{rem}

We also recall that, in fact, $\mathcal{H}_{\infty}$ can be regarded as an ``infinite dimensional homogeneous space'', with a well-defined notion of geodesics, as explained in \cite[Section 6.2]{Kobook}.

\begin{definition} \label{defgeod}
A path $\{ h_s \}_{s \in \mathbb{R}} \subset \mathcal{H}_{\infty}$ is called a \textbf{geodesic} in $\mathcal{H}_{\infty}$ if it satisfies
\begin{equation} \label{geodcalh}
	\partial_s (h_s^{-1} \partial_s h_s ) =0,
\end{equation}
as an equation in $\mathrm{End}_{C^{\infty}_X}(E)$; an overall constant scaling $h_s := e^{bs}h_0$ for some $b \in \mathbb{R}$ will be called a \textbf{trivial} geodesic.	
\end{definition}
An elementary yet crucially important fact is that $\mathcal{H}_{\infty}$ is \textit{geodesically complete}; for any $h_0, h_1 \in \mathcal{H}_{\infty}$ there exists a geodesic path $\{ h_s \}_{0 \le s \le 1}$ connecting them; we can write it explicitly as $h_s = \exp( s \log h_1 h^{-1}_0) h_0$.

\begin{rem} \label{rembergsubgd}
	Although the Bergman 1-PS $\{ h_{\sigma_{t}} \}_{t \ge 0}$ does not satisfy the geodesic equation (\ref{geodcalh}), we can show \cite[Corollary 4.5]{HK3} that it is in fact a \textit{subgeodesic} in $\mathcal{H}_{\infty}$, \emph{i.e.}~$\partial_{t} (h_{\sigma_{t}}^{-1} \partial_{t} h_{\sigma_{t}} )$ is positive semidefinite as an element in $\mathrm{End}_{C^{\infty}_X}(E)$. While we do not use this result in this paper, it seems interesting to point out an analogy to the case of varieties, in which test configurations define subgeodesics (see \emph{e.g.}~\cite[Section 3]{BHJ2}).
\end{rem}

\subsection{$Q^*$ as a $C^{\infty}$-quotient map} \label{qstcinfqm}
We give here an interpretation of $Q^*$ as a ``$C^{\infty}$-analogue'' of the quotient map $\rho : H^0 (\mathcal{E}(k)) \otimes \mathcal{O}_X (-k) \to \mathcal{E}$ as defined in (\ref{fglgenquot}). While the argument in this section is elementary, it will provide a geometric intuition for the later arguments.

We start with some observations. Suppose that $\mathcal{E}(k)$ is a globally generated holomorphic vector bundle with the quotient map
	\begin{equation} \label{defquotrhol}
		\rho: H^0 (\mathcal{E}(k) ) \otimes \mathcal{O}_X (-k) \longrightarrow \mathcal{E}.
	\end{equation}
	Forgetting the holomorphic structure, we get a map
	\begin{equation} \label{defquotrinf}
		\rho_{C^{\infty}} : H^0 (\mathcal{E}(k) ) \otimes C^{\infty}_X (-k) \longrightarrow E.
	\end{equation}
Given the above definitions, the following is immediate.
\begin{lem} \label{cdiamact}
	Writing $V := H^0 ( \mathcal{E}(k))$, we have the following commutative diagram:
	\begin{displaymath}
		\xymatrix{ \SL(V) \times V \otimes C^{\infty}_X (-k) \ar@{->}[r]^-{a} & V \otimes C^{\infty}_X (-k) \ar@{->}[r]^-{\rho_{C^{\infty}}} & E \\
		\SL(V) \times V \otimes \mathcal{O}_{X} (-k) \ar@{->}[r]_-{a} \ar@{->}[u]^{\textup{forget}} & V \otimes \mathcal{O}_{X } (-k) \ar@{->}[r]_-{\rho} \ar@{->}[u]^{\textup{forget}} & \mathcal{E}  \ar@{->}[u]^{\textup{forget}}
		}
	\end{displaymath}
	where $a : \SL(V) \times V \to V$ is the linear action on $V$, and the vertical arrows are the forgetful maps forgetting the holomorphic structures.
\end{lem}

Choosing Hermitian metrics for $H^0 (\mathcal{E}(k) )$ and $E(k)$, we can take the ``metric dual'' of (\ref{defquotrinf}). More precisely, recalling $I_{\mathbb{S}}$ and $h_{\mathrm{ref}}$ that appeared in the previous section, we have the following $C^{\infty}$-maps
\[
	I_{\mathbb{S}} : H^0 (\mathcal{E}(k) ) \stackrel{\sim}{\longrightarrow} \overline{H^0 (\mathcal{E}(k) )^{\vee}},\quad
	h_{\mathrm{ref}} : E \stackrel{\sim}{\longrightarrow} \overline{E^{\vee}},\quad\text{and}\quad  h_L : C^{\infty}_X (-k) \stackrel{\sim}{\longrightarrow} \overline{C^{\infty}_X (k)},
\]
by means of the metric dual. Composing with these maps, we get
\begin{equation} \label{metdualrhoift}
	\rho^*_{C^{\infty}} : \overline{ H^0 (\mathcal{E} (k) )^{\vee}} \otimes \overline{C^{\infty}_X (k)} \longrightarrow \overline{E^{\vee}},
\end{equation}
which looks similar to the definition of $Q^*$ that we had before.

We now state the relationship between $\rho^*_{C^{\infty}}$ and $Q^*$.

\begin{lem} \label{cdiamq}
	Suppose that $\mathcal{E}(k)$ is a globally generated holomorphic vector bundle with the quotient map $\rho$. Fix a Hermitian metric $I_{\mathbb{S}}$ for $H^0(\mathcal{E}(k))$, $h_{\mathrm{ref}}$ for $E$, and $h_L$ for $L$, so that we have $Q = Q_{\mathbb{S} , \mathrm{ref}}$ as in Lemma \ref{lemfsvbs}. Then there exists a $C^{\infty}$-gauge transform of $\overline{E^{\vee}}$, written $f_{\mathbb{S} , \mathrm{ref}} \in \mathrm{End}_{C^{\infty}_X} (\overline{E^{\vee}})$, which is fibrewise invertible, such that
	\begin{equation} \label{lemeqqstqtrho}
		Q^* = f_{\mathbb{S} ,  \mathrm{ref}} \circ \rho^*_{C^{\infty}}.
	\end{equation}
\end{lem}

\begin{proof}
	Since both $Q^* $ and $\rho^*_{C^{\infty}}$ can be written locally as an $N_k \times r$ matrix of fibrewise full-rank which depends smoothly on $x$ (\emph{cf.}~Lemma \ref{lemfsvbs}), it is immediate that there exists a fibrewise invertible smooth linear map $f_{\mathbb{S} , \mathrm{ref}} : \overline{E^{\vee}} \longrightarrow \overline{E^{\vee}}$ so that (\ref{lemeqqstqtrho}) holds.
\end{proof}

\begin{rem} \label{remfsfbwgs}
	$f_{\mathbb{S} , \mathrm{ref}}$ above should be regarded as performing a ``fibrewise Gram--Schmidt process'' with respect to $I_{\mathbb{S}}$ and $h_{\mathrm{ref}}$, so that $\{ s^*_1 , \dots , s^*_{N_k} \}$ are mapped to the frame $\{ e^*_1 (x) , \dots , e^*_r (x) \}$ as in Lemma \ref{lemfsvbs}.
	\end{rem}

While the above lemma looks rather trivial, we observe the following consequence that can be obtained from it combined with Lemma \ref{cdiamact}; when we consider the Quot-scheme limit $\hat{\mathcal{E}}$ of $\mathcal{E}$ generated by a $\mathbb{C}^*$-action given by $\zeta \in \mathfrak{sl}( H^0 (\mathcal{E} (k))^{\vee})$, the limit of the Bergman 1-PS $Q^*e^{\zeta^* t} e^{\zeta t} Q$ as $t \to \infty$ must be in a certain sense ``compatible'' with this Quot-scheme limit. Indeed, this will be the main topic to be discussed in the next section.

\subsection{Donaldson's functional and the Hermitian--Einstein equation}\label{defMDon}
As before, $(X,L)$ stands for a polarised smooth complex projective variety and $\mathcal{E}$ a holomorphic vector bundle of rank $r$. We write $\Vol_L := \int_X c_1(L)^n/n!$. The following functional plays a central role in our paper.

\begin{definition}
 Given two Hermitian metrics $h_0$ and $h_1$ on $E$, \textbf{the Donaldson functional} ${\mathcal M}^{\mathrm{Don}} : \mathcal{H}_{\infty}\times \mathcal{H}_{\infty} \to \mathbb{R}$ is defined as $${\mathcal M}^{\mathrm{Don}}(h_1,h_0) := {\mathcal M}_1^{\mathrm{Don}}(h_1,h_0) - \mu(E){\mathcal M}_2^{\mathrm{Don}}(h_1,h_0),$$
 where $${\mathcal M}_1^{\mathrm{Don}}(h_1,h_0) :=\int_0^1 \int_X \tr \left( h_t^{-1} \partial_t h_t \cdot F_t\right) \frac{\omega^{n-1}}{(n-1)!} dt$$ and $${\mathcal M}_2^{\mathrm{Don}}(h_1,h_0) :=\frac{1}{\mathrm{Vol}_L}\int_X \log\det(h_0^{-1}h_1)\frac{\omega^n}{n!}.$$
 
 In the above, $\{ h_t \}_{0 \le t \le 1} \subset \mathcal{H}_{\infty}$ is a smooth path of Hermitian metrics between $h_0$ to $h_1$ and $F_t$ denotes ($\ai / 2 \pi$) times the Chern curvature of $h_t$, with respect to the fixed holomorphic structure of $E$ (such that the sheaf of germs of holomorphic sections of $E$ is identified with $\mathcal{E}$).
\end{definition}

\begin{rem} \label{remordargdfn}
	Throughout in what follows, we shall fix the second argument of ${\mathcal M}^{\mathrm{Don}}$ as a reference metric. Thus ${\mathcal M}^{\mathrm{Don}}(h,h_0)$ is regarded as a function of $h$ with a fixed reference $h_0$.
\end{rem}

We recall some well-known properties of this functional as established in \cite{Don1985}; the reader is also referred to \cite[Section 6.3]{Kobook} for more details. First of all it is well-defined, \emph{i.e.}~does not depend on the path $\{ h_t \}_{0 \le t \le 1}$ chosen to connect $h_0$ and $h_1$ (\emph{cf.}~\cite[Lemma 6.3.6]{Kobook}). This easily implies the following \textbf{cocycle property}
\begin{equation} \label{cocyclemdon}
	\mathcal{M}^{\mathrm{Don}} (h_2,h_0) = \mathcal{M}^{\mathrm{Don}} (h_2,h_1)+ \mathcal{M}^{\mathrm{Don}} (h_1,h_0) ,
\end{equation}
for any $h_0 , h_1 , h_2 \in \mathcal{H}_{\infty}$.

\begin{rem}
	Recalling $\mathcal{M}^{\mathrm{Don}} (e^{c} h,h)=0$ for any constant $c \in \mathbb{R}$ \cite[Lemma 6.3.23]{Kobook}, (\ref{cocyclemdon}) implies
	\begin{equation*}
		{\mathcal M}^{\mathrm{Don}}(e^{c} h,h_1) = {\mathcal M}^{\mathrm{Don}}(e^{c}h,h)+ {\mathcal M}^{\mathrm{Don}}(h,h_1) = {\mathcal M}^{\mathrm{Don}}(h,h_1)
	\end{equation*}
	for any $c \in \mathbb{R}$, \emph{i.e.}~$\mathcal{M}^{\mathrm{Don}}(-,h_1)$ is invariant under an overall constant scaling.
\end{rem}

Second, the Euler--Lagrange equation for $\mathcal{M}^{\mathrm{Don}}(-,h_{\mathrm{ref}})$ is the \textit{Hermitian--Einstein equation}. More precisely,  a critical point of $\mathcal{M}^{\mathrm{Don}}(-,h_{\mathrm{ref}})$ is the following object.

\begin{definition}
	A Hermitian metric $h \in \mathcal{H}_{\infty}$ is called a \textbf{Hermitian--Einstein metric} if it satisfies $$\Lambda_\omega F_h = \frac{\mu(E)}{\Vol_L} \Id_E,$$
where $\Lambda_\omega$ is the contraction with respect to the K\"ahler metric $\omega$ on the manifold (\emph{i.e.}~the dual Lefschetz operator).
\end{definition}

Critical points of $\mathcal{M}^{\mathrm{Don}}$ being Hermitian--Einstein metrics is a consequence of the following lemma.

\begin{lem} \textup{(\emph{cf.}~\cite[Section 1.2]{Don1985})} \label{derivmdon}
	Fixing a reference metric $h_{\mathrm{ref}} \in \mathcal{H}_{\infty}$, we have
	\begin{align*}
		\frac{d}{dt} \mathcal{M}^{\mathrm{Don}}_1 (h_t , h_{\mathrm{ref}}) &= \int_X \tr \left( h_t^{-1} \partial_t h_t \cdot F_t \right) \frac{\omega^{n-1}}{(n-1)!}, \\
		\frac{d}{dt} \mathcal{M}^{\mathrm{Don}}_2 (h_t , h_{\mathrm{ref}}) &= \frac{1}{\mathrm{Vol}_L}\int_X \tr( h^{-1}_t \partial_t h_t ) \frac{\omega^n}{n!},
	\end{align*}
	for a path $\{ h_t \}_{0 \le t \le 1} \subset \mathcal{H}_{\infty}$ of smooth metrics with $h_0 =h_{\mathrm{ref}}$, and hence
	\begin{equation*}
		\frac{d}{dt} \mathcal{M}^{\mathrm{Don}} (h_t , h_{\mathrm{ref}}) = \int_X \tr \left( h_t^{-1} \partial_t h_t \left( \Lambda_{\omega} F_t - \frac{\mu(E)}{\mathrm{Vol}_L} \mathrm{Id}_E \right) \right) \frac{\omega^{n}}{n!}.
	\end{equation*}
\end{lem}

The final important point is that $\mathcal{M}^{\mathrm{Don}}$ is convex along geodesics in $\mathcal{H}_{\infty}$, as defined in (\ref{geodcalh}). Recalling the important fact that $\mathcal{H}_{\infty}$ is geodesically complete, these properties together imply that a critical point of $\mathcal{M}^{\mathrm{Don}}$ is necessarily the global minimum. We summarise the above in the following proposition.

\begin{prop} \textup{(\emph{cf.}~\cite[Section 6.3]{Kobook})} \label{lemmdonconvH}
 The functional $\mathcal{M}^{\mathrm{Don}}$ is convex along geodesics in $\mathcal{H}_{\infty}$, and its critical point $($if it exists$\, )$ attains the global minimum. Moreover, $\mathcal{M}^{\mathrm{Don}}$ is strictly convex along nontrivial geodesics if $E$ is irreducible $($in particular if $E$ is simple$\, )$ and in this case the critical point $($if it exists$\, )$ is unique up to an overall constant scaling.
\end{prop}

\section{Quot-scheme limit of Fubini--Study metrics}\label{Qlimit}

\subsection{Filtrations and metric duals}

A central role in this paper is played by 1-parameter subgroups in $\SL(H^0(\mathcal{E}(k))^{\vee})$ and the weight decomposition of $H^0(\mathcal{E}(k))$ associated to them. This is of course related to the Bergman 1-PS's in Definition \ref{dfbgmopmsg}, and for the compatibility with the differential-geometric situation therein we need to assume that 1-parameter subgroups that we consider are the complexifications of the fixed unitary action; more precisely, we fix a Hermitian form $I_{\mathbb{S}}$ for $H^0(\mathcal{E}(k))$ which fixes $Q^*$ as in Section \ref{fsmetvb}, or alternatively the unitary equivalence class of its orthonormal basis $\mathbb{S}$, and decree that the unitary group $U(N_k)$ acting on $H^0(\mathcal{E}(k))^{\vee}$ is always with respect to $I_{\mathbb{S}}$.

It is well known that the set of all 1-parameter subgroups in a torus (in the category of algebraic groups) forms a lattice, and we can take its tensor product over $\mathbb{Q}$. We cannot naively repeat this operation for $\SL(H^0(\mathcal{E}(k))^{\vee})$ as it is nonabelian, but the following is still well defined and plays a very important role in the rest of the paper.
\begin{definition} \label{defcochq}
	We write $\mathcal{X}_{\mathbb{Q}} (k)$ (or simply $\mathcal{X}_{\mathbb{Q}}$) for the set of 1-parameter subgroups $\sigma :  \mathbb{C}^* \to \SL(H^0(\mathcal{E}(k))^{\vee})$ in the category of complex Lie groups with the following property: for each $\sigma \in \mathcal{X}_{\mathbb{Q}}(k)$ there exists an integer $m \ge1$ such that $\sigma^m :  \mathbb{C}^* \to \SL(H^0(\mathcal{E}(k))^{\vee})$ is a 1-parameter subgroup in the category of algebraic groups and that $\sigma^m$ is the complexification of a unitary 1-parameter subgroup $U(1) \to U(N_k)$ where the unitarity is defined with respect to $I_{\mathbb{S}}$ on $H^0(\mathcal{E}(k))$.
\end{definition}

In a more down-to-earth terminology, $\mathcal{X}_{\mathbb{Q}}$ can be defined as matrices as follows. An element $\sigma \in \mathcal{X}_{\mathbb{Q}}$ corresponds one-to-one with a Hermitian element $\zeta \in \mathfrak{sl} (H^0(\mathcal{E}(k))^{\vee})$ with rational eigenvalues by $\sigma_{t} := e^{\zeta t}$; it is always understood that $\zeta$ is Hermitian with respect to $I_{\mathbb{S}}$. We may change $\mathbb{S}$ (without affecting $Q^*Q$) to a unitarily equivalent basis $\tau \cdot \mathbb{S}$ ($\tau \in U(N_k)$), say, so that it is compatible with the weight space decomposition

 \begin{equation} \label{wtdcpzeta}
 	H^0 ( \mathcal{E}(k))^{\vee} = \bigoplus_{i = 1}^{\nu} V^{\vee}_{w_i}
 \end{equation}
defined by $\sigma \in \mathcal{X}_{\mathbb{Q}}(k)$, where its generator $\zeta \in \mathfrak{sl} (H^0(\mathcal{E}(k))^{\vee})$ acts on $V^{\vee}_{w_i}$ with weight $w_i \in \mathbb{Q}$ for all $i = 1 , \dots , \nu$. Note that $\nu$ depends on $\zeta$. Henceforth we shall assume the ordering
\begin{equation} \label{ordlambda}
	w_1 > \cdots > w_{\nu} .
\end{equation}

\
\begin{rem} \label{condopnm}
	We shall assume without loss of generality that the operator norm (\emph{i.e.}~the modulus of the maximum eigenvalue) of $\zeta$ is at most 1; this can be achieved by an overall constant multiple of $\zeta$, which just results in a different scaling of the parameter $t$.
	
	Note also that a Bergman 1-PS $\{ h_{\sigma_{t}} \}_{t \ge 0}$ is rational in the sense of Definition \ref{dfbgmopmsg} if and only if $\sigma \in \mathcal{X}_{\mathbb{Q}}$.
\end{rem}

We shall introduce an auxiliary variable 
\begin{equation*}
	T := e^{-t},
\end{equation*}
so that the limit $t \to + \infty$ corresponds to $T \to 0$, which makes it easier for us to compare the limit $t \to + \infty$ to the description of the Quot-scheme limit as a central fibre, as presented in Lemma \ref{llimquotsch}. In this description, for each $\sigma \in \mathcal{X}_{\mathbb{Q}}$ we may write
\begin{equation} \label{chvsgttz}
	\sigma_t = T^{- \zeta},
\end{equation}
for some $\zeta \in \mathfrak{sl} (H^0(\mathcal{E}(k))^{\vee})$ and the weight space decomposition (\ref{wtdcpzeta}) defines a filtration
\begin{equation} \label{wtdcpfilt}
V^{\vee}_{ \le - w_i } := \bigoplus_{j= 1}^{i} V^{\vee}_{w_j},
\end{equation}
where we note that the minus sign appears because (\ref{chvsgttz}) gives the ordering of the weights
\begin{equation} \label{ordmlambda}
-w_1 < \dots < - w_{\nu}
\end{equation}
with respect to $T$.

We now consider the filtration that $\zeta \in \mathfrak{sl} (H^0(\mathcal{E}(k))^{\vee})$ induces on $H^0(\mathcal{E}(k))$. It turns out that the natural dual action is not the one relevant for our purpose. Recall that Lemma \ref{cdiamq} implies that we have the following diagram
\begin{equation} \label{cdiamqd}
		\xymatrix{ \overline{H^0(\mathcal{E}(k))^{\vee}} \otimes \overline{C^{\infty}_X (k)} \ar@{->}[r]^-{Q^*} \ar@{->}[d]_{*} & \overline{E^{\vee}} \ar@{->}[d]^{*, \ \textup{gauge}} \\
		 H^0(\mathcal{E}(k)) \otimes \mathcal{O}_{X } (-k) \ar@{->}[r]_-{\rho}  & \mathcal{E}
		}
\end{equation}
where the left vertical arrow stands for the metric duality isomorphism given by $I_{\mathbb{S}}$ and $h_L$, and the right vertical arrow for the one given by $h_{\mathrm{ref}}$ and the $C^{\infty}$-gauge transformation (written $f_{\mathbb{S}, \mathrm{ref}}$ in Lemma \ref{cdiamq}). The fact that the right vertical arrow involves taking the dual of the bundle $\mathcal{E}$ will have nontrivial consequences concerning the sign (\emph{cf.}~Remark \ref{remmdsign}).

We decree that the $I_{\mathbb{S}}$-metric dual gives a weight-preserving isomorphism between $H^0(\mathcal{E}(k))$ and $\overline{H^0(\mathcal{E}(k))^{\vee}}$, and hence defines a weight decomposition
\begin{equation*}
	H^0 ( \mathcal{E}(k)) = \bigoplus_{i = 1}^{\nu} V_{w_i }
\end{equation*}
in which $\zeta \in \mathfrak{sl} (H^0(\mathcal{E}(k))^{\vee})$ acts with weight $w_i$ on $V_{w_i }$. We thus get a filtration
\begin{equation*}
	V_{ \le - w_i } := \bigoplus_{j= 1}^{i} V_{w_j},
\end{equation*}
as in (\ref{wtdcpfilt}), which defines a filtration
\begin{equation} \label{pfiltevs}
		0 \neq \mathcal{E}_{\le - w_{1}} \subset \cdots \subset \mathcal{E}_{\le - w_{\nu}} = \mathcal{E}
\end{equation}
of $\mathcal{E}$ by subsheaves of $\mathcal{E}$, where $\mathcal{E}_{\le - w_{i}}$ is defined by the quotient map
\begin{equation*}
	\rho : V_{\le - w_i } \otimes \mathcal{O}_X(-k) \longrightarrow \mathcal{E}_{\le - w_i}.
\end{equation*}
Recall, as in Lemma \ref{llimquotsch}, that the Quot-scheme limit associated to the above filtration is given by $$\hat{\mathcal{E}} := \bigoplus_{i=1}^{\nu} \mathcal{E}_{-w_i},$$ where $ \mathcal{E}_{-w_i} = \mathcal{E}_{ \le -w_i} / \mathcal{E}_{\le -w_{i-1}}$ for $i =1 , \dots , \nu$ (in which we assume $\mathcal{E}_{-w_0} =0$).

The components of $\hat{\mathcal{E}}$ with nontrivial rank are the objects that we need later.

\begin{definition}
	 We define a subset $\{ \hat{1}, \dots , \hat{\nu} \}$ of $\{ 1 , \dots , \nu \}$, with ordering $\hat{1} < \dots < \hat{\nu}$, so that $\mathrm{rk} (\mathcal{E}_{- w_{\alpha}}) >0$ if and only if $\alpha \in \{ \hat{1}, \dots , \hat{\nu} \}$.
\end{definition}

The above definition, together with the invariance of the rank under the flat limit in Quot-schemes, implies
\begin{equation*}
	\sum_{\alpha = \hat{1}}^{\hat{\nu}} \mathrm{rk} (\mathcal{E}_{- w_{\alpha}}) = \mathrm{rk} (\mathcal{E}).
\end{equation*}

\begin{rem} \label{remcombfstfac}
	Since $\mathcal{E}_{\le - w_{1}}$ is a non-zero torsion-free subsheaf of $\mathcal{E}$, we find $\hat{1} =1$ in the above notation. In particular, note that $\mathcal{E}_{\le - w_{1}} \subset \mathcal{E}_{\le - w_{\alpha}}$ for all $\alpha \in \{ \hat{1}, \dots , \hat{\nu} \}$.
\end{rem}

We shall mainly work over an open dense subset of $X$ where each $\mathcal{E}_{\le - w_i}$ is a holomorphic subbundle of $\mathcal{E}$. We make the following definition, by recalling the standard definition of the singular locus.
\begin{definition}
	The \textbf{singular locus} $\mathrm{Sing} (\mathcal{E}_{\le - w_i})$ of $\mathcal{E}_{\le - w_i}$ is the Zariski closed set in $X$ where $\mathcal{E}_{\le - w_i}$ is not locally free. We define the \textbf{regular locus} to be the complement of
	\begin{equation*}
		\bigcup_{i=1}^{\nu} \mathrm{Sing} (\mathcal{E}_{\le - w_i}) \cup \mathrm{Sing} (\mathcal{E} / \mathcal{E}_{\le - w_i})
	\end{equation*}
	in $X$, and denote it by $X^{\mathrm{reg}} (\zeta)$ or $X^{\mathrm{reg}}$.
\end{definition}

Thus each $\mathcal{E}_{\le - w_i}$ is a holomorphic subbundle of $\mathcal{E}$ over $X^{\mathrm{reg}}$, by recalling that a subsheaf $\mathcal{F}$ of a locally free sheaf $\mathcal{E}$ is a holomorphic subbundle of $\mathcal{E}$ if and only if both $\mathcal{F}$ and $\mathcal{E} / \mathcal{F}$ are locally free. Observe that $X \setminus X^{\mathrm{reg}}$ has codimension at least 1 in $X$ by Lemma \ref{lemcdrefshf}.

We define a $C^{\infty}$ complex vector bundle $\widehat{E}$ over $X^{\mathrm{reg}}$ as
\begin{equation*}
	\widehat{E} := \bigoplus_{\alpha =\hat{1}}^{\hat{\nu}} E_{\le -w_{\alpha}} / E_{\le -w_{\alpha-1}},
\end{equation*}
where $E_{\le -w_{\alpha}}$ is the $C^{\infty}$ complex vector bundle over $X^{\mathrm{reg}}$ defined by $\mathcal{E}_{\le -w_{\alpha}}$, and the quotient is the one of the $C^{\infty}$ complex vector bundles over $X^{\mathrm{reg}}$. Note that we have an isomorphism $\widehat{E} \isom E$ as $C^{\infty}$-vector bundles over $X^{\mathrm{reg}}$. Combining this with another $C^{\infty}$-isomorphism $E \isom \overline{E^{\vee}}$, afforded by the $h_{\mathrm{ref}}$-metric dual, we get
\begin{equation} \label{isomeddce}
	\overline{E^{\vee}} \isom \bigoplus_{\alpha =\hat{1}}^{\hat{\nu}} E_{-w_{\alpha}}
\end{equation}
where we wrote $E_{ -w_{\alpha}}  = E_{\le -w_{\alpha}} / E_{\le -w_{\alpha-1}}$.

It turns out that what we need later is the modification of the filtration $0 \neq \mathcal{E}_{\le - w_{1}} \subset \cdots \subset \mathcal{E}_{\le - w_{\nu}} = \mathcal{E}$ on a Zariski closed subset of $X$, given by the following lemma.

\begin{lem} \label{lemcdquot}
	There exists a filtration
\begin{equation} \label{filtevs}
		0 \neq \mathcal{E}'_{\le - w_{1}} \subset \cdots \subset \mathcal{E}'_{\le - w_{\nu}} = \mathcal{E}
\end{equation}
of $\mathcal{E}$ by subsheaves of $\mathcal{E}$, such that 
\begin{enumerate}
	\item $\mathcal{E}'_{\le -w_{i}}$ is saturated in $\mathcal{E}$ and agrees with $\mathcal{E}_{\le {-w_i}}$ on a Zariski open subset of $X$ for all $i = 1 , \dots , \nu$,
\item setting $\mathcal{E}'_{-w_i} := \mathcal{E}'_{\le -w_{i}} / \mathcal{E}'_{\le -w_{i-1}}$, $\mathrm{rk}(\mathcal{E}'_{-w_i}) >0$ if and only if $i \in \{\hat{1}, \dots , \hat{\nu} \}$.
\end{enumerate}
\end{lem}

\begin{proof}
	We define (\ref{filtevs}) inductively as follows. $\mathcal{E}'_{\le - w_{\hat{\nu}}}$ is defined as the saturation of $\mathcal{E}_{\le - w_{\hat{\nu}}}$ in $\mathcal{E}$; note that this implies $\mathcal{E}'_{\le - w_{\hat{\nu}}} = \mathcal{E}$, since $\mathcal{E}'_{\le - w_{\hat{\nu}}}$ has full rank. We then replace $\mathcal{E}_{\le - w_{\hat{\nu}+1}}, \dots , \mathcal{E}_{\le -w_{\nu}}$ by $\mathcal{E}'_{\le - w_{\hat{\nu}}} = \mathcal{E}$. Continuing inductively, for $\alpha \in \{ \hat{1} , \dots , \widehat{\nu-1} \}$ we define $\mathcal{E}'_{\le - w_{\alpha}}$ to be the saturation of $\mathcal{E}_{\le - w_{\alpha}}$ in $\mathcal{E}$. Defining $\beta$ to be the smallest element of $\{ \hat{1} , \dots , \hat{\nu} \}$ that is larger than $\alpha \in \{ \hat{1} , \dots , \widehat{\nu -1} \}$, we replace $\mathcal{E}_{\le - w_{\alpha +1}}, \dots , \mathcal{E}_{\le -w_{\beta -1 }}$ by $\mathcal{E}'_{\le - w_{\alpha}}$. The end product that we get is a filtration $0 \neq \mathcal{E}'_{\le - w_{1}} \subset \cdots \subset \mathcal{E}'_{\le - w_{\nu}} = \mathcal{E}$ of $\mathcal{E}$ by saturated subsheaves, in which strict inclusion holds if and only if $i \in \{ \hat{1} , \dots , \hat{\nu} \}$ (recall also $\hat{1}=1$ as in Remark \ref{remcombfstfac}).
	
	Since the above construction does not change the rank of each $\mathcal{E}_{\le - w_{i}}$ for all $i = 1 , \dots , \nu$, and hence does not change $\mathcal{E}_{\le - w_{i}}$ where it is locally free, we see that $\mathcal{E}_{\le - w_{i}}$ and $\mathcal{E}'_{\le - w_{i}}$ agree on the (Zariski open) locus where they are locally free.
\end{proof}

\subsection{Renormalised Quot-scheme limit of Fubini--Study metrics}
Our aim in this section is to find an expansion for $h_{\sigma_{t}} := Q^* \sigma^*_{t} \sigma_{t} Q$ as $t \to + \infty$. Although this limit is divergent, we can find an appropriate rescaling (in terms of the constant gauge transformation) to get a 1-PS of Hermitian metrics $\{ \hat{h}_{\sigma_{t}} \}_{t \ge 0}$ which is convergent. The limit $\hat{h}$, which we call the renormalised Quot-scheme limit of $h_{\sigma_{t}}$, is well defined only on the Zariski open subset $X^{\mathrm{reg}}$ of $X$, and it may tend to a degenerate metric as it approaches the singular locus $X \setminus X^{\mathrm{reg}}$. In particular, its curvature may blow up. In spite of this problem, this limit will play an important role later. A combinatorial argument making use of the ordering $w_1 > \cdots > w_{\nu}$ (\ref{ordlambda}) is of crucial importance for the definition of $\hat{h}$.

\begin{definition} \label{defqij}
Suppose that we fix a $C^{\infty}$-isomorphism $\overline{E^{\vee}} \isom \bigoplus_{\alpha =\hat{1}}^{\hat{\nu}} E_{- w_{\alpha}}$ of $C^{\infty}$-vector bundles over $X^{\mathrm{reg}}$ defined by the reference metric as in (\ref{isomeddce}). Let $\mathrm{pr}_{\alpha} : \overline{E^{\vee}} \relbar\joinrel\twoheadrightarrow E_{- w_{\alpha}}$ be the surjection to each factor.
	\begin{enumerate}
		\item A $C^{\infty}$-map $\tilde{Q}^*_{\alpha} : \overline{V^{\vee}} \otimes \overline{C^{\infty}_{X^{\mathrm{reg}}} (k)} \longrightarrow E_{- w_{\alpha}}$ of locally free sheaves of $C^{\infty}_{X^{\mathrm{reg}}}$-modules is defined by the composition of $Q^*$ and the projection $\mathrm{pr}_{\alpha} : \overline{E^{\vee}} \relbar\joinrel\twoheadrightarrow E_{- w_{\alpha}}$.
		\item Let $j \ge \alpha$. A map $Q_{j \to \alpha}^* : \overline{V^{\vee}_{-w_j }} \otimes \overline{C^{\infty}_{X^{\mathrm{reg}}} (k)} \longrightarrow E_{- w_{\alpha}} $ of locally free sheaves of $C^{\infty}_{X^{\mathrm{reg}}}$-modules is defined as the composition of $Q^* |_{V^{\vee}_{-w_j }} : \overline{V^{\vee}_{-w_j }} \otimes \overline{C^{\infty}_{X^{\mathrm{reg}}} (k)} \longrightarrow \overline{E^{\vee}} $ and $\mathrm{pr}_{\alpha}$ (where the $C^{\infty}$-isomorphism (\ref{isomeddce}) is understood), \emph{i.e.}
		\begin{equation*}
		Q^*_{j \to \alpha} := \mathrm{pr}_{\alpha} \circ Q^* |_{V^{\vee}_{-w_j }}.
		\end{equation*}
		We have thus a diagram
	\begin{displaymath}
		\xymatrixcolsep{6pc}\xymatrix{ \overline{V^{\vee}_{-w_j}} \otimes \overline{C^{\infty}_{X^{\mathrm{reg}}} (k)}   \ar@{->}[r]^-{Q^* |_{V^{\vee}_{-w_j }}} \ar@{->}[dr]_{Q_{j \to \alpha}^*} & \overline{E^{\vee}} \ar@{->}[d]^{\mathrm{pr}_{\alpha}}  \\
		& E_{- w_{\alpha}}.
		}
	\end{displaymath}
	Recalling $E_{-w_{\alpha}} := E_{\le -w_{\alpha}} / E_{\le -w_{\alpha -1}}$, we find $Q^*_{j \to \alpha} = 0$ for $j < \alpha$ (see \emph{e.g.}~(\ref{exprhoint}) and (\ref{ordmlambda})).
	\end{enumerate}	
\end{definition}

\begin{rem}
	In what follows, we assume a convention in which $Q^*_{j \to \alpha}$ is defined on the whole $V^{\vee}_k$, with $Q^*_{j \to \alpha} |_{V^{\vee}_{-w_l}} =0$ if $l \neq j$.
\end{rem}

We define $\tilde{Q}_{\alpha}$ and $Q_{\alpha \to j}$ analogously, taking the fibrewise Hermitian conjugate of $\tilde{Q}^*_{\alpha}$ and $Q^*_{j \to \alpha}$, which is equivalent to taking the Hermitian metric dual with respect to the reference metrics defined in Section \ref{fsmetvb}.

The main technical result that we establish in this section is the following.

\begin{prop} \label{propexpfsz}
Suppose that we write
	\begin{equation} \label{defewt}
	e^{wt} := \mathrm{diag} (e^{ w_{\hat{1}} t} , \ldots , e^{w_{\hat{\nu}} t}),	
	\end{equation}
	with respect to the block decomposition (\ref{isomeddce}). Then the Fubini--Study metric $h_{\sigma_{t}} = Q^* e^{\zeta^* t + \zeta t} Q$ can be written as $h_{\sigma_{t}} = e^{w t} \hat{h}_{\sigma_{t}} e^{w t}$, where $\hat{h}_{\sigma_{t}}$ is defined as
\begin{equation*}
	\hat{h}_{\sigma_{t}} :=
	\begin{pmatrix}
		Q^*_{\hat{1} \to \hat{1}} Q_{\hat{1} \to \hat{1}} & \cdots & 0 \\
		\vdots & \ddots & \vdots \\
		0 & \cdots & Q^*_{{\hat{\nu}} \to \hat{\nu}} Q_{\hat{\nu} \to {\hat{\nu}}}
	\end{pmatrix}
	+ \bar{I}_{\sigma_{t}} ,
\end{equation*}
over $X^{\mathrm{reg}}$, in terms of the decomposition $\bigoplus_{\alpha =1}^{\hat{\nu}} E_{-w_{\alpha}}$; $\bar{I}_{\sigma_t}$ is a fibrewise Hermitian form of rank $r$, whose $(\alpha , \beta )^\mathrm{th}$ block is given by 
	\begin{equation*}
	\sum_{j> \alpha , \beta} e^{-( w_{\alpha} - w_j ) t} e^{-( w_{\beta} - w_j ) t} Q^*_{j \to \alpha} Q_{\beta \to j},
\end{equation*}
which decays exponentially as $t \to + \infty$ $($since $w_\alpha - w_j  >0$ for $j>\alpha$ by (\ref{ordlambda})$\,)$.
\end{prop}

Thus, we can ``separate'' the part $\hat{h}_{\sigma_{t}}$ of $h_{\sigma_{t}}$ which converges as $t \to + \infty$; moreover, we have an explicit description of the rate of convergence. Since they play an important role in what follows, we introduce the following definition.

\begin{definition} \label{defrenbgops}
	The path of Hermitian metrics $\{ \hat{h}_{\sigma_{t}} \}_{t \ge 0}$, where $\hat{h}_{\sigma_{t}}$ is a Hermitian metric defined on the regular locus $X^{\mathrm{reg}}$ by
	\begin{equation*}
		\hat{h}_{\sigma_{t}} :=  e^{- w t} h_{\sigma_{t}}  e^{- w t} \in \Gamma_{C^{\infty}_{X^{\mathrm{reg}}}} ( E^{\vee} \otimes  \overline{E^{\vee}}),
	\end{equation*}
	is called the \textbf{renormalised Bergman 1-parameter subgroup} associated to $\sigma_{t}$, where the $C^{\infty}$-isomorphism $\overline{E^{\vee}} \isom  \bigoplus_{\alpha =\hat{1}}^{\hat{\nu}} E_{-w_{\alpha}}$ and the notation (\ref{defewt}) are understood. We also call its limit in $\Gamma_{C^{\infty}_{X^{\mathrm{reg}}}} ( E^{\vee} \otimes  \overline{E^{\vee}})$
	\begin{equation*}
		\hat{h} :=   \lim_{t \to + \infty} e^{- w t} h_{\sigma_{t}}  e^{- w t} =  \bigoplus_{\alpha=\hat{1}}^{\hat{\nu}} Q^*_{{\alpha} \to \alpha} Q_{\alpha \to {\alpha}}
	\end{equation*}
	the \textbf{renormalised Quot-scheme limit} of $h_{\sigma_{t}}$, and each component of $\hat{h}$ is written as $\hat{h}_{\alpha} := Q^*_{{\alpha} \to \alpha} Q_{\alpha \to {\alpha}}$, so that $\hat{h} = \bigoplus_{\alpha=\hat{1}}^{\hat{\nu}} \hat{h}_{\alpha}$, where $\hat{h}_{\alpha} \in \Gamma_{C^{\infty}_{X^{\mathrm{reg}}}} ( E_{-w_{\alpha}} \otimes E_{-w_{\alpha}} )$, with $\overline{E^{\vee}} \isom  \bigoplus_{\alpha =\hat{1}}^{\hat{\nu}} E_{-w_{\alpha}}$ understood.
\end{definition}

Thus, ``renormalised'' in the above means that we apply the constant gauge transform over $X^{\mathrm{reg}}$ to $h_{\sigma_{t}}$, by ($-1$ times) the weight of $\zeta$ that acts on $\widehat{E} = \bigoplus_{\alpha=\hat{1}}^{\hat{\nu}} E_{-w_{\alpha}}$; we then get a component $\hat{h}_{\sigma_{t}}$ of $h_{\sigma_{t}}$ which converges as $t \to + \infty$.

It turns out that the renormalised Quot-scheme limit $\hat{h}$ defines a Hermitian metric on the vector bundle $E$ over a Zariski open subset $X^{\mathrm{reg}} = X^{\mathrm{reg}}(\zeta)$ of $X$. In general, however, $\hat{h}$ fails to be positive definite on the Zariski closed subset $X \setminus X^{\mathrm{reg}}$ of $X$, which fundamentally comes from the singularity of the sheaves $\mathcal{E}_{\le - w_i}$ and $\mathcal{E} / \mathcal{E}_{\le - w_i}$.

We now start the proof of Proposition \ref{propexpfsz} by proving the following result, which can be regarded as a $C^{\infty}$-analogue of the expansion (\ref{exprhoint}).

\begin{lem} \label{lemdcpqst}
	We have the following expansion over $X^{\mathrm{reg}}$
	\begin{equation*}
		Q^* \circ e^{\zeta t} = \bigoplus_{\alpha = \hat{1}}^{\hat{\nu}} e^{w_{\alpha} t} \left( Q^*_{{\alpha} \to \alpha} + \sum_{j > \alpha} e^{-(w_{\alpha} - w_j )t } Q^*_{j \to \alpha} \right) ,
	\end{equation*}
	where the direct sum is understood in terms of the $C^{\infty}$-isomorphism $\overline{E^{\vee}} \isom \bigoplus_{\alpha =\hat{1}}^{\hat{\nu}} E_{- w_{\alpha}}$.
\end{lem}

\begin{proof}
Definition \ref{defqij} immediately yields
\begin{equation*}
	\tilde{Q}^*_{\alpha} = Q^*_{{\alpha} \to \alpha} + \sum_{j > \alpha} Q^*_{j \to \alpha}.
\end{equation*}
With the $C^{\infty}$-isomorphism $\overline{E^{\vee}} \cong \bigoplus_{\alpha =\hat{1}}^{\hat{\nu}} E_{-w_{\alpha}}$, we can write
\begin{equation*}
	Q^* = \bigoplus_{\alpha =\hat{1}}^{\hat{\nu}} \tilde{Q}^*_{\alpha} = \bigoplus_{\alpha =\hat{1}}^{\hat{\nu}} \left( Q^*_{{\alpha} \to \alpha} + \sum_{j > \alpha} Q^*_{j \to \alpha} \right).
\end{equation*}
Recalling that $e^{\zeta t} = T^{-\zeta}$ acts on $V_{-w_j }$ as $e^{w_j t} = T^{-w_j}$, and hence
\begin{equation*}
	Q^*_{j \to \alpha} \circ e^{\zeta t} = e^{ w_j t} Q^*_{j \to \alpha},
\end{equation*}
we get the claim.
\end{proof}

\begin{proof}[Proof of Proposition \ref{propexpfsz}]
	We start by rewriting Lemma \ref{lemdcpqst} in a block-matrix notation, according to the decompositions $\bigoplus_{\alpha =\hat{1}}^{\hat{\nu}} E_{-w_{\alpha}}$ and $\bigoplus_{j=1}^{\nu} V_{-w_j }$.

%

We can write $Q^*$ in the following (block) row echelon form
\begin{equation*}
	Q^* =
	\begin{pmatrix}
		Q^*_{1 \to \hat{1}} & Q^*_{2 \to \hat{1}} & \cdots & \cdots & \cdots & \cdots & Q^*_{{\nu} \to \hat{1}} \\
		0  & \cdots & 0 & Q^*_{\hat{2} \to \hat{2}} & \cdots & \cdots & Q^*_{{\nu} \to \hat{2}} \\
		\vdots & \vdots & \vdots & \cdots & \ddots & \ddots & \vdots \\
		0 & 0 & 0 & \cdots & Q^*_{\hat{\nu} \to \hat{\nu}} & \cdots & Q^*_{{\nu} \to \hat{\nu}}
	\end{pmatrix},
\end{equation*}
where we recall $Q^*_{j \to \alpha} = 0$ for $j < \alpha$, and also $\hat{1} =1$ (as in Remark \ref{remcombfstfac}).

Thus, recalling that $\zeta$ is Hermitian and acts on $V_{-w_j }$ as $T^{-w_j}$ with $w_j \in \mathbb{Q}$, we compute
{\footnotesize  \renewcommand{\arraystretch}{0.2}
\begin{align*}
	Q^* e^{\zeta^* t} = Q^* \circ T^{- \zeta} &=\begin{pmatrix}
		T^{-w_{1}} Q^*_{1 \to \hat{1}} & T^{-w_{2}} Q^*_{2 \to \hat{1}} & \cdots & \cdots & \cdots & \cdots & T^{-w_{\nu}} Q^*_{{\nu} \to \hat{1}} \\
		0  & \cdots & 0 & T^{-w_{\hat{2}}} Q^*_{\hat{2} \to \hat{2}} & \cdots & \cdots & T^{-w_{\nu}} Q^*_{{\nu} \to \hat{2}} \\
		\vdots & \vdots & \vdots & \cdots & \ddots & \ddots & \vdots \\
		0 & 0 & 0 & \cdots & T^{-w_{\hat{\nu}}} Q^*_{\hat{\nu} \to \hat{\nu}} & \cdots & T^{-w_{\nu}} Q^*_{{\nu} \to \hat{\nu}}
	\end{pmatrix} \\
	&=\begin{pmatrix}
		T^{-w_{1}} & 0  & \cdots & 0 \\
		0 & T^{-w_{\hat{2}}} & \cdots & 0 \\
		\vdots & \vdots & \ddots & \vdots \\
		0 & 0 & \cdots & T^{-w_{\hat{\nu}}}
	\end{pmatrix}
	\\
	&\,\hspace{0.2cm}\times
	\begin{pmatrix}
		 Q^*_{1 \to \hat{1}} & T^{-( w_{2} - w_{1} )} Q^*_{2 \to \hat{1}} & \cdots & \cdots & \cdots & \cdots & T^{- (w_{\nu} - w_{1} )} Q^*_{{\nu} \to \hat{1}} \\
		0  & \cdots & 0 &  Q^*_{\hat{2} \to \hat{2}} & \cdots & \cdots & T^{- (w_{\nu} - w_{\hat{2}} ) } Q^*_{{\nu} \to \hat{2}} \\
		\vdots & \vdots & \vdots & \cdots & \ddots & \ddots & \vdots \\
		0 & 0 & 0 & \cdots &  Q^*_{\hat{\nu} \to \hat{\nu}} & \cdots & T^{-( w_{\nu} - w_{\hat{\nu}} )} Q^*_{{\nu} \to \hat{\nu}}
	\end{pmatrix}.
\end{align*}}
By taking the conjugate transpose, we obtain a similar formula for $e^{\zeta t} Q$. Since $\hat{1} =1$ (\emph{cf.}~Remark \ref{remcombfstfac}) and $T^{-w_j} = e^{w_j t}$, we thus get
\begin{align*}
	Q^* e^{\zeta^* t + \zeta t} Q &= e^{wt} \left(  
	\begin{pmatrix}
		Q^*_{\hat{1} \to \hat{1}} Q_{\hat{1} \to \hat{1}} & \cdots & 0 \\
		\vdots & \ddots & \vdots \\
		0 & \cdots & Q^*_{\hat{\nu} \to \hat{\nu}} Q_{\hat{\nu} \to \hat{\nu}}
	\end{pmatrix}
	+ \bar{I}_{\sigma_t}
	\right)e^{wt},
\end{align*}
where the $(\alpha, \beta )^\mathrm{th}$ block of $\bar{I}_{\sigma_t}$ in the above, with $\alpha , \beta \in \{ \hat{1} , \dots , \hat{\nu} \}$, can be written as
\begin{equation*}
	\sum_{j> \alpha , \beta} T^{-( w_j - w_{\alpha} )} T^{-( w_j - w_{\beta} )} Q^*_{j \to \alpha} Q_{\beta \to j} .
\end{equation*}
Recalling again $T = e^{-t}$, we get the result as claimed.
\end{proof}

\begin{lem} \label{pdxreghath}
	The metric $\hat{h}$ is strictly positive definite over $X^{\mathrm{reg}}$.
\end{lem}

\begin{proof}
	It suffices to show that the map $Q_{\alpha \to \alpha}^* : \overline{V^{\vee}_{-w_{\alpha} }} \otimes \overline{C^{\infty}_{X^{\mathrm{reg}}} (k)} \to E_{- w_{\alpha}} $ defined in Definition \ref{defqij} is surjective, since it implies that $Q^*_{\alpha \to \alpha} Q_{\alpha \to \alpha}$ is strictly nondegenerate.
	
	First consider the map $\rho_{\alpha \to \alpha} : V_{- w_{\alpha}} \otimes \mathcal{O}_X(-k) \to \mathcal{E}_{-w_{\alpha}}$, which is defined by the composition of the map $\rho |_{ V_{- w_{\alpha}}} : V_{- w_{\alpha}} \otimes \mathcal{O}_X(-k) \to \mathcal{E}_{\le -w_{\alpha}}$ and the projection $\mathcal{E}_{\le -w_{\alpha}} \surj \mathcal{E}_{\le -w_{\alpha}}/\mathcal{E}_{\le -w_{\alpha -1}} = \mathcal{E}_{-w_{\alpha}}$. Recalling the definition $\mathcal{E}_{-w_{\alpha}} = \rho (V_{\le - w_{\alpha}}) / \rho (V_{\le - w_{\alpha -1}})$, we see that $\rho_{\alpha \to \alpha}$ is surjective.
	
	Lemma \ref{cdiamq} and the diagram (\ref{cdiamqd}) imply that the surjectivity of $Q_{\alpha \to \alpha}^*$ follows from that of $\rho_{\alpha \to \alpha}$.
\end{proof}

While $\hat{h}$ may degenerate as it approaches $X \setminus X^{\mathrm{reg}}$ and hence has no uniform lower bound, it has a uniform upper bound given by $Q^*Q$, stated as follows. 

\begin{lem} \label{lemqjalpqunif}
$| Q^*_{j \to \alpha} |^2 = \tr (Q^*_{j \to \alpha} Q_{\alpha \to j})$ can be bounded by $|Q^*|^2$ for any $j =1, \dots , \nu$ and $\alpha = \hat{1} , \dots , \hat{\nu}$.
\end{lem}

\begin{proof}
This is straightforward, by noting $Q^*_{j \to \alpha} Q_{\beta \to m} =0$ if $j \neq m$ and so on.	
\end{proof}

Note also that $\sum_{j> \alpha , \beta} Q^*_{j \to \alpha} Q_{\beta \to j}$ can be bounded by $\sum_j |Q_{\alpha \to j}|^2$, by Cauchy--Schwarz inequality.

\section{Estimates for Donaldson's functional} \label{unifestimsect}

We establish several technical estimates, involving the renormalised Fubini--Study metrics and the renormalised Quot-scheme limit. Much of their proofs is elementary, but we shall also need to use the resolution of singularities for saturated subsheaves, as explained in \cite{Jacob,Sibley}.

\begin{prop}\label{sumprop}
 Let $\{ h_{\sigma_{t}} \}_{t \ge 0}$ be a path of Fubini--Study metrics emanating from $h_k$ generated by $\zeta \in \mathfrak{sl} (H^0(\mathcal{E}(k))^{\vee})$. Then there exists a constant $c_1 (h_k ) >0$ depending only on $h_k$ $($and $k)$, and a constant $c_2 (\zeta , h_k) >0$ depending on $h_k$ and $\zeta$ $($and $k)$, such that
	\begin{equation*}
		- c_1 (h_k )  \le \mathcal{M}^{\mathrm{Don}} (h_{\sigma_t} , h_k) + \sum_{i =1}^{\nu} 2 w_{i} \left( \deg (\mathcal{E}'_{-w_{i}}) - \mu(\mathcal{E}) \mathrm{rk} (\mathcal{E}'_{-w_{i}}) \right) t \le c_2 (\zeta , h_k) 
	\end{equation*}
	holds for all $t \ge 0$, where the graded pieces $\mathcal{E}'_{-w_{i}} := \mathcal{E}'_{\le -w_{i}} / \mathcal{E}'_{\le -w_{i-1}}$ are defined as in Lemma \ref{lemcdquot} by using the filtration $0 \neq \mathcal{E}'_{\le - w_{1}} \subset \cdots \subset \mathcal{E}'_{\le - w_{\nu}} = \mathcal{E}$.
\end{prop}

The entire section is devoted to the proof of the above proposition. As the proof is rather involved, the overall strategy is summarised in Section \ref{scstpfpsp}, which is preceded by some preliminary results in Section \ref{scprmcw}.

\subsection{Preliminaries} \label{scprmcw}

We start with some preliminary results that are necessary for the proof of Proposition \ref{sumprop}.

\begin{lem} \label{evibarprime}
	Writing $h_{\sigma_{t}} = e^{w t} \hat{h}_{\sigma_{t}} e^{w t}$ by using the renormalised Fubini--Study metric $\hat{h}_{\sigma_{t}}$ $($as in Definition \ref{defrenbgops}$\,)$, we have
	\begin{equation*}
		h_{\sigma_{t}}^{-1} \partial_t h_{\sigma_t} = e^{-w t} \left( 2 w + \hat{h}_{\sigma_t}^{-1} [w , \hat{h}_{\sigma_t} ] + \hat{h}_{\sigma_t}^{-1} \partial_t \hat{h}_{\sigma_t} \right) e^{w t} ,
	\end{equation*}
	where $w = \mathrm{diag}(w_{\hat{1}} , \dots , w_{\hat{\nu}})$ and $[\cdot ,\cdot]$ is the commutator.
\end{lem}

\begin{proof}
We can write $h_{\sigma_t} = e^{w t} \hat{h}_{\sigma_t} e^{w t}$ by using the renormalised Fubini--Study metric $\hat{h}_{\sigma_t}$, and note
\begin{equation*}
	\frac{d}{dt} h_{\sigma_t} = e^{w t} \left( w \hat{h}_{\sigma_t} + \frac{d}{dt} \hat{h}_{\sigma_t} +\hat{h}_{\sigma_t} w \right) e^{w t},
\end{equation*}
which immediately gives us the formula that we claimed.
\end{proof}

Note that the connection 1-form of $h_{\sigma_t}$, which is $h_{\sigma_t}^{-1} \partial h_{\sigma_t}$, can be evaluated as $e^{-w t} (\hat{h}_{\sigma_t}^{-1} \partial \hat{h}_{\sigma_t} )e^{w t}$. This immediately implies that, writing $F (\hat{h}_{\sigma_t})$ for ($\ai /2 \pi$ times) the curvature of $\hat{h}_{\sigma_t}$, we have the following pointwise equality over $X^{\mathrm{reg}}$:
\begin{equation*}
	F (h_{\sigma_t}) = e^{-w t} F (\hat{h}_{\sigma_t}) e^{w t}.
\end{equation*}

We thus get the following result.

\begin{lem}\label{lm1}
	Writing $\hat{h}_{\sigma_t}$ for the renormalised Fubini--Study metric, we have
	\begin{align*}
		\mathcal{M}^{\mathrm{Don}}(h_{\sigma_t } , h_k) = &- \int_0^t d \underline{t} \int_{X^{\mathrm{reg}}} \tr \left( (2 w + \hat{h}_{\sigma_{\underline{t}}}^{-1} [w , \hat{h}_{\sigma_{\underline{t}}} ] )(F(\hat{h}_{\sigma_{\underline{t}}}) - \mu(E) \mathrm{Id}_E) \right) \frac{\omega^n}{n!} \\
		&+ \mathcal{M}^{\mathrm{Don}}(\hat{h}_{\sigma_t} , h_k).
	\end{align*}
\end{lem}
\begin{rem}
 Note that by the cocycle property (\ref{cocyclemdon}), the same result holds if we replace the metric $h_k$ by any Hermitian metric on $E$.
\end{rem}

\begin{rem} \label{remmdsign} 
	An important remark about the sign is in order. We considered the renormalised Bergman 1-PS as a Hermitian metric on $E$ by means of the $C^{\infty}$-isomorphism $\overline{E^{\vee}} \isom \bigoplus_{\alpha = \hat{1}}^{\hat{\nu}} E_{- w_{\alpha}}$ which involves the metric dual. In other words, the (renormalised) Bergman 1-PS is a priori a Hermitian metric on $E^{\vee}$, which is identified with the one on $E$ by means of the metric dual. Although $h_{\sigma_t}$ is a Hermitian metric on $E$, when we write it as $h_{\sigma_t} = e^{wt} \hat{h}_{\sigma_t} e^{wt}$ as we do here, it is tacitly assumed that $E$ is identified with $\overline{E^{\vee}}$ with the metric duality isomorphism as in the diagram (\ref{cdiamqd}).
	
	With this metric duality isomorphism understood, we should consider $h_{\sigma_t} = e^{wt} \hat{h}_{\sigma_t} e^{wt}$ as a Hermitian metric on $E^{\vee}$, and hence we should consider the Donaldson functional for $E^{\vee}$, which is nothing but $-1$ times the usual definition. Thus, given $h_{\sigma_t} \in \mathcal{H}_{\infty}(\mathcal{E})$, we apply the metric duality isomorphism $\mathcal{H}_{\infty}(\mathcal{E}) \isom \mathcal{H}_{\infty}(\mathcal{E}^{\vee})$, $h_{\sigma_t} \mapsto e^{wt} \hat{h}_{\sigma_t}e^{wt} =: h^{\vee}_{\sigma_t}$ to get
	\begin{align*}
		\mathcal{M}^{\mathrm{Don}}(e^{wt} \hat{h}_{\sigma_t}e^{wt} , h_k) &= \mathcal{M}^{\mathrm{Don}}(h^{\vee}_{\sigma_t}, h_k) \\
		&=-\int_0^t d \underline{t} \int_X \tr((h^{\vee}_{\sigma_{\underline{t}}})^{-1} \partial_{\underline{t}} h^{\vee}_{\sigma_{\underline{t}}} (\Lambda_{\omega} F(h^{\vee}_{\sigma_{\underline{t}}}) - \mu(\mathcal{E}) \Id_E)) \frac{\omega^n}{n!},
	\end{align*}
	in which $F(h^{\vee}_{\sigma_t})$ is meant to be the curvature of $\mathcal{E}$, defined by composing $h^{\vee}_{\sigma_t} = e^{wt} \hat{h}_{\sigma_t}e^{wt}$ with the metric duality isomorphism. We shall not make distinctions between $h_{\sigma_t}$ and $h^{\vee}_{\sigma_t} = e^{wt} \hat{h}_{\sigma_t}e^{wt}$ in what follows to avoid further complication in the notational convention, but it is important to note that the sign in Lemma \ref{lm1} is consistent with the metric duality isomorphism that is implicit in this convention. Finally, it is perhaps worth pointing out that $\mathcal{E}$ admits a Hermitian--Einstein metric if and only if $\mathcal{E}^{\vee}$ does, and that subbundles of $\mathcal{E}$ correspond to the quotient bundles of $\mathcal{E}^{\vee}$.
\end{rem}

Lemma \ref{lm1} indicates an interesting role played by the renormalised Fubini--Study metric $\hat{h}_{\sigma_t}$. Indeed, Proposition \ref{sumprop} will be proved by bounding each terms in Lemma \ref{lm1} individually as we discuss in Section \ref{scstpfpsp}. Before we start evaluating these terms, we recall the following important theorem that will also be used in the proof of Proposition \ref{sumprop}.

\begin{thm}[Chern--Weil formula \cite{Sibley,SimpYM}] \label{thmchnwl}
	Let $\mathcal{S}$ be a saturated subsheaf of a holomorphic vector bundle $\mathcal{E}$ endowed with a Hermitian metric $h$. Then we have
	\begin{equation*}
		\int_X \tr (F(h) |_{\mathcal{S}}) \frac{\omega^{n-1}}{(n-1)!} = \deg ( \mathcal{S}) + \int_X |\mathrm{II}(h) |^2 \frac{\omega^{n-1}}{(n-1)!},
	\end{equation*}
	where $\mathrm{II}(h)$ is the second fundamental form defined by $h$ associated to $\mathcal{S} \subset \mathcal{E}$.
\end{thm}
In the above, $F(h) |_{\mathcal{S}}$ is meant to be the composition of $F(h)$ with the orthogonal projection $\pi : \mathcal{E} \longrightarrow \mathcal{S}$ defined by $h$. We shall use the following observation concerning this theorem. Suppose that $\mathcal{S}$ is a subsheaf of $\mathcal{E}$ that is not necessarily saturated, and let $\mathcal{S}'$ be its saturation in $\mathcal{E}$. Since $\mathcal{S}$ and $\mathcal{S}'$ agree on a Zariski open subset in $X$, we find
\begin{equation*}
	\int_X \tr (F(h) |_{\mathcal{S}}) \frac{\omega^{n-1}}{(n-1)!} = \int_X \tr (F(h) |_{\mathcal{S}'}) \frac{\omega^{n-1}}{(n-1)!}.
\end{equation*}
Thus, we get
\begin{equation} \label{chernweilnsat}
		\int_X \tr (F(h) |_{\mathcal{S}}) \frac{\omega^{n-1}}{(n-1)!} = \deg ( \mathcal{S}') + \int_X |\mathrm{II}'(h) |^2 \frac{\omega^{n-1}}{(n-1)!},
\end{equation}
where the second fundamental form that appears on the right hand side is defined with respect to $h$ and the saturation $\mathcal{S}'$ of $\mathcal{S}$.

\subsection{Strategy of the proof of Proposition \ref{sumprop}} \label{scstpfpsp}

We decomposed $\mathcal{M}^{\mathrm{Don}}(h_{\sigma_t} , h_k)$ in two terms as in Lemma \ref{lm1}, and we bound each term individually to get Proposition \ref{sumprop}. The first term can be evaluated as follows.

\begin{prop} \label{lemdtmd}
There exists a constant $c_3 (h_k) >0$ which depends on the reference metric $h_k$ $($and $k)$ such that
\begin{align*}
	&- \sum_{i =1}^{\nu} 2 w_{i} \left( \deg (\mathcal{E}'_{-w_{i}}) - \mu(\mathcal{E}) \mathrm{rk} (\mathcal{E}'_{-w_{i}}) \right) t - c_3 (h_k)\\
	&\le- \int_0^t d \underline{t} \int_{X^{\mathrm{reg}}} \tr \left( (2 w + \hat{h}_{\sigma_{\underline{t}}}^{-1} [w , \hat{h}_{\sigma_{\underline{t}}} ] )(\Lambda_{\omega} F(\hat{h}_{\sigma_{\underline{t}}}) - \mu(E) \mathrm{Id}_E) \right) \frac{\omega^n}{n!} \\
	&\le - \sum_{i =1}^{\nu} 2 w_{i} \left( \deg (\mathcal{E}'_{-w_{i}}) - \mu(\mathcal{E}) \mathrm{rk} (\mathcal{E}'_{-w_{i}}) \right) t ,
\end{align*}
uniformly for all $t \ge 0$ and $\zeta \in \mathfrak{sl} (H^0 (\mathcal{E} (k))^{\vee})$, where $\mathcal{E}'_{-w_{i}} := \mathcal{E}'_{\le -w_{i}} / \mathcal{E}'_{\le -w_{i-1}}$ $($cf.~Lemma \ref{lemcdquot}$\,)$.
\end{prop}

The proof of this proposition will be given in Section \ref{lowersec}. Together with the following result that we prove in Section \ref{uppersec}, which bounds the second term in Lemma \ref{lm1}, we get the desired proof of Proposition \ref{sumprop}.


\begin{prop}\label{sumprophat}
 Let $\{ \hat{h}_{\sigma_t} \}_{t \ge 0}$ be a renormalised Bergman 1-PS associated to $\zeta \in \mathfrak{sl}(H^0(\mathcal{E}(k))^{\vee})$. Then there exists a constant $c_4 (h_k) >0$ which depends only on $h_k$ $($and $k)$ and a constant $c_5 (\zeta, h_k)>0$ depending on $h_k$ and $\zeta$ $($and $k)$, such that
	\begin{equation*}
		- c_4 (h_k ) \le  \mathcal{M}^{\mathrm{Don}} (\hat{h}_{\sigma_t}, h_k) \le c_5 (\zeta ,h_k)
	\end{equation*}
	for all $t \ge 0$.
\end{prop}

The rest of the section is devoted to the proof of the above results.

\subsection{Proof of Proposition \ref{lemdtmd}} \label{lowersec}

The proof of Proposition \ref{lemdtmd} relies on the resolution of singularities, which we recall below. Recall that we have a sequence of saturated subsheaves $\mathcal{E}'_{\le - w_{\hat{1}}} \subset \cdots \subset \mathcal{E}'_{\le - w_{\hat{\nu}}}$ of $\mathcal{E}$. We apply the regularisation (or the resolution of singularities) for the saturated sheaves, as in Jacob \cite{Jacob} and Sibley \cite{Sibley}, to reduce to the case of holomorphic subbundles. We closely follow the argument by Jacob \cite[Sections 3 and 4]{Jacob}. After a finite sequence of blowups $\pi : \tilde{X} \to {X}$, we have a holomorphic subbundle $\tilde{\mathcal{E}}'_{\le - w_{\alpha}}$ of $\pi^* \mathcal{E}$ (\emph{i.e.}~$\tilde{\mathcal{E}}'_{\le - w_{\alpha}}$ and $\pi^* \mathcal{E} / \tilde{\mathcal{E}}'_{\le - w_{\alpha}}$ are both locally free) over $\tilde{X}$, such that $\pi_* \tilde{\mathcal{E}}'_{\le - w_{\alpha}} = \mathcal{E}'_{\le - w_{\alpha}}$ (see \cite[Proposition 3]{Jacob} and \cite[Proposition 4.3]{Sibley}). We thus get a filtration of $\pi^* \mathcal{E}$ by holomorphic subbundles over $\tilde{X}$ as 
\begin{equation*}
	0 \neq \tilde{\mathcal{E}}'_{\le - w_{\hat{1}}} \subset \cdots \subset \tilde{\mathcal{E}}'_{\le - w_{\hat{\nu}}} = \pi^* \mathcal{E},
\end{equation*}
and hence a $C^{\infty}$-isomorphism
\begin{equation} \label{dfdccstggt}
	\pi^*E \cong \bigoplus_{\alpha =\hat{1}}^{\hat{\nu}} \tilde{E}'_{\le -w_{\alpha}} / \tilde{E}'_{\le -w_{\alpha-1}}
\end{equation}
of $C^{\infty}$-complex vector bundles over $\tilde{X}$. The point of this isomorphism is that $\pi^* E$ admits a constant gauge transformation by $e^{wt}$ in the notation of (\ref{defewt}), \textit{globally} over $\tilde{X}$.

We now consider the pullback $\pi^* h_{\sigma_t}$ of the Hermitian metric $h_{\sigma_t}$. This defines a metric on 
\[\tilde{E}'_{\le - w_{\alpha}} \oplus \pi^*E / \tilde{E}'_{\le - w_{\alpha}} \cong \pi^* E \quad \text{(as $C^{\infty}$-vector bundles)},\]
as in \cite[Section 2]{Jacob}. As in \cite[Proposition 4]{Jacob}, we construct Hermitian metrics $\tilde{h}_{\sigma_t} |_{\tilde{E}'_{\le - w_{\alpha}}}$ on $\tilde{E}'_{\le - w_{\alpha}}$ and $\tilde{h}_{\sigma_t} |_{\pi^*E / \tilde{E}'_{\le - w_{\alpha}}}$ on $\pi^*E / \tilde{E}'_{\le - w_{\alpha}}$, by removing the factors that vanish (or blow up) on the exceptional divisor of $\pi$. Now we consider the constant gauge transformation by $e^{wt}$ on $\pi^* E$, and hence on $\pi^* h_{\sigma_t}$, given by the $C^{\infty}$-isomorphism (\ref{dfdccstggt}). Since this gauge transformation is an overall constant scaling on each summand in (\ref{dfdccstggt}), the Hermitian metrics
\begin{itemize}
	\item $e^{-wt} \tilde{h}_{\sigma_t} e^{-wt} |_{\tilde{E}'_{\le - w_{\alpha}}}$ on $\tilde{E}'_{\le - w_{\alpha}}$,
	\item $e^{-wt} \tilde{h}_{\sigma_t} e^{-wt} |_{\pi^*E / \tilde{E}'_{\le - w_{\alpha}}}$ on $\pi^*E / \tilde{E}'_{\le - w_{\alpha}}$,
\end{itemize}
are both well defined globally over $\tilde{X}$.

We now repeat the same operation for the renormalised metric $h^{\mathrm{ren}}_{\sigma_t}:=\hat{h}_{\sigma_t}= e^{-wt} h_{\sigma_t} e^{-wt}$, which is only well defined on $X^{\mathrm{reg}}$, to get Hermitian metrics
\begin{itemize}
	\item $\tilde{h}^{\mathrm{ren}}_{\sigma_t}  |_{\tilde{E}'_{\le - w_{\alpha}}}$ on $\tilde{E}'_{\le - w_{\alpha}}$,
	\item $\tilde{h}^{\mathrm{ren}}_{\sigma_t} |_{\pi^*E / \tilde{E}'_{\le - w_{\alpha}}}$ on $\pi^*E / \tilde{E}'_{\le - w_{\alpha}}$,
\end{itemize}
defined over $\pi^{-1} (X^{\mathrm{reg}})$, which is Zariski open in $\tilde{X}$; following the notation as above $\tilde{h}^{\mathrm{ren}}_{\sigma_t}$ perhaps should be written as $\tilde{\hat{h}}_{\sigma_t}$, but we prefer not to use it as it is harder to read. Although these metrics are defined only over $\pi^{-1} (X^{\mathrm{reg}})$, they clearly agree with $e^{-wt} \tilde{h}_{\sigma_t} e^{-wt} |_{\tilde{E}'_{\le - w_{\alpha}}}$ or $e^{-wt} \tilde{h}_{\sigma_t} e^{-wt} |_{\pi^*E / \tilde{E}'_{\le - w_{\alpha}}}$ over $\pi^{-1} (X^{\mathrm{reg}})$.

We summarise the conclusion of the above discussion in the following lemma.

\begin{lem} \label{lmrhmexgltx}
	With the notation as above, the Hermitian metric $\tilde{h}^{\mathrm{ren}}_{\sigma_t}$ extends as a smooth Hermitian metric on $\pi^*E$, globally over $\tilde{X}$.
\end{lem}

We prove in Lemma \ref{eqcwfxrg} that an analogue of the Chern--Weil formula applies to the above metric $\tilde{h}^{\mathrm{ren}}_{\sigma_t}$. We start by recalling that the $L^2$-norm of the second fundamental form remains unchanged under the regularisation process.

\begin{lem}[Jacob \cite{Jacob}] \label{fact1}
Suppose that we write $\mathrm{II}'_\alpha(\tilde{h}^{\mathrm{ren}}_{\sigma_t})$ for the second fundamental form defined by the above metric $\tilde{h}^{\mathrm{ren}}_{\sigma_t}$ on $\pi^*E$ associated to  $\tilde{\mathcal{E}}'_{\le - w_{\alpha}}\subset \pi^* \mathcal{E}$.
 Then we have
 \begin{equation} \label{int2ndFunda}
 \int_{\pi^{-1}(X^{\mathrm{reg}})}\vert \mathrm{II}'_\alpha(\tilde{h}^{\mathrm{ren}}_{\sigma_t})\vert^2\frac{\pi^*\omega^{n-1}}{(n-1)!}=\int_{X^{\mathrm{reg}}} \vert \mathrm{II}'_\alpha (\hat{h}_{\sigma_t})\vert^2\frac{\omega^{n-1}}{(n-1)!}<+\infty.
 \end{equation}
\end{lem}

\begin{proof}[Sketch of the proof]
	We only provide a sketch of the proof as the details can be found in \cite[Proofs of Propositions 1 and 4, Lemma 1]{Jacob}.
	
	Let us write $p$ (respectively $\tilde{p}$) for the orthogonal projection to $\mathcal{E}'_{\le - w_{\alpha}}$ with respect to $\hat{h}_{\sigma_t}$ (respectively 
	to $\tilde{\mathcal{E}}'_{\le - w_{\alpha}}$ and with respect to $\tilde{h}^{\mathrm{ren}}_{\sigma_t}$). Arguing as in \cite[Proof of Proposition 1, Equation (2.6)]{Jacob}, we find by expressing the second fundamental form in terms of the curvatures of $E$ and $\mathcal{E}'_{\le - w_{\alpha}}$,
	\begin{align}
	\int_{X^{\mathrm{reg}}} \vert \mathrm{II}'_\alpha (\hat{h}_{\sigma_t})\vert^2\frac{\omega^{n-1}}{(n-1)!} =& \int_{\pi^{-1} (X^{\mathrm{reg}})} \pi^* \vert \mathrm{II}'_\alpha (\hat{h}_{\sigma_t})\vert^2 \wedge \frac{\pi^*\omega^{n-1}}{(n-1)!} \nonumber\\
	=&\int_{\pi^{-1} (X^{\mathrm{reg}})} \pi^*\tr(p\circ {F_{\hat{h}_{\sigma_t}}})\wedge \frac{\pi^*\omega^{n-1}}{(n-1)!} \nonumber \\
	&\quad -\int_{\pi^{-1} (X^{\mathrm{reg}})} \pi^*\tr(F_{\hat{h}_{\sigma_t}})\wedge \frac{\pi^*\omega^{n-1}}{(n-1)!}.\label{SdForm1}
	\end{align}
	If $z$ defines local holomorphic coordinates and $\{z=0\}$ corresponds locally to the exceptional divisor associated to a blow up of the regularisation process, one can express locally the degeneracy of the involved metrics on the bundles. This shows two facts. Firstly, one can identify the two projections $p$ and $\tilde{p}$ via the pull-back map. Secondly, 
	one can express the curvature of the metrics and obtain the following equality in the sense of currents 
	$$\pi^* \tr(F_{\hat{h}_{\sigma_t}})=\tr(F_{\tilde{h}^{\mathrm{ren}}_{\sigma_t}})+\sum_{i} a_i \ddbar \log \vert z\vert^2,$$
	for some non-negative integers $a_i$. \\
	Note that the  term $\int_{\tilde{X}} \sum_{i} a_i \ddbar \log \vert z\vert^2 \wedge \pi^*\omega^{n-1}$ vanishes. Consequently we get from \eqref{SdForm1}, 
	\begin{align*}
	\int_{X^{\mathrm{reg}}} \vert \mathrm{II}'_\alpha (\hat{h}_{\sigma_t})\vert^2\frac{\omega^{n-1}}{(n-1)!} 
	=& \int_{\pi^{-1}(X^{\mathrm{reg}})} \tr(\tilde{p}\circ {F_{\tilde{h}^{\mathrm{ren}}_{\sigma_t}}})\wedge \frac{\pi^*\omega^{n-1}}{(n-1)!}\\
	&-\int_{\pi^{-1}(X^{\mathrm{reg}})} \left(\tr(F_{\tilde{h}^{\mathrm{ren}}_{\sigma_t}})+\sum_{i} a_i \ddbar \log \vert z\vert^2\right) \wedge \frac{\pi^*\omega^{n-1}}{(n-1)!}\\
		=&\int_{\pi^{-1}(X^{\mathrm{reg}})} \tr(\tilde{p}\circ {F_{\tilde{h}^{\mathrm{ren}}_{\sigma_t}}})\wedge \frac{\pi^*\omega^{n-1}}{(n-1)!} -\int_{\pi^{-1}(X^{\mathrm{reg}})} \tr(F_{\tilde{h}^{\mathrm{ren}}_{\sigma_t}}) \wedge \frac{\pi^*\omega^{n-1}}{(n-1)!}\\
	=&\int_{\pi^{-1} (X^{\mathrm{reg}})}\vert \mathrm{II}'_\alpha(\tilde{h}^{\mathrm{ren}}_{\sigma_t})\vert^2\frac{\pi^*\omega^{n-1}}{(n-1)!}.
	\end{align*}
	By Lemma \ref{lmrhmexgltx} $\tilde{h}^{\mathrm{ren}}_{\sigma_t}$ extends as a smooth Hermitian metric globally over $\tilde{X}$ which is compact, we get the desired upper bound for the inequality in \eqref{int2ndFunda}.
\end{proof}

\begin{lem} \label{eqcwfxrg}
With the notation as above, we have
 \begin{equation*}
		\int_{X^{\mathrm{reg}}} \tr (F(\hat{h}_{\sigma_t} ) |_{E'_{\le - w_{\alpha}}}) \frac{\omega^{n-1}}{(n-1)!} = \deg (\mathcal{E}'_{\le - w_{\alpha}}) + \int_{X^{\mathrm{reg}}} |\mathrm{II}_{\alpha}' (\hat{h}_{\sigma_t}) |^2 \frac{\omega^{n-1}}{(n-1)!}.
	\end{equation*}
\end{lem}

The point of the above equation is that the analogue of the Chern--Weil formula (\ref{chernweilnsat}) holds for the metric $\hat{h}_{\sigma_t}$ that is well defined a priori only over $X^{\mathrm{reg}}$.

\begin{proof}
	This is a consequence of Lemma \ref{lmrhmexgltx}: we apply the Chern--Weil formula (Theorem \ref{thmchnwl}) for $\tilde{h}^{\mathrm{ren}}_{\sigma_t}$ on $\pi^* E$, which is globally defined over $\tilde{X}$, and get the claimed result by applying also Lemma \ref{fact1} and \cite[Lemma 2]{Jacob} (which states that the degree remains unchanged under the regularisation, by using a degenerate K\"ahler metric $\pi^* \omega$ on $\tilde{X}$).
\end{proof}

\begin{proof}[Proof of Proposition \ref{lemdtmd}]
	Recall that we have the $C^{\infty}$-isomorphism $E \cong \bigoplus_{\alpha=\hat{1}}^{\hat{\nu}} E_{-w_{\alpha}}$ over $X^{\mathrm{reg}}$. We compose this with another $C^{\infty}$-isomorphism so that the $E_{-w_{\alpha}}$ are pairwise orthogonal with respect to $\hat{h}_{\sigma_t}$. Thus,
\[
		\tr \left( w F(\hat{h}_{\sigma_t} ) \right) = \sum_{\alpha=\hat{1}}^{\hat{\nu}} w_{\alpha} \tr (F(\hat{h}_{\sigma_t} ) |_{E_{-w_{\alpha}}})
		=\sum_{\alpha=\hat{1}}^{\hat{\nu}} w_{\alpha} \left( \tr (F(\hat{h}_{\sigma_t} ) |_{E_{\le - w_{\alpha}}}) - \tr (F(\hat{h}_{\sigma_t} ) |_{E_{\le - w_{\alpha -1}}}) \right).
\]
	Recall now that each sheaf $\mathcal{E}_{\le -w_{\alpha}}$ agrees with its saturation $\mathcal{E}'_{\le -w_{\alpha}}$ over $X^{\mathrm{reg}}$, as they are both holomorphic subbundles over $X^{\mathrm{reg}}$. Lemma \ref{eqcwfxrg} implies
	\begin{equation*}
		\int_{X^{\mathrm{reg}}} \tr (F(\hat{h}_{\sigma_t} ) |_{E_{\le - w_{\alpha}}}) \frac{\omega^{n-1}}{(n-1)!} = \deg (\mathcal{E}'_{\le - w_{\alpha}}) + \int_{X^{\mathrm{reg}}} |\mathrm{II}_{\alpha}' (\hat{h}_{\sigma_t}) |^2 \frac{\omega^{n-1}}{(n-1)!},
	\end{equation*}
	where $\mathrm{II}_{\alpha}' (\hat{h}_{\sigma_t})$ is the second fundamental form associated to the \textit{saturated} subsheaf $\mathcal{E}'_{\le - w_{\alpha}}$ defined with respect to $\hat{h}_{\sigma_t}$. Thus,
\[
-\int_{X^{\mathrm{reg}}} \tr \left( w F(\hat{h}_{\sigma_t} ) \right) \frac{\omega^{n-1}}{(n-1)!} =- \sum_{\alpha=\hat{1}}^{\hat{\nu}} w_{\alpha} \deg (\mathcal{E}'_{-w_{\alpha}}) - \sum_{\alpha=\hat{1}}^{\hat{\nu}} (w_{\alpha -1} - w_{\alpha }) \int_X |\mathrm{II}_{\alpha -1}' (\hat{h}_{\sigma_t}) |^2 \frac{\omega^{n-1}}{(n-1)!}.
\]
	
In Lemma \ref{lembdscfdmfm} that follows, we prove the existence of a positive constant $c_3 (h_k)$ which depends only on the reference metric $h_k$ (and $k$) such that
	\begin{equation} \label{bdscfdmfm}
		0 \le \sum_{\alpha=\hat{1}}^{\hat{\nu}} (w_{\alpha -1} - w_{\alpha }) \int_0^{t} d\underline{t} \int_X |\mathrm{II}_{\alpha -1}' (\hat{h}_{\sigma_{\underline{t}}}) |^2 \frac{\omega^{n-1}}{(n-1)!} \le c_3 (h_k),
	\end{equation}
	uniformly for all $t>0$  and $\zeta \in \mathfrak{sl} (H^0 (\mathcal{E}(k))^{\vee})$.

Similarly, we find
\begin{equation*}
	\int_{X^{\mathrm{reg}}} \tr(w \cdot \mathrm{Id}_E)=\sum_{\alpha=\hat{1}}^{\hat{\nu}} w_\alpha \rk(E_\alpha)=\sum_{\alpha =\hat{1}}^{\hat{\nu}} w_{\alpha} \rk(\mathcal{E}'_{\alpha}),
\end{equation*}
by noting that saturation does not change the rank.

Finally, since $\bigoplus_{\alpha=1}^{\hat{\nu}} E_{-w_{\alpha}}$ is assumed to be orthogonal with respect to $\hat{h}_{\sigma_t}$, by our choice of $C^{\infty}$-isomorphism, $w$ must commute with $\hat{h}_{\sigma_t}$, \emph{i.e.}~$[w , \hat{h}_{\sigma_t} ] =0$. Recalling (\ref{bdscfdmfm}), we thus get
\begin{align*}
	&-\sum_{\alpha =\hat{1}}^{\hat{\nu}} 2 w_{\alpha} \left( \deg (\mathcal{E}'_{-w_{\alpha}}) - \mu(\mathcal{E}) \mathrm{rk} (\mathcal{E}'_{-w_{\alpha}}) \right) t - c_3 (h_k) \\
	&\le - \int_0^t d\underline{t} \int_{X^{\mathrm{reg}}} \tr \left( (2 w + \hat{h}_{\sigma_{\underline{t}}}^{-1} [w , \hat{h}_{\sigma_{\underline{t}}} ] )(\Lambda_{\omega} F(\hat{h}_{\sigma_{\underline{t}}}) - \mu(E) \mathrm{Id}_E) \right) \frac{\omega^n}{n!} \\
	&\le - \sum_{\alpha =\hat{1}}^{\hat{\nu}} 2 w_{\alpha} \left( \deg (\mathcal{E}'_{-w_{\alpha}}) - \mu(\mathcal{E}) \mathrm{rk} (\mathcal{E}'_{-w_{\alpha}}) \right) t .
\end{align*}
Recall from Lemma \ref{lemcdquot} that each $\mathcal{E}'_{-w_i}$ is torsion-free or zero for each $i=1 , \dots , \nu$, and that $\mathrm{rk}(\mathcal{E}'_{-w_i})>0$ if and only if $i \in \{ \hat{1} , \dots , \hat{\nu}\}$. This means that $\mathcal{E}'_{-w_i}$ is simply zero for all $i \notin \{ \hat{1} , \dots , \hat{\nu}\}$, and hence $\deg(\mathcal{E}'_{-w_i})=0$ if $i \notin \{ \hat{1} , \dots , \hat{\nu}\}$. We thus get
\begin{equation*}
	\sum_{\alpha =\hat{1}}^{\hat{\nu}} 2 w_{\alpha} \left( \deg (\mathcal{E}'_{-w_{\alpha}}) - \mu(\mathcal{E}) \mathrm{rk} (\mathcal{E}'_{-w_{\alpha}}) \right) = \sum_{i =1}^{\nu} 2 w_{i} \left( \deg (\mathcal{E}'_{-w_{i}}) - \mu(\mathcal{E}) \mathrm{rk} (\mathcal{E}'_{-w_{i}}) \right),
\end{equation*}
as required. Thus, granted (\ref{bdscfdmfm}) that we prove below, we complete the proof of Proposition \ref{lemdtmd}.
\end{proof}

We now prove (\ref{bdscfdmfm}).

\begin{lem} \label{lembdscfdmfm}
	Using the notation above, there exists a constant $c_3 (h_k) >0$ which depends only the reference metric $h_k$ $($and $k)$ such that
	\begin{equation*}
		0 \le \sum_{\alpha=\hat{1}}^{\hat{\nu}} (w_{\alpha-1} - w_{\alpha}) \int_0^t d\underline{t} \int_{X^{\mathrm{reg}}} |\mathrm{II}_{\alpha -1}' (\hat{h}_{\sigma_{\underline{t}}}) |^2 \frac{\omega^{n-1}}{(n-1)!} \le c_3 (h_k) ,
	\end{equation*}
	uniformly for all $t \ge 0$ and $\zeta \in \mathfrak{sl} (H^0 (\mathcal{E}(k))^{\vee})$ $($with $\Vert \zeta \Vert_{\mathrm{op}} \le 1$, as in Remark \ref{condopnm}$\,)$.
\end{lem}

\begin{proof}
	First note that $w_{\alpha} - w_{\alpha -1} < 0$ by (\ref{ordlambda}), which establishes the first inequality. Note also that the integral $\int_0^t d\underline{t} \int_X |\mathrm{II}_{\alpha -1}' (\hat{h}_{\sigma_{\underline{t}}}) |^2 \frac{\omega^{n-1}}{(n-1)!}$ is monotonically increasing in $t$, and hence it suffices to establish the bound for $t= + \infty$.
	
	We start by proving a weaker estimate
	\begin{equation} \label{weakestzetaii}
		 \sum_{\alpha=1}^{\hat{\nu}} (w_{\alpha} - w_{\alpha -1}) \int_0^{+ \infty} dt \int_{X^{\mathrm{reg}}} |\mathrm{II}_{\alpha -1}' (\hat{h}_{\sigma_t}) |^2 \frac{\omega^{n-1}}{(n-1)!} \le  c_3 (h_k,  \zeta)
	\end{equation}
	where the constant also depends on $\zeta \in \mathfrak{sl} (H^0 (\mathcal{E} (k))^{\vee})$. It suffices to show
	\begin{equation} \label{wweakestzetaii}
		\int_0^{+ \infty} dt \int_{X^{\mathrm{reg}}} |\mathrm{II}_{\alpha -1}' (\hat{h}_{\sigma_t}) |^2 \frac{\omega^{n-1}}{(n-1)!} < + \infty
	\end{equation}
	for each $\alpha =\hat{1} , \dots , \hat{\nu}$. Note that, by the monotone convergence theorem, we may exchange the order of the integrals.
	
	\medskip
	Step 1: the case when all the $\mathcal{E}_{\le - w_{i}}$'s are holomorphic subbundles of $\mathcal{E}$. In this case, $X^{\mathrm{reg}} = X$, and it is easy to see that $|\mathrm{II}_{\alpha -1}' (\hat{h}_{\sigma_t}) |^2$ decays exponentially as $t \to + \infty$ with decay rate at least $e^{-|w_{\alpha} - w_{\alpha -1}|t}$, by Proposition \ref{propexpfsz}. Recalling Lemmas \ref{pdxreghath} and \ref{lemqjalpqunif}, this immediately establishes (\ref{weakestzetaii}).
	
	\medskip
	Step 2: the case when all the $\mathcal{E}_{< - w_{i}}$'s are saturated, \emph{i.e.}~$\mathcal{E}_{\le - w_i} = \mathcal{E}'_{\le - w_i}$. We apply the regularisation (or the resolution of singularities) for the saturated sheaves to reduce to the case of holomorphic subbundles, just as we did at the beginning of this section. After a finite sequence of blowups $\pi : \tilde{X} \to {X}$, we have a locally free subsheaf $\tilde{\mathcal{E}}_{\le - w_{\alpha}}$ of $\pi^* \mathcal{E}$ with a locally free quotient $\pi^* \mathcal{E} / \tilde{\mathcal{E}}_{\le - w_{\alpha}}$ over $\tilde{X}$, such that $\pi_* \tilde{\mathcal{E}}_{\le - w_{\alpha}} = \mathcal{E}_{\le - w_{\alpha}}$ (see \cite[Proposition 3]{Jacob} and \cite[Proposition 4.3]{Sibley}). Now, the pullback $\pi^*h^{\mathrm{ren}}_{\sigma_t}$  of the renormalised metric $h^{\mathrm{ren}}_{\sigma_t}:=\hat{h}_{\sigma_t}$ can be regarded as defining the induced metric on  $$\tilde{E}_{\le - w_{\alpha}} \oplus \pi^*E / \tilde{E}_{\le - w_{\alpha}} \cong \pi^* E$$ (as $C^{\infty}$-vector bundles), as in \cite[Section 2]{Jacob}. As we did at the beginning of this section and as in \cite[Proposition 4]{Jacob}, we construct Hermitian metrics
	\begin{itemize}
		\item $\tilde{h}^{\mathrm{ren}}_{\sigma_t} |_{\tilde{E}_{\le - w_{\alpha}}}$ on $\tilde{E}_{\le - w_{\alpha}}$ and
		\item $\tilde{h}^{\mathrm{ren}}_{\sigma_t}|_{\pi^*E / \tilde{E}_{\le - w_{\alpha}}}$ on $\pi^*E / \tilde{E}_{\le - w_{\alpha}}$
	\end{itemize}
	by removing the factors that vanish (or blow up) on the exceptional divisor of $\pi$.

	Recalling Lemma \ref{fact1}, we find the following: during the regularisation process for saturated sheaves, we obtain the metric $\tilde{h}^{\mathrm{ren}}$ without altering the $L^2$-norm of the second fundamental form. Note that we use a degenerate K\"ahler metric on $\tilde{X}$ in \eqref{int2ndFunda} but this does not matter for our main purpose.\\
	We now consider the limit where $t$ tends to $+ \infty$. Since the renormalised metric $\hat{h}$ as in Definition \ref{defrenbgops} may be degenerate on $X \setminus X^{\mathrm{reg}}$, $\tilde{h}^{\mathrm{ren}}$ may develop further degeneracy on the exceptional divisor of $\pi$, as $t \to + \infty$. Recall also that Lemma \ref{cdiamq} implies that the degeneracy of $\hat{h}$ is exactly of the form as described in \cite[equation (3.8), see also Proposition 4]{Jacob}, \emph{i.e.} the metric can be written locally as a smooth Hermitian matrix multiplied by a matrix-valued holomorphic function which can have a zero.  Now after pulling back by $\pi$, we may further remove these factors that vanish on the exceptional divisors, to get a well-defined metric $\tilde{h}$, say, on $\pi^* E \cong \tilde{E}_{\le - w_{\alpha}} \oplus \pi^*E / \tilde{E}_{\le - w_{\alpha}}$. The $L^2$-norm of the second fundamental form of the metric $\tilde{h}$ agrees with the one defined by $\hat{h}$ by the same reasoning as in the  proof of  Lemma \ref{fact1}. But $\hat{h}$ has no off-diagonal block (\emph{cf.}~Proposition \ref{propexpfsz} and Definition \ref{defrenbgops}). Thus, the $L^2$-norm of the second fundamental form of the metric $\tilde{h}$ is zero.
	
We thus get Inequality (\ref{wweakestzetaii}) by recalling that the decay (as $t \to + \infty$) of $|\mathrm{II}_{\alpha -1}' (\hat{h}_{\sigma_t}) |$ at each $x \in X^{\mathrm{reg}}$ is exponential, as we saw in Step 1. Hence, summarising the argument above, after a finite sequence of blowups $\pi : \tilde{X} \to {X}$, we may reduce to the case in Step 1. This establishes (\ref{weakestzetaii}) when all $\mathcal{E}_{\le - w_{i}}$'s are all saturated.

	\medskip
	Step 3: we now consider the general case when the $\mathcal{E}_{\le - w_{i}}$'s are (not necessarily saturated) subsheaves of $\mathcal{E}$. We first apply a sequence of blowups $\pi$ so that $\bigcup_{\alpha=\hat{1}}^{\hat{\nu}} \mathrm{Sing} (\mathcal{E}_{\le - w_{\alpha}}) \cup \mathrm{Sing} (\mathcal{E} / \mathcal{E}_{\le - w_{\alpha}})$, where $\hat{h}_{\sigma_t}$ may be degenerate as $t \to + \infty$, is contained in the union of normal crossing divisors (see \cite[Theorem 4.4]{Sibley}). As in Step 2, we may remove the factors of $\hat{h}$ that vanish on the exceptional divisors of $\pi$, without changing the $L^2$-norm of the second fundamental form at $t = + \infty$ (as in the proof of Lemma \ref{fact1}, which holds for an arbitrary blowup). Note that we do not claim that, by such a sequence of blowups, the subsheaves $\mathcal{E}_{\le - w_{\alpha}}$ can be pulled back to holomorphic subbundles; we only remove the factors that vanish on $\mathrm{Sing} (\mathcal{E}_{\le - w_{\alpha}}) \cup \mathrm{Sing} (\mathcal{E} / \mathcal{E}_{\le - w_{\alpha}})$ that makes $\hat{h}_{\sigma_t}$ degenerate as $t \to + \infty$. We then apply the argument in Step 2 to the saturation $\mathcal{E}'_{\le - w_{\alpha}}$ of $\mathcal{E}_{\le - w_{\alpha}}$, by composing $\pi$ with a further sequence of blowups, where we note that $\mathcal{E}_{\le - w_{\alpha}}$ and its saturation $\mathcal{E}'_{\le - w_{\alpha}}$ differ only on a Zariski closed subset and hence this does not affect the value of the integral (\ref{wweakestzetaii}), and that the second fundamental form in Step 2 is defined with respect to the saturation $\mathcal{E}'_{\le - w_{\alpha}}$. Thus, we can reduce to the case in Step 2, to establish (\ref{weakestzetaii}) in general.

	We finally note that the map defined by
	\begin{equation*}
		\mathfrak{sl} (H^0 (\mathcal{E} (k))^{\vee}) \ni \zeta \mapsto \sum_{\alpha=\hat{1}}^{\hat{\nu}} (w_{\alpha} - w_{\alpha -1}) \int_0^{+ \infty} dt \int_{X^{\mathrm{reg}}} |\mathrm{II}_{\alpha -1}' (\hat{h}_{\sigma_t}) |^2 \frac{\omega^{n-1}}{(n-1)!} \in \mathbb{R}
	\end{equation*}
	is continuous, where we recall that $|\mathrm{II}_{\alpha -1}' (\hat{h}_{\sigma_t}) |^2$ decays exponentially as $t \to + \infty$ at each $p \in X$, with decay rate at least (and in fact faster than) $e^{-|w_{\alpha} - w_{\alpha -1}|t}$ by Proposition \ref{propexpfsz}. (Note in particular that, when we consider the limit of $w_{\alpha -1}$ tending to $w_{\alpha}$, the integral over $t$ of $(w_{\alpha} - w_{\alpha -1}) e^{-|w_{\alpha} - w_{\alpha -1}|t} $ remains bounded.) By recalling $\Vert \zeta \Vert_{\mathrm{op}} \le 1$ (\emph{cf.}~Remark \ref{condopnm}), we obtain that the bound that does not depend on $\zeta$, as claimed.
\end{proof}

\subsection{Proof of Proposition \ref{sumprophat}} \label{uppersec}

We now prove Proposition \ref{sumprophat}, which establishes the required bound for the second term in Lemma \ref{lm1}. We start with the lower bound.

\begin{lem} \label{proplb}
Let $\{ \hat{h}_{\sigma_t} \}_{t \ge 0}$ be a renormalised Bergman 1-PS associated to $\zeta \in \mathfrak{sl}(H^0(\mathcal{E}(k))^{\vee})$. Then there exists a constant $c_4 (h_k ) >0$ which depends only on the reference metric $h_k$ $($and $k)$ such that
	\begin{equation*}
		\mathcal{M}^{\mathrm{Don}} (\hat{h}_{\sigma_t}, h_k) > -c_4 (h_k )
	\end{equation*}
	irrespectively of $t \ge 0$ and $\zeta \in \mathfrak{sl}(H^0(\mathcal{E}(k))^{\vee})$.
\end{lem}

\begin{rem}
	The sign convention concerning the metric duality isomorphism, as mentioned in Remark \ref{remmdsign}, still applies, so $\hat{h}_{\sigma_t}$ is meant to be a Hermitian metric on $\mathcal{E}^{\vee}$.
\end{rem}

\begin{proof}
We first consider the following general situation. Suppose that we have a geodesic segment $\gamma(s):= e^{sv}$ ($0 \le s \le 1$), with $v:= \log hh^{-1}_0$, which connects $h_0$ and $h$, where $h_0$ is a fixed reference metric. The convexity of the Donaldson functional (Proposition \ref{lemmdonconvH}) implies
\begin{equation}\label{rineq}
	\mathcal{M}^{\mathrm{Don}} (h , h_0) \ge  \left.\frac{d}{ds} \right|_{s=0} \mathcal{M}^{\mathrm{Don}} (\gamma (s) , h_0) ,
\end{equation}
where we recall
\begin{equation*}
	\left. \frac{d}{ds} \right|_{s=0} \mathcal{M}^{\mathrm{Don}} (\gamma (s) , h_0) = \int_X \tr \left( v \left( \Lambda_{\omega} F(h_0) - \frac{\mu(E)}{\mathrm{Vol}_L} \mathrm{Id}_E \right) \right) \frac{\omega^n}{n!}.
\end{equation*}
Now define
\begin{equation} \label{average}
	\bar{v} := \frac{1}{r \mathrm{Vol}_L} \int_X \tr (v) \frac{\omega^n}{n!} \mathrm{Id}_E ,\end{equation}
so that $v-\bar{v}$ has average 0. Since $\bar{v}$ is a constant multiple of the identity, we have
\begin{equation*}
	\left. \frac{d}{ds} \right|_{s=0} \mathcal{M}^{\mathrm{Don}} (\gamma (s) , h_0) = \int_X \tr \left( (v-\bar{v}) \left( \Lambda_{\omega} F(h_0) - \frac{\mu(E)}{\mathrm{Vol}_L} \mathrm{Id}_E \right) \right) \frac{\omega^n}{n!}.
\end{equation*}
Thus, by using Cauchy--Schwarz, we have
\begin{equation*}
	 \left. \frac{d}{ds} \right|_{s=0} \mathcal{M}^{\mathrm{Don}} (\gamma (s) , h_0)  \geq - \Vert v - \bar{v} \Vert_{L^2} \left\Vert \Lambda_{\omega} F(h_0) - \frac{\mu(E)}{\mathrm{Vol}_L} \mathrm{Id}_E \right\Vert_{L^2} 
\end{equation*}
where the norm $\Vert \cdot \Vert_{L^2}$ is defined as
\begin{equation*}
	\Vert v -\bar{v} \Vert^2_{L^2} := \int_X \tr \left( (v -\bar{v}) \cdot  (v -\bar{v})) \right) \frac{\omega^n}{n!}=\int_{X} \tr (v^2) - \tr (\bar{v}^2) \frac{\omega^n}{n!} \ge 0.
\end{equation*}
Defining a constant
\begin{equation*}
	\underline{C}(h_0) = \left\Vert \Lambda_{\omega} F(h_0) - \frac{\mu(E)}{\mathrm{Vol}_L} \mathrm{Id}_E \right\Vert_{L^2} \ge 0,
\end{equation*}
we get
\begin{equation}\label{CSineq}
	\mathcal{M}^{\mathrm{Don}} (h , h_0) \ge  - \underline{C}(h_0) \Vert v - \bar{v} \Vert_{L^2},
\end{equation}
by recalling \eqref{rineq}.

Let us consider $\{ \hat{h}_{\sigma_t} \}_{t \ge 0}$, a renormalised Bergman 1-PS associated to $\zeta \in \mathfrak{sl}(H^0(\mathcal{E}(k))^{\vee})$ emanating from $h_k$. We apply the above argument to $h = \hat{h}_{\sigma_t}$ and $h_0 = h_k$, in which we replace $X$ by $X^{\mathrm{reg}}$. We note that the convexity of the Donaldson functional (Proposition \ref{lemmdonconvH}) holds for the geodesic segment $\gamma_t (s) := e^{sv_t}$ connecting $h_k$ and $\hat{h}_{\sigma_t}$, with $v_t := \log \hat{h}_{\sigma_t} h^{-1}_k$ which is only well defined on $X^{\mathrm{reg}}$, since the only nontrivial part in proving the convexity is the integration by parts (see \cite[Proof of Proposition 8]{Don1985}) which we are allowed to do since $\hat{h}_{\sigma_t}$, $\hat{h}_{\sigma_t}^{-1}$, and $F(\hat{h}_{\sigma_t})$ are all bounded over $X^{\mathrm{reg}}$ for any fixed $t$ (we can also prove the convexity by using the regularisation, as we do later in the proof of Lemma \ref{lemC3}).


Thus (\ref{CSineq}) implies that we have
\begin{equation*}
	\mathcal{M}^{\mathrm{Don}} (\hat{h}_{\sigma_t},h_k) \ge   - \underline{C}(h_k) \Vert v_t - \bar{v}_t \Vert_{L^2} .
\end{equation*}
To be more precise, the curvature $F(\hat{h}_{\sigma_t})$ should be defined with respect to $\hat{h}_{\sigma_t}$ composed with the metric duality isomorphism (as in Remark \ref{remmdsign}), but in any case this does not affect the above argument.

We prove that $\Vert v_t - \bar{v}_t \Vert_{L^2}$ can be bounded by a constant depending only on $\zeta \in \mathfrak{sl}(H^0(\mathcal{E}(k))^{\vee})$ and $k$, irrespectively of $t$.

We choose an orthonormal frame with respect to which $\hat{h}_{\sigma_t} h^{-1}_k$ is diagonalised as
\begin{equation*}
	\hat{h}_{\sigma_t} h^{-1}_k = \mathrm{diag}(\lambda_1(x,t) , \dots , \lambda_r (x,t) )
\end{equation*}
at each point $x \in X^{\mathrm{reg}}$. By the definition of $\hat{h}_{\sigma_t}$, each $\lambda_{\alpha} (x,t)$ can be written as
\begin{equation*}
	\lambda_{\alpha} (x,t) = \lambda_{\alpha ,1} (x) + e^{-c_{\alpha} t} \lambda_{\alpha ,2}(x),
\end{equation*}
for some smooth functions $\lambda_{\alpha ,1}(x)$ and $\lambda_{\alpha ,2}(x)$, and some $c_{\alpha} >0$. Note that $\lambda_{1,1}(x) , \dots , \lambda_{r,1} (x)$ are the eigenvalues of the renormalised Quot-scheme limit $\hat{h}$ at $x \in X^{\mathrm{reg}}$. Thus $\lambda_{\alpha ,1} (x) >0$ for $x \in X^{\mathrm{reg}}$, but it may tend to zero as it approaches the singular locus $X \setminus X^{\mathrm{reg}}$ that is Zariski closed in $X$. Since $\hat{h}_{\sigma_t} = h_k$ at $t=0$, observe that
\begin{equation} \label{boundinit}
	\lambda_{\alpha ,1} (x) +  \lambda_{\alpha ,2}(x) =1
\end{equation}
at all $x \in X^{\mathrm{reg}}$. Moreover, since $\tr (\hat{h}_k) \le \tr (h_k)$ for all $t \ge 0$ (by Proposition \ref{propexpfsz} and Definition \ref{defrenbgops}), we have
\begin{equation} \label{ublambdaallt}
	\lambda_{\alpha ,1} (x) + e^{-c_{\alpha} t} \lambda_{\alpha ,2}(x) \le \tr (h_k)
\end{equation}
for all $t \ge 0$ and all $x \in X^{\mathrm{reg}}$.

We evaluate
\begin{equation*}
	\int_{X^{\mathrm{reg}}} \tr (v_t^2) \frac{\omega^n}{n!} = \sum_{\alpha=1}^r \int_{X^{\mathrm{reg}}} \left( \log (\lambda_{\alpha ,1} (x) + e^{-c_{\alpha} t} \lambda_{\alpha ,2}(x)) \right)^2 \frac{\omega^n}{n!}.
\end{equation*}
We fix $\alpha$ and consider the integral
\begin{equation*}
	\int_{X^{\mathrm{reg}}} \left( \log (\lambda_{\alpha ,1} (x) + e^{-c_{\alpha} t} \lambda_{\alpha ,2}(x)) \right)^2 \frac{\omega^n}{n!}.
\end{equation*}
The estimates (\ref{boundinit}) and (\ref{ublambdaallt}) imply that, if $\lambda_{\alpha ,2} (x) \le 0$, then $1 \le \lambda_{\alpha,1}(x) \le \tr (h_k)$, and hence
\begin{equation*}
	\left( \log (\lambda_{\alpha ,1} (x) + e^{-c_{\alpha} t} \lambda_{\alpha ,2}(x)) \right)^2 \le \left( \log \tr (h_k) \right)^2.
\end{equation*}
Thus, without loss of generality, we may restrict to a subset of $X^{\mathrm{reg}}$ on which we have $0< \lambda_{\alpha,2} (x) \le 1$. We may further assume that $\lambda_{\alpha,1}(x) \le 1/2$, since otherwise
\begin{equation*}
	\left( \log (\lambda_{\alpha ,1} (x) + e^{-c_{\alpha} t} \lambda_{\alpha ,2}(x)) \right)^2 \le \left( \log 2 \right)^2.
\end{equation*}

We thus evaluate the integral of $\left( \log (\lambda_{\alpha ,1} (x) + e^{-c_{\alpha} t} \lambda_{\alpha ,2}(x)) \right)^2$ over the subset of $X^{\mathrm{reg}}$ where $\lambda_{\alpha,1}(x) \le 1/2$ and $0< \lambda_{\alpha,2} (x) \le 1$. With this assumption we have
\begin{equation*}
	\left( \log (\lambda_{\alpha ,1} (x) + e^{-c_{\alpha} t} \lambda_{\alpha ,2}(x)) \right)^2 \le 2 \left( \log (\lambda_{\alpha ,1} (x)) \right)^2
\end{equation*}
for all $t \ge 0$.

We may assume that the singular locus is contained in a finite union of normal crossing divisors $D:=\cup_{j} D_j$, by \emph{e.g.}~pulling back to the resolution of $X \setminus X^{\mathrm{reg}}$ (see \cite[Theorem 4.4]{Sibley}). We first consider the case where we have a neighbourhood $U_x$ of $x$ such that $U_x$ intersects with $D$ only at the smooth locus of $D$. Then, writing $z$ for the local holomorphic coordinate which defines $D$ as $z=0$ in such a neighbourhood, we find that $\lambda_{\alpha ,1} (x)$ tends to zero along the singular locus at a polynomial order in $|z|$, by recalling Lemma \ref{cdiamq} and that the quotient map $\rho$ (defined in (\ref{fglgenquot})) is algebraic. Thus we get, for all $1 \leq \alpha \leq r$, 
\begin{equation*}
	\int_{U_x \setminus D} \left( \log (\lambda_{\alpha ,1} (x)) \right)^2 \frac{\omega^n}{n!} \le \text{const.} \int_0^c \mathrm{r} \left( \log \mathrm{r}^{m_\alpha} \right)^2 d\mathrm{r} < + \infty,
\end{equation*}
where $\mathrm{r}:= |z|$, $m_\alpha \in \mathbb{N}$, and $c$ is some constant. The case when $U_x$ contains a (normal crossing) singularity of $D$ can be treated by the same argument.

Thus, the conclusion is that there exists a constant $c_4 (\zeta , h_k) >0$ such that
\begin{equation} \label{rfsmbczetak}
	\int_{X^{\mathrm{reg}}} \tr (v_t^2) \frac{\omega^n}{n!} \le c_4 (\zeta , h_k )
\end{equation}
for all $t \ge 0$. Since 
\begin{equation*}
	\Vert v_t - \bar{v}_t \Vert^2_{L^2} =\int_{X^{\mathrm{reg}}} \tr (v_t^2) \frac{\omega^n}{n!} - \int_{X^{\mathrm{reg}}} \tr (\bar{v}^2_t) \frac{\omega^n}{n!} \ge 0
\end{equation*}
implies $\int_{X^{\mathrm{reg}}} \tr (v_t^2) \frac{\omega^n}{n!} \ge \int_{X^{\mathrm{reg}}} \tr (\bar{v}^2_t) \frac{\omega^n}{n!}$, we get
\begin{equation} \label{bdvtaverage}
	\Vert v_t - \bar{v}_t \Vert^2_{L^2} \le 2 \int_{X^{\mathrm{reg}}} \tr (v_t^2) \frac{\omega^n}{n!} \le 2 c_4 (\zeta , h_k ).
\end{equation}

Moreover, the above argument shows that the integrand $\tr (v_t^2) \frac{\omega^n}{n!}$ extends continuously to the whole of $X$. Thus, henceforth in the proof we may replace $X^{\mathrm{reg}}$ by $X$.

We now prove that the constant $c_4 (\zeta , h_k)$ can be bounded uniformly in $\zeta$. To simplify the notation, we write $v (\zeta t)$ for $v_t = \log \hat{h}_{\sigma_t} h^{-1}_k$, by noting that $\zeta$ always appears in the form $\zeta t$. Suppose by contradiction that there exists a sequence $\{ \zeta_j \}_{j \in \mathbb{N}} \subset \mathfrak{sl}(H^0 (\mathcal{E}(k))^{\vee})$, $|| \zeta_j ||_{\mathrm{op}} \le 1$ (\emph{cf.}~Remark \ref{condopnm}) such that
\begin{equation*}
	\sup_t \int_{X} \tr (v(\zeta_j t)^2 ) \frac{\omega^n}{n!} \to + \infty 
\end{equation*}
as $j \to + \infty$. We may assume that there exists a sequence $\{ t_j \}_{j \in \mathbb{N}} \subset \mathbb{R}_{\ge 0}$ such that
\begin{equation} \label{zetajblowup}
	 \int_{X} \tr (v(\zeta_j t_j)^2 ) \frac{\omega^n}{n!} \to + \infty 
\end{equation}
as $j \to + \infty$, and also assume that $\{ \zeta_j \}_j$ contains a subsequence which converges to $\zeta_{\infty} \in \mathfrak{sl}(H^0 (\mathcal{E}(k))^{\vee})$, by $|| \zeta_j ||_{\mathrm{op}} \le 1$ (\emph{cf.}~Remark \ref{condopnm}). By arguing as in the proof of (\ref{rfsmbczetak}), we find
\begin{equation} \label{zetainffin}
	\sup_t \int_{X} \tr (v(\zeta_{\infty} t)^2 ) \frac{\omega^n}{n!} \le c_4 (\zeta_{\infty} , h_k ) < + \infty. 
\end{equation}
If $\{t_j \}_j$ is bounded in $\mathbb{R}_{\ge 0}$, we see that (\ref{zetajblowup}) contradicts (\ref{zetainffin}) since
\begin{equation*} 
	\mathfrak{sl}(H^0(\mathcal{E}(k))^{\vee}) \ni \zeta \mapsto \int_{X} \tr (v(\zeta t)^2 ) \frac{\omega^n}{n!} \in \mathbb{R} 
\end{equation*}
is continuous for each fixed $t$. Thus, by taking a further subsequence in $j$, we may assume that $\{ t_j \}_j$ monotonically increases to $+ \infty$. Recalling Definition \ref{defrenbgops}, the Lebesgue integration theorem implies
\begin{equation} \label{limzetaj}
	\lim_{j \to \infty} \int_X \tr (v(\zeta_{j} t_j)^2 ) \frac{\omega^n}{n!} = \int_X \tr \left( (\log \hat{h} h_k^{-1})^2 \right) \frac{\omega^n}{n!} 
\end{equation}
where we wrote
\begin{equation*}
	\hat{h}:= \lim_{j \to + \infty} e^{-w_j t_j} h_{\sigma_{j, t_j}} e^{-w_j t_j},
\end{equation*}
$w_j$ being the eigenvalues of $\zeta_j$ (as in (\ref{defewt})) and $\sigma_{j , t_j} := \exp (\zeta _j t_j)$. However, the right hand side of (\ref{limzetaj}) is finite by the same argument as in the proof of (\ref{rfsmbczetak}) (the limit depends on the choice of the subsequence in $\{ \zeta_j \}_j$ and $\{ t_j \}_j$, but the point is that it is finite for any choice of subsequence). This is a contradiction, and hence there exists a constant $c_4 (h_k ) >0$ that depends only on $k$ and the reference metric $h_k$, and not on $\zeta$, such that
\begin{equation*}
	\sup_t \int_X \tr (v(\zeta t)^2 ) \frac{\omega^n}{n!} \le c_4 (\zeta , h_k ) \le c_4 (h_k ), 
\end{equation*}
as claimed.

Combining (\ref{CSineq}) and (\ref{bdvtaverage}), we establish all the claims stated in the proposition.
\end{proof}

We prove the following upper bound, which establishes Proposition \ref{sumprophat}.

\begin{lem}\label{lemC3}
Let $\{ \hat{h}_{\sigma_t} \}_{t \ge 0}$ be a renormalised Bergman 1-PS associated to $\zeta \in \mathfrak{sl}(H^0(\mathcal{E}(k))^{\vee})$. Then there exists $c_5 (\zeta, h_k)$ depending on $h_k$ and $\zeta$ $($and $k)$ such that
\begin{equation*}
		\mathcal{M}^{\mathrm{Don}}(\hat{h}_{\sigma_t} , h_k) < c_5 ( \zeta, h_k)
\end{equation*}	
for all $t \ge 0$.
\end{lem}

\begin{proof}
	We argue in three steps, as in the proof of Lemma \ref{lembdscfdmfm}: when each $\mathcal{E}_{< w_i}$ is locally free, saturated, and the general case.
	
	Step 1: when each $\mathcal{E}_{\le - w_i}$ is locally free, and hence $X^{\mathrm{reg}} = X$, $\hat{h}_{\sigma_t}$ converges exponentially to the product metric on the $C^{\infty}$-vector bundle $\bigoplus_{\alpha =\hat{1}}^{\hat{\nu}} E_{-w_{\alpha}}$ over $X$. The claimed estimate immediately follows from Lemmas \ref{pdxreghath} and \ref{lemqjalpqunif}.
	
	Step 2: when each $\mathcal{E}_{\le - w_i}$ is saturated, we again argue as in \cite[Sections 3 and 4]{Jacob}, as we did in Lemma \ref{lembdscfdmfm}. We pull back by a finite sequence of blowups, and remove the factors that vanish (or blow up) on the exceptional divisors of $\pi$ (as in the proof of Lemma \ref{lembdscfdmfm}) without altering the value of the Donaldson functional \cite[Proof of Propositions 1 and 5]{Jacob}. We thus reduce to the case in Step 1.
	
	Step 3: the general case, where $\mathcal{E}_{\le - w_i}$'s are not necessarily saturated, can be treated by a further sequence of blowups, as in the proof of Lemma \ref{lembdscfdmfm}. By removing the factors that vanish on the exceptional divisors of $\pi$, we reduce to the case in Step 2.
\end{proof}

\begin{rem} \label{remuniflb}
	Suppose that the reference metrics $\{ h_k \}_{k \in \mathbb{N}}$ converge (say in $C^p$, $p \ge 2$) to $h_{\mathrm{ref}} \in \mathcal{H}_{\infty}$ as $k \to + \infty$. It seems tempting to speculate that in fact we have a \textit{uniform} constant $c_{\mathrm{ref}} >0$ which depends only on $h_{\mathrm{ref}}$ such that
	\begin{equation*}
		\mathcal{M}^{\mathrm{Don}} (\hat{h}_{\sigma_t}, h_{\mathrm{ref}}) > - c_{\mathrm{ref}}
	\end{equation*}
	in place of Lemma \ref{proplb} and
	\begin{equation*}
		\sum_{\alpha=\hat{1}}^{\hat{\nu}} (w_{\alpha} - w_{\alpha -1}) \int_0^{+ \infty} dt \int_X |\mathrm{II}_{\alpha -1}' (\hat{h}_{\sigma_{t}}) |^2 \frac{\omega^{n-1}}{(n-1)!} > -c_{\mathrm{ref}},
	\end{equation*}
	in place of Lemma \ref{lembdscfdmfm}, for all $t \ge 0$ and $\zeta \in \mathfrak{sl}(H^0(\mathcal{E}(k))^{\vee})$. In the paper \cite{HK2}, we prove that the above uniform bounds lead to an alternative proof of the theorem of Donaldson \cite{Do-1983,Don1985,Don1987}.
\end{rem}

\section{The non-Archimedean Donaldson functional \label{RMstab}}

Recalling Definition \ref{defcochq}, we find that for any $\sigma \in \mathcal{X}_{\mathbb{Q}}$ there exists $m \in \mathbb{N}$ such that $\sigma^m$ is a 1-parameter subgroup in the category of algebraic groups, which in particular has integral weights. This allows us to clear up the denominators in the materials in Section \ref{Qlimit}, so that the filtration $0 \neq \mathcal{E}_{\le - w_{1}} \subset \cdots \subset \mathcal{E}_{\le - w_{\nu}} = \mathcal{E}$ in (\ref{pfiltevs}) is graded by integers; more precisely, writing $w_1, \dots , w_{\nu}$ for the weights of $\sigma$ and defining $m w_i =: \bar{w}_i \in \mathbb{Z}$, where $m \in \mathbb{N}$ is the least integer such that $\sigma^m$ is algebraic, we define the filtration
\begin{equation} \label{filtsbsint}
	0 \neq \mathcal{E}_{\le - \bar{w}_{1}} \subset \cdots \subset \mathcal{E}_{\le - \bar{w}_{\nu}} = \mathcal{E}
\end{equation}
in an obvious way.

Recall also that the 1-PS's $\mathcal{X}_{\mathbb{Q}}$ in Definition \ref{defcochq} correspond bijectively with Hermitian elements $\zeta \in \mathfrak{sl} (H^0(\mathcal{E}(k))^{\vee})$ with rational eigenvalues, which provides a more down-to-earth description of the above. In particular, the integer $m$ that appeared above is exactly the following.

\begin{definition} \label{defofjzetak}
	Let $w_1 , \dots , w_{\nu} \in \mathbb{Q}$ be the weights of $\zeta \in \mathfrak{sl} (H^0(\mathcal{E}(k))^{\vee})$ corresponding to an element in $\mathcal{X}_{\mathbb{Q}} (k)$. We define $j(\zeta , k) \in \mathbb{N}$ to be the smallest positive integer such that $$j(\zeta , k) w_i \in \mathbb{Z}$$ for all $i = 1 , \dots , \nu$.
\end{definition}

Given the integrally graded filtration (\ref{filtsbsint}), we also have the filtration $0 \neq \mathcal{E}'_{\le - \bar{w}_{1}} \subset \cdots \subset \mathcal{E}'_{\le - \bar{w}_{\nu}} = \mathcal{E}$ of $\mathcal{E}$ by saturated subsheaves as in Lemma \ref{lemcdquot}. With this convention understood, we prove the following result.
\begin{lem} \label{dgrkmna}
	Recalling $\mathcal{E}'_{-w_{i}} = \mathcal{E}'_{\le -w_{i}}/\mathcal{E}'_{\le - w_{i-1}}$ as in Lemma \ref{lemcdquot}, we have
	\begin{equation*}
		-\sum_{i =1}^{\nu}  w_{i} \left( \deg (\mathcal{E}'_{-w_{i}}) - \mu(\mathcal{E}) \mathrm{rk} (\mathcal{E}'_{-w_{i}}) \right)  = \frac{1}{j(\zeta , k)} \sum_{q \in \mathbb{Z}} \rk (\mathcal{E}'_{\le q}) (\mu (\mathcal{E}) -\mu(\mathcal{E}'_{\le q} )).
	\end{equation*}
\end{lem}

\begin{proof}
	Note
	\begin{equation*}
		\rk (\mathcal{E}'_{-w_{i}}) = \rk \left( \mathcal{E}'_{\le -w_{i}}/\mathcal{E}'_{\le - w_{i-1}} \right) =\rk (\mathcal{E}'_{\le - w_{i}}) - \rk (\mathcal{E}'_{\le - w_{i -1}})
	\end{equation*}
	and
	\begin{equation*}
		\deg (\mathcal{E}'_{-w_{i}}) = \deg \left( \mathcal{E}'_{\le -w_{i}}/\mathcal{E}'_{\le - w_{i -1}} \right) =\deg (\mathcal{E}'_{\le - w_{i}}) - \deg (\mathcal{E}'_{\le - w_{i -1}}).
	\end{equation*}
	Thus, defining $\bar{w}_i := j(\zeta , k )w_i \in \mathbb{Z}$, we have
	\begin{align*}
		&- \sum_{i =1}^{\nu}  w_{i} \left( \deg (\mathcal{E}'_{-w_{i}}) - \mu(\mathcal{E}) \mathrm{rk} (\mathcal{E}'_{-w_{i}}) \right)  \\
		&=\frac{-1}{j(\zeta , k)} \sum_{i =1}^{\nu}  \bar{w}_{i} \left( \deg (\mathcal{E}'_{\le - \bar{w}_{i }}) - \deg (\mathcal{E}'_{\le - \bar{w}_{i-1}}) - \mu(\mathcal{E}) (\rk (\mathcal{E}'_{\le - \bar{w}_{i }}) - \rk (\mathcal{E}'_{\le - \bar{w}_{i -1}} ) ) \right)  \\
		&=\frac{1}{j(\zeta , k)} \sum_{q \in \mathbb{Z}}  q \left( \deg (\mathcal{E}'_{\le q }) - \deg (\mathcal{E}'_{\le  q-1}) - \mu(\mathcal{E}) (\rk (\mathcal{E}'_{\le q } ) - \rk (\mathcal{E}'_{\le q-1} ) ) \right),
	\end{align*}
	by recalling the ordering $w_1 > \cdots > w_{\nu}$ as in (\ref{ordlambda}).
	
	Since $ q  \deg (\mathcal{E}'_{\le q-1 }) = (q-1) \deg (\mathcal{E}'_{\le q-1 }) + \deg (\mathcal{E}'_{\le q-1 })$, we get
	\begin{equation*}
		\sum_{q \in \mathbb{Z}}  q \left( \deg (\mathcal{E}'_{\le q }) - \deg (\mathcal{E}'_{\le q-1}) \right) = -\sum_{q \in \mathbb{Z}} \deg (\mathcal{E}'_{\le q-1 }) = -\sum_{q \in \mathbb{Z}} \deg (\mathcal{E}'_{\le q }),
	\end{equation*}
	and similarly for the rank
	\begin{equation*}
		\sum_{q \in \mathbb{Z}}  q \left( \rk (\mathcal{E}'_{\le q }) - \rk (\mathcal{E}'_{\le q-1}) \right) ) = -\sum_{q \in \mathbb{Z}} \rk (\mathcal{E}'_{\le q }).
	\end{equation*}
	Although the above infinite sums are divergent, we substitute them in
	\begin{align*}
		\sum_{q \in \mathbb{Z}}  q \left( \deg (\mathcal{E}'_{\le q }) - \deg (\mathcal{E}'_{\le q-1}) - \mu(\mathcal{E}) (\rk (\mathcal{E}'_{\le q } ) - \rk (\mathcal{E}'_{\le q-1} ) ) \right)
		&= - \sum_{q \in \mathbb{Z}} \left( \deg (\mathcal{E}'_{\le q }) - \mu (\mathcal{E}) \rk (\mathcal{E}'_{\le q}) \right)\\
		&=\sum_{q \in \mathbb{Z}} \rk (\mathcal{E}'_{\le q}) (\mu (\mathcal{E}) -\mu(\mathcal{E}'_{\le q} )),
	\end{align*}
	to get a finite sum; since the filtration $0 \neq \mathcal{E}'_{\le- w_{1}} \subset \cdots \subset \mathcal{E}'_{\le- w_{\nu}} = \mathcal{E}$ is finite, there are only finitely many nonzero terms in the above sum.
\end{proof}

We now define the non-Archimedean Donaldson functional, following the terminology of \cite[Definition~3.4]{BHJ2}.
\begin{definition} \label{defmdonna}
	Suppose that we take $\sigma \in \mathcal{X}_{\mathbb{Q}}$ which gives rise to the filtration $\{ \mathcal{E}_{\le q} \}_{q \in \mathbb{Z}}$ of $\mathcal{E}$ as in (\ref{filtevs}), which we may assume is graded by integers (as in (\ref{filtsbsint})). Taking the saturation $\mathcal{E}'_{\le q}$ of $\mathcal{E}_{\le q}$ in $\mathcal{E}$, the \textbf{non-Archimedean Donaldson functional} $\mathcal{M}^{\mathrm{NA}}  : \mathcal{X}_{\mathbb{Q}} \to \mathbb{Q}$ is defined as
	\begin{equation*}
		\mathcal{M}^{\mathrm{NA}} (\sigma ) := \frac{2}{j (\sigma )} \sum_{q \in \mathbb{Z}} \mathrm{rk} (\mathcal{E}'_{\le q}) \left(  \mu(\mathcal{E}) - \mu (\mathcal{E}'_{\le q}) \right),
	\end{equation*}
	where $j (\sigma ) \in \mathbb{N}$ is the smallest positive integer such that $\sigma^{j(\sigma)}$ is a 1-PS in the category of algebraic groups.
	
	While the above formulation is conceptually straightforward, a formula in terms of the generator $\zeta \in \mathfrak{sl} (H^0(\mathcal{E}(k))^{\vee})$ of $\sigma \in \mathcal{X}_{\mathbb{Q}}$ is also useful in a sequel \cite{HK2} to this paper. By using the notation in Definition \ref{defofjzetak}, we write
	\begin{equation*}
		\mathcal{M}^{\mathrm{NA}} (\zeta , k) = \frac{2}{j(\zeta , k)} \sum_{q \in \mathbb{Z}} \mathrm{rk} (\mathcal{E}'_{\le q}) \left(  \mu(\mathcal{E}) - \mu (\mathcal{E}'_{\le q}) \right),
	\end{equation*}
	for the generator $\zeta$ of each element in $\mathcal{X}_{\mathbb{Q}}(k)$.
\end{definition}

The name ``non-Archimedean'' originally comes from the analogy with a theory on Berkovich spaces, as explained in \cite[Section 6.8]{BHJ1}. Recall also that a filtration on a vector space $V$ corresponds one-to-one with a non-Archimedean norm on $V$ \cite[Section 1.1]{BHJ1}; see also \cite[Section 1]{Boucksom-icm}.

\section{Two-step filtrations associated to saturated subsheaves}\label{2stp}

We now consider the special case when the filtration is given in two steps. A particularly important case is the one associated to saturated subsheaves of $\mathcal{E}$, which defines a natural two-step filtration for $H^0 (\mathcal{E}(k)) =: V$.

\begin{definition}\label{twostepfiltdef}
	
	Suppose that we are given a saturated subsheaf $\mathcal{F}$ of $\mathcal{E}$, and choose $k \in \mathbb{N}$ sufficiently large such that $\mathcal{F}(k)$ and $\mathcal{E}(k)$ are both globally generated. Then it defines $0 \neq H^0(\mathcal{F}(k)) \subset H^0(\mathcal{E}(k)) = V$ a two-step filtration (whose weights will be specified later in (\ref{defwtfone}) and (\ref{defwtftwo})). We write $\zeta_{\mathcal{F}} \in \mathfrak{sl}(H^0(\mathcal{E}(k))^{\vee})$ for the element defining this filtration. 
\end{definition}

Note that the filtration of $\mathcal{E}$ generated by $\zeta_{\mathcal{F}}$ as in (\ref{filtevs}) is exactly $0 \neq \mathcal{F} \subset \mathcal{E}$.

We now specify the weights of $\zeta_{\mathcal{F}}$ defined in Definition \ref{twostepfiltdef}. Recall that the weights $w_1 , w_2$ of $\zeta_{\mathcal{F}}$ are determined by $\tr (\zeta_{\mathcal{F}}) = 0$; more precisely (\emph{cf.}~\cite[Section 4]{GF-K-R}), we get
	\begin{align}
	w_1 &= \frac{h^0 ( \mathcal{F}(k))}{h^0 (\mathcal{E}(k))}\ \ \text{and} \label{defwtfone} \\
	w_2 &=  -\frac{h^0 (\mathcal{E}(k)) - h^0 ( \mathcal{F}(k))}{h^0(\mathcal{E} (k))}, \label{defwtftwo}
	\end{align}
	where the factor $1/ h^0(\mathcal{E}(k))$ is inserted so as to make the operator norm of $\zeta$ to be less than 1 (\emph{cf.} Remark~\ref{condopnm}). Note that these weights are rational numbers, and hence that the 1-PS generated by $\zeta_{\mathcal{F}}$ defines an element in $\mathcal{X}_{\mathbb{Q}}$.

We can restate some of the results in Section \ref{unifestimsect} for the two-step filtrations as follows.

\begin{prop} \label{prop2}
Suppose that $\zeta_{\mathcal{F}} \in\mathfrak{sl} (H^0 (\mathcal{E}(k))^{\vee})$ defines a two-step filtration associated to a saturated subsheaf $\mathcal{F}\subset \mathcal{E}$ for large enough $k$ as in Definition \ref{twostepfiltdef}. For the rational Bergman 1-PS $\{ h_{\sigma_t} \}_{t \ge 0}$ defined by $\zeta_{\mathcal{F}}$ emanating from $h_k$, there exist a constant $c_6 (h_k ,  k) >0$ such that
\begin{equation*}
	\mathcal{M}^{\mathrm{Don}} (h_{\sigma_t} , h_{k}) \ge \rk (\mathcal{F}) (\mu (\mathcal{E}) - \mu( \mathcal{F})) \cdot 2 t - c_6 (h_k , k)
\end{equation*}
for all $t \ge 0$, and a constant $c_7 (h_k , \zeta_{\mathcal{F}} , k) >0$ such that
\begin{equation*}
	\mathcal{M}^{\mathrm{Don}} (h_{\sigma_t} , h_{k}) \le \rk (\mathcal{F}) (\mu (\mathcal{E}) - \mu( \mathcal{F})) \cdot 2 t + c_7 (h_k , \zeta_{\mathcal{F}} , k)
\end{equation*}
holds for all sufficiently large $t > 0$.
\end{prop}

\begin{proof}
	We simply write down Propositions \ref{sumprop} for the specific filtration defined by $\zeta_{\mathcal{F}}$. Direct computation immediately yields the claimed results.
\end{proof}

Finally, we end this section by proving the following rather technical result which we need later. Recall that the distance $\mathrm{dist} (h_1, h_2)$ between two Hermitian metrics $h_1 , h_2 \in \mathcal{H}$ can be defined by using the geodesic between them, as in Section \ref{fsmetvb}.

\begin{lem}\label{lemdistinfty}
	Given a Bergman 1-PS $\{ h_{\sigma_t} \}_{t \ge 0}$ emanating from $h_k$, generated by $\zeta_{\mathcal{F}} \in \mathfrak{sl}(H^0(E(k)))$ associated to a saturated subsheaf $\mathcal{F} \subset \mathcal{E}$, we have $\mathrm{dist}(h_{\sigma_t} , h_{k}) \to + \infty$ as $t \to + \infty$.
\end{lem}

\begin{proof}
	We can explicitly write down the geodesic $\{ \gamma_s ( t) \}_{0 \le s \le 1}$ connecting $h_{k}$ and $h_{\sigma_t}$ as $e^{sv(\sigma_t)}h_{k}$ where $v ( \sigma_t ):= \log (h_{\sigma_t} h^{-1}_{k})$. We compute its Riemannian length as
	\begin{equation*}
		|\partial_s \gamma_s ( t)| = \int_X \tr (v ( \sigma_t ) \cdot v ( \sigma_t ))^{1/2} \frac{\omega^n}{n!} ,
	\end{equation*}
	where we recall that $\gamma_s (t)$ being a geodesic is equivalent to $\partial_s \gamma_s ( t)$ being independent of $s$. We show that this length segment tends to $+ \infty$ as $t \to + \infty$.
	
	We take an $h_{k}$-orthonormal frame. We now recall Proposition \ref{propexpfsz}, which states
	\begin{equation*}
	v(\sigma_t) = \log e^{w t} \left( \hat{h} + o(t) \right) e^{w t},	
	\end{equation*}
	where $w := \mathrm{diag} (w_1 , w_2)$ (as a block-diagonal matrix), where $w_1 , w_2$ are as defined in (\ref{defwtfone}) and (\ref{defwtftwo}). We also wrote $\hat{h} := \bigoplus_{\alpha =1}^{2} Q^*_{\alpha \to \alpha} Q_{\alpha \to \alpha} >0$ for the renormalised Quot-scheme limit.
	Observe that we have
	\begin{equation*}
		\log e^{w t} \left( \hat{h}/2 \right) e^{w t} < v( \sigma_t) < \log e^{w t} \left( 2 \hat{h}  \right) e^{w t} ,
	\end{equation*}
	for all large enough $t$, and hence
	\begin{equation*}
		2 w t + \log \left( \hat{h} /2 \right) < v( \sigma_t) < 2 w t + \log \left( 2 \hat{h} \right) ,
	\end{equation*}
	noting that $e^{w t}$ commutes with the renormalised Quot-scheme limit $\hat{h}$. Thus each eigenvalue of $v(\sigma_t)$ can be estimated as $2 w_{\alpha} t (1+ o(t))$, $\alpha = 1, 2$.
	Hence we get
	\begin{equation*}
		\tr ( v(\sigma_t) v(\sigma_t) ) > 4 t^2 \sum_{\alpha =1}^{2} w_{\alpha}^2 (1+o(t)) > 2 t^2 \sum_{\alpha =1}^{2} w_{\alpha}^2,
	\end{equation*}
	which tends to $+ \infty$ as $t \to + \infty$ at each point $x \in X^{\mathrm{reg}}$.
	
	Thus we can compute the geodesic length between $h_{k}$ and $h_{\sigma_t}$, as we did in the proof of Lemma \ref{proplb}, as
	\begin{equation*}
		\mathrm{dist}(h_{\sigma_t} , h_{k}) = \int_0^1 ds \int_X \tr (v ( \sigma_t ) \cdot v ( \sigma_t ))^{1/2}  \frac{\omega^n}{n!} > \text{const}\cdot t^2 \sum_{\alpha =1}^{2} w_{\alpha}^2,
	\end{equation*}
	which tends to $+ \infty$ as $t \to + \infty$.
\end{proof}

\section{Summary of the main results}\label{mainsect}

We recall the following terminology (\emph{cf.}~\cite[Definition 5.1]{BHJ2}).

\begin{definition}\label{lognorm}
 Let $\mathcal{N}$ be a function $$\mathcal{N} : \SL (N , \mathbb{C})  \to \mathbb{R}$$ that is continuous with respect to the topology on $\SL (N , \mathbb{C})$ defined by a submultiplicative matrix norm $|| \cdot ||$ (\emph{e.g.}~the operator norm).
 \begin{enumerate}
 	\item $\mathcal{N}$ is said to be \textbf{coercive} if there exist two constants $a,b \in \mathbb{R}$ with $a>0$ such that
  $$\mathcal{N}( g ) \geq a \log \Vert g \Vert - b$$ holds for any $g \in \SL (N , \mathbb{C})$;
  \item $\mathcal{N}$ is said to have \textbf{log norm singularities} if it can be written as $$\mathcal{N}(g ) = a \log \Vert g \Vert +O(1)$$ for some constant $a \in \mathbb{R}$, where $O(1)$ stands for the terms that remain bounded over $\SL (N , \mathbb{C})$.
 \end{enumerate}
 
We shall also say that $\mathcal{N}$ is coercive (resp.~has log norm singularities) \textit{along any rational Bergman 1-PS} if in the above notation it satisfies $\mathcal{N}(\sigma_t ) \geq a \log \Vert \sigma_t \Vert - b$ (resp.~$\mathcal{N}(\sigma_t ) = a \log \Vert \sigma_t \Vert +O(1)$) for any $ \sigma \in \mathcal{X}_{\mathbb{Q}}$.
\end{definition}

Observe that by means of Bergman 1-PS's, we can naturally define $\mathcal{M}^{\mathrm{Don}}$ as a map from $\SL(H^0(\mathcal{E}(k))^{\vee})$ to $\mathbb{R}$ as
\begin{displaymath}
\mathcal{M}^{\mathrm{Don}} (-,h_k):
\left\{
  \begin{array}{rcl}
    \SL(H^0(\mathcal{E}(k))^{\vee}) & \longrightarrow &\mathbb{R} \\
    \sigma & \longmapsto & \mathcal{M}^{\mathrm{Don}} (h_{\sigma},h_k) \\
  \end{array}
\right.
\end{displaymath}
with the fixed reference metric $h_k$, while we are mostly interested in the case when the domain of $\mathcal{M}^{\mathrm{Don}} (-,h_k)$ is restricted to a Bergman 1-PS $\{ \sigma_t \}_{t \ge 0}$.

Important functionals concerning the constant scalar curvature K\"ahler metrics have log norm singularities, as proved by Paul \cite{Paul}; see also \cite[Section 5]{BHJ2}. Note also that the property of having log norm singularities is independent of the choice of a matrix norm, up to changing the constants $c_0$ and $c_1$.

A particularly important role is played by coercive functions; they are in particular bounded from below. If $\mathcal{N}$ is coercive then it is proper, \emph{i.e.}~the preimage of any compact set is compact, which is equivalent in this setting to saying that $\mathcal{N}(g )$ tends to $+\infty$ as $\Vert g \Vert$ tends to $+\infty$.

We now summarise what we have established so far. Suppose that we take $k_0 \in \mathbb{N}$ to be large enough for $\mathcal{E} (k)$ to be globally generated, and also that we have a sequence $\{ h_k \}_{k \ge k_0} \subset \mathcal{H}_{\infty}$, $h_k \in \mathcal{H}_k$, converging to the reference metric $h_{\mathrm{ref}} \in \mathcal{H}_{\infty}$ in $C^p$ for $p \ge 2$, afforded by (\ref{defbergsp2}). Write $\{ h_{\sigma_t} \}_{t \ge 0}$ for the Bergman 1-PS emanating from $h_k$ associated to the 1-PS $\sigma \in \mathcal{X}_{\mathbb{Q}}$. In what follows we assume $k \ge k_0$.

\begin{thm} \label{thmlnsdf}
There exists a constant $c_k (h_{\mathrm{ref}}) >0$ that depends only on $h_{\mathrm{ref}}$ and $k \in \mathbb{N}$ such that
\begin{equation*}
		\mathcal{M}^{\mathrm{Don}}(h_{\sigma_t} ,  h_{\mathrm{ref}}) \ge  \mathcal{M}^{\mathrm{NA}} ( \sigma ) t - c_{k} (h_{\mathrm{ref}})
\end{equation*}
holds for all $t \ge 0$ and all $ \sigma \in \mathcal{X}_{\mathbb{Q}}(k)$, where
\begin{equation*}
		\mathcal{M}^{\mathrm{NA}} ( \sigma ) = \frac{2}{j(\sigma )} \sum_{q \in \mathbb{Z}} \mathrm{rk} (\mathcal{E}'_{\le q}) \left(  \mu(\mathcal{E}) - \mu (\mathcal{E}'_{\le q}) \right)
\end{equation*}
is as defined in Definition \ref{defmdonna}.
\end{thm}

\begin{proof}
This follows from what we have established so far. By the cocycle property (\ref{cocyclemdon}) of the Donaldson functional, we have
\begin{equation*}
	\mathcal{M}^{\mathrm{Don}}(h_{\sigma_t} ,  h_{\mathrm{ref}}) = \mathcal{M}^{\mathrm{Don}}(h_{\sigma_t} ,  h_k) + \mathcal{M}^{\mathrm{Don}}(h_k ,  h_{\mathrm{ref}}).
\end{equation*}
By the convergence $h_k \to h_{\mathrm{ref}}$ in $C^p$, we may assume $|\mathcal{M}^{\mathrm{Don}}(h_k ,  h_{\mathrm{ref}})| < 1$ for all $k$. The result now follows from Proposition \ref{sumprop} and Lemma \ref{dgrkmna}, by replacing $c_1 (h_k) + \mathcal{M}^{\mathrm{Don}}(h_k ,  h_{\mathrm{ref}})$ by some constant $c_k (h_{\mathrm{ref}}) >0$.
\end{proof}

A similar argument also yields the following result.

\begin{thm}\label{uppbound}
There exists a constant $c_k (\zeta, h_{\mathrm{ref}}) >0$ depending on $h_{\mathrm{ref}}$, $k \in \mathbb{N}$, and $\zeta \in \mathfrak{sl}(H^0(\mathcal{E}(k))^{\vee})$, such that
	\begin{equation*}
		\mathcal{M}^{\mathrm{Don}}(h_{\sigma_t} ,  h_{\mathrm{ref}}) \le \mathcal{M}^{\mathrm{NA}} ( \sigma ) t + c_k (\zeta, h_{\mathrm{ref}})
\end{equation*}
holds for any $\sigma \in \mathcal{X}_{\mathbb{Q}}(k)$ and all large enough $t>0$, with $\mathcal{M}^{\mathrm{NA}} ( \sigma )$ as defined in Definition \ref{defmdonna}.
\end{thm}

Note the following important observation for $\mathcal{M}^{\mathrm{NA}}$ (where a filtration is called trivial if it is $0 \subsetneq \mathcal{E}$), which follows immediately by considering the two-step filtrations (and 1-PS's associated to them) as in Definition \ref{twostepfiltdef}.


\begin{prop} \label{cormnastab}
The non-Archimedean Donaldson functional $\mathcal{M}^{\mathrm{NA}} (\sigma )$ is positive $($resp.~nonnegative$\,)$ for all $\sigma \in \mathcal{X}_{\mathbb{Q}}(k)$ whose associated filtration by saturated subsheaves (\ref{filtevs}) is nontrivial and for all $k \in \mathbb{N}$ such that $\mathcal{E} (k)$ is globally generated, if and only if $\mathcal{E}$ is slope stable $($resp.~semistable$\,)$.
\end{prop}

We finally summarise particularly important consequences of our main results as follows.

\begin{cor}\label{prop1}
With notation as in Theorems \ref{thmlnsdf} and \ref{uppbound} and Definition \ref{lognorm}, we have the following results.
\begin{enumerate}

\item  If $\mathcal{E}$ is slope stable $($resp.~semistable$\,)$, then the Donaldson functional is coercive $($resp.~bounded from below$\,)$ along any rational Bergman 1-PS in $\mathcal{H}_k \subset \mathcal{H}_{\infty}$, for all $k \in \mathbb{N}$ such that $\mathcal{E} (k)$ is globally generated.
\item The Donaldson functional has log norm singularities along any rational Bergman 1-PS in $\mathcal{H}_k \subset \mathcal{H}_{\infty}$ and for all $k \in \mathbb{N}$ such that $\mathcal{E} (k)$ is globally generated, i.e.~for any $ \sigma  \in \mathcal{X}_{\mathbb{Q}} (k)$ we have
	\begin{equation*}
		\mathcal{M}^{\mathrm{Don}} (h_{\sigma_t}, h_{\mathrm{ref}}) = \mathcal{M}^{\mathrm{NA}} ( \sigma ) t + O(1),
	\end{equation*}
	 where $O(1)$ stands for the terms that remain bounded as $t \to + \infty$. In particular, the Donaldson functional is not bounded from below if $\mathcal{E}$ is slope unstable.
	
	\item We have $$\lim_{t \to + \infty} \frac{\mathcal{M}^{\mathrm{Don}} (h_{\sigma_t}, h_{\mathrm{ref}})}{t} = \mathcal{M}^{\mathrm{NA}} ( \sigma )$$ for any $\sigma \in \mathcal{X}_{\mathbb{Q}}(k)$ and for all $k \in \mathbb{N}$ such that $\mathcal{E} (k)$ is globally generated.
\end{enumerate}
\end{cor}

\section{From Hermitian--Einstein metrics to slope stability}\label{cantostab}

In this section we provide a more geometric proof of a theorem of Kobayashi \cite{Kobookjp} and L\"ubke \cite{Luebke}, as a simple application of Theorem \ref{uppbound} and the geodesic convexity of the Donaldson functional. We start by considering an irreducible holomorphic vector bundle $\mathcal{E}$ endowed with a Hermitian--Einstein metric.

Suppose that there exists a critical point for the Donaldson functional, which is necessarily the global minimum, say $h_{\mathrm{min}}$, over $\mathcal{H}_{\infty}$ by strict geodesic convexity (\emph{cf.}~Proposition \ref{lemmdonconvH}). Now we take $h_{\mathrm{min}}$ as the reference metric $h_{\mathrm{ref}}$, and consider the Donaldson functional $$\mathcal{M}^{\mathrm{Don}} (h):= \mathcal{M}^{\mathrm{Don}}(h, h_{\mathrm{min}})$$ for $h \in \mathcal{H}_k \subset \mathcal{H}_{\infty}$.

Using the results of Section \ref{2stp}, consider $\zeta_{\mathcal{F}} \in \mathfrak{sl} (H^0 (\mathcal{E}(k))^{\vee})$ associated to a saturated subsheaf $\mathcal{F}$ of $\mathcal{E}$, where we take $k \in \mathbb{N}$ to be large enough so that $\mathcal{E} (k)$ and $\mathcal{F} (k)$ are both globally generated, and write $\sigma_{\mathcal{F}} \in \mathcal{X}_{\mathbb{Q}} (k)$ for the 1-PS generated by $\zeta_{\mathcal{F}}$. If $\mathcal{M}^{\mathrm{NA}} (\sigma_{\mathcal{F}}) < 0$, then from Theorem \ref{uppbound}, we have $$\mathcal{M}^{\mathrm{Don}}(h_{\sigma_t} ) \to - \infty$$ as $t \to + \infty$, for the path $\{ h_{\sigma_t} \}_{t \ge 0}$ in $\mathcal{H}_k$. In particular, for large enough $t$, we have $\mathcal{M}^{\mathrm{Don}}(h_{\sigma_t}) < \mathcal{M}^{\mathrm{Don}}(h_{\mathrm{min}})$, contradicting the definition of the minimum $h_{\mathrm{min}}$. Thus we get $\mathcal{M}^{\mathrm{NA}} (\sigma_{\mathcal{F}}) \ge 0$ for any saturated subsheaf $\mathcal{F}$.

We now prove $\mathcal{M}^{\mathrm{NA}} (\sigma_{\mathcal{F}}) > 0$. Assume for contradiction $\mathcal{M}^{\mathrm{NA}} (\sigma_{\mathcal{F}})=0$. Using the geodesic completeness of $\mathcal{H}_{\infty}$, we consider the (nontrivial) geodesic between  $h_{\sigma_t}$ and $h_{\mathrm{min}}$ (it may be not contained in $\mathcal{H}_k$, but this does not matter), and the geodesic distance $\mathrm{dist}(h_{\sigma_t} , h_{\mathrm{min}})$ between them. Then, from  Lemma \ref{lemdistinfty}, we have $\mathrm{dist}(h_{\sigma_t} , h_{\mathrm{min}}) \to + \infty$ as $t \to +\infty$; recall that the sequence $\{ h_k \}_{k \in \mathbb{N}} \subset \mathcal{H}_{\infty}$, $h_k \in \mathcal{H}_k$, converges to the reference metric $h_{\mathrm{min}}$ in $C^p$ for $p \ge 2$, as in Section \ref{mainsect}. Thus, using the strict geodesic convexity of $\mathcal{M}^{\mathrm{Don}}$ we obtain 
$$\mathcal{M}^{\mathrm{Don}}(h_{\sigma_t}) \to + \infty,$$ as $t \to + \infty$. 
But Theorem \ref{uppbound} implies that this contradicts the assumption $\mathcal{M}^{\mathrm{NA}}(\sigma_{\mathcal{F}})=0$. Consequently, we get $\mathcal{M}^{\mathrm{NA}} (\sigma_{\mathcal{F}}) > 0$, which implies (by recalling Proposition \ref{prop2} or \ref{cormnastab})
$$\mu(\mathcal{E})>\mu(\mathcal{F}),$$
and hence the slope stability of $\mathcal{E}$ by Lemma \ref{refl}.

\medskip

Finally, the case when $\mathcal{E}$ is reducible can be treated similarly. By applying the previous argument to each irreducible component of $\mathcal{E}$, we find that $\mathcal{E}$ is a direct sum of slope stable bundles. The Hermitian--Einstein equation implies that the slopes of these components are equal, and hence $\mathcal{E}$ is slope polystable.


\ifx\undefined\bysame
\newcommand{\bysame}{\leavevmode\hbox to3em{\hrulefill}\,}
\fi

\end{document}